\documentclass[12pt,twoside,reqno]{amsart}
\usepackage{amsmath,amssymb,amsthm}
\usepackage{fancyhdr}
\usepackage{epsfig,color,colordvi} 
\usepackage[numbers,sort&compress,square,comma]{natbib}

\theoremstyle{definition}

\theoremstyle{remark}

\theoremstyle{plain}
\newtheorem{theorem}{\sc Theorem}[section]
\newtheorem{lemma}{\sc Lemma}[section]
\newtheorem{proposition}[lemma]{\sc Proposition}
\newcommand*{\be}{\begin{equation}}
\newcommand*{\ee}{\end{equation}}
\newcommand*{\ba}{\begin{aligned}}
\newcommand*{\ea}{\end{aligned}}
\newcommand*{\nn}{\nonumber}

\def\gabar{\bar{\gamma}}
\def\altil{\tilde{\alpha}}
\def\rhotil{\tilde{\rho}}

\def\Xbar{{\bar{X}}}
\def\Ybar{{\bar{Y}}}
\def\betabar{{\bar{\beta}}}

\def\rhobar{{\bar{\rho}}}
\def\zetabar{{\bar{\zeta}}}
\def\ombar{{\bar{\omega}}}
\def\Gammabar{{\bar{\Gamma}}}
\def\qbar{{\bar{q}}}
\def\Mbar{{\bar{M}}}
\def\Zbar{{\bar{Z}}}
\def\Lbar{{\bar{L}}}
\def\wbar{{\bar{w}}}
\def\Pv{\mathbf{P}}
\def\Ev{\mathbf{E}}
\def\Tbar{{\bar{T}}}

\def\vtil{{\tilde{v}}}

\def\kD{\mathfrak{D}}
\def\Qtil{{\widetilde Q}}

\def\zbar{{\bar{z}}}
\def\gatil{{\tilde{\gamma}}}

\def\elltil{{\tilde{\ell}}}

\def\Mtil{{\widetilde{M}}}

\def\e{\varepsilon}
\def\PP{\mathbf{P}}   % Look up \P below. It unifies all this.
\def\Fc{\mathcal{F}}  % There's also vP (bold face, used by Marton)
\def\Ic{\mathcal{I}}  % There's also vP (bold face, used by Marton)
\def\Hc{\mathcal{H}}
\def\cG{\mathcal{G}}
\def\Gc{\mathcal{G}}
\def\cP{\mathcal{P}}

\def\cona{b_0}
\def\conb{b}

\def\wiener{W}  %%%distr of limit process

\def\btil{{\tilde{b}}}

\def\mtil{{\tilde{m}}}

\def\etatil{{\tilde{\eta}}}

\def\Etil{{\widetilde{E}}}

\def\Xtil{{\widetilde{X}}}

\newcommand*{\tfl}[1]{\lfloor{#1}\rfloor}

\def\Ac{\mathcal{A}}

\def\cS{\mathcal{S}}

\def\uhat{{\hat u}}

\def\mutil{{\tilde\mu}}

\def\nutil{{\tilde\nu}}

\def\kS{{\mathfrak S}}

\def\bbR{\mathbb R}
\def\bbN{\mathbb N}
\def\bbZ{\mathbb Z}
\def\bbQ{\mathbb Q}
\def\bbP{\mathbb P}
\def\bbE{\mathbb E}

\def\fQ{{{\rm Q}\kern-.65em {}^{{}_/ }\,}}
\def\fQQ{ {{\rm Q}\kern-.57em \scriptscriptstyle{}^{]\kern.055em[}\,}}

\def\one{{{\rm 1\mkern-1.5mu}\!{\rm I}}}
\def\w{\omega}

\def\vard{d_{\scriptscriptstyle\rm Var}}
\def\linf{\varliminf}
\def\lsup{\varlimsup}

\def\ord{\kern0.1em o\kern-0.02em{}_{\ds\breve{}}\kern0.1em}
\def\Ord{\kern0.1em O\kern-0.02em{\ds\breve{}}\kern0.1em}
\def\ds{\displaystyle}

\def\fmonth{\ifcase\month\or Jan\or Feb\or Mar\or Apr
\or May\or Jun\or Jul\or Aug\or Sep
\or Oct\or Nov\or Dec\fi\ }
\def\mmddyyyy{\the\month.\the\day.\the\year}
\def\ddmmyyyy{\the\day.\the\month.\the\year}
\def\Mddyyyy{\fmonth~\the\day,~\the\year}

\def\R{\bbR}
\def\N{\bbN}
\def\Z{\bbZ}
\def\Q{\bbQ}
\def\P{\bbP}
\def\E{\bbE}

\providecommand{\abs}[1]{\left\vert#1\right\vert}
\providecommand{\norm}[1]{\left\Vert#1\right\Vert}
\providecommand{\pp}[1]{\langle#1\rangle}

\numberwithin{equation}{section}

\allowdisplaybreaks[1]

\begin{document}

%%%%%%%
%%% Author(s), abstract, and other info
%%%%%%%

\author[F.~Rassoul-Agha]{Firas~Rassoul-Agha$^1$}   % First
\thanks{$^1$Department of Mathematics, University of Utah}
\thanks{$^1$Supported in part by NSF Grant DMS-0505030}
\address{F.~Rassoul-Agha, 155 S 1400 E, Salt Lake City, UT 84112}
\email{firas@math.utah.edu}
\urladdr{www.math.utah.edu/$\sim$firas}
\author[T.~Sepp\"al\"ainen]{Timo~Sepp\"al\"ainen$^{2}$} % Second
\thanks{$^2$Mathematics Department, University of Wisconsin-Madison}
\thanks{$^2$Supported in part by NSF Grant DMS-0402231}
\address{T.~Sepp\"al\"ainen, 419 Van Vleck Hall, Madison, WI 53706}
\email{seppalai@math.wisc.edu}
\urladdr{www.math.wisc.edu/$\sim$seppalai}

\date{\today}
% this will put the date as a footnote on the first page


%\frenchspacing

\begin{abstract}
%Insert your abstract here.
We consider a non-nestling random walk in a product
random environment. We assume an exponential moment for the step of 
the walk, uniformly in the environment.
We prove  an invariance principle (functional central
limit theorem) under almost every  environment for the centered and
diffusively scaled walk.  
The main point behind the invariance principle is that the quenched mean of the walk
behaves subdiffusively.
\end{abstract}

\title[Quenched Functional CLT for RWRE]
{Almost sure  functional central limit theorem  for non-nestling
random walk in random environment}

% You can put the date manually or use \today, \mmddyyyy,
% \ddmmyyy, or \Mddyyyy  to have it automatically

\maketitle

%% add more definitions here (and if you want them incorporated
%% in the template, tell Firas about them)
\def\bfP{{\mathbf P}}
\def\bfE{{\mathbf E}}
\def\R{\bbR}
\def\N{\bbN}
\def\Z{\bbZ}
\def\Q{\bbQ}
\def\P{\bbP}
\def\E{\bbE}
\def\xbar{{{\bar x}}}
\def\Qtil{{{\widetilde Q}}}
\def\si{{\sigma}}

%TEXT STARTS HERE

\section{Introduction and main result}
\label{intro}
We prove a quenched functional central limit theorem for
non-nestling random walk in random environment (RWRE) on the  $d$-dimensional
integer lattice $\Z^d$ in dimensions $d\ge 2$. 
Here is a  general description of  the model,
fairly standard since quite a while. 
An environment $\w$  is a configuration of transition probability vectors
$\w=(\w_x)_{x\in\Z^d}\in\Omega={\mathcal P}^{\Z^d},$  where
$\mathcal P=\{(p_z)_{z\in\Z^d}:p_z\ge0,\sum_z p_z=1\}$ is the simplex of all
probability  vectors on $\Z^d$.
Vector $\w_x=(\w_{x,z})_{z\in\Z^d}$ gives 
the transition probabilities out of state $x$,
denoted by $\pi_{x,y}(\w)=\w_{x,y-x}$. 
To run the random walk, fix an environment $\w$ and 
 an initial state $z\in\Z^d$.
 The random walk  $X_{0,\infty}=(X_n)_{n\geq0}$ in environment $\w$ 
started at $z$ is then the canonical
Markov chain with state space $\Z^d$
whose path measure  $P_z^\w$ satisfies
\begin{align*}
P_z^\w(X_0=z)=1\quad\text{and}\quad 
P_z^\w(X_{n+1}=y|X_n=x)=\pi_{x,y}(\w).
\end{align*}

On the space $\Omega$ we put its  product
$\si$-field $\kS$,  natural shifts 
$\pi_{x,y}(T_z\w)=\pi_{x+z,y+z}(\w)$, and  
 a $\{T_z\}$-invariant probability measure $\P$ that makes the system
$(\Omega,\kS,(T_z)_{z\in\Z^d},\P)$ ergodic.  In this paper
 $\P$ is an  i.i.d.\  product  measure on ${\mathcal P}^{\Z^d}$.
In other words,    the vectors 
$(\w_x)_{x\in\Z^d}$ are i.i.d.\ across the sites $x$ under $\P$.

Statements, probabilities and expectations 
 under a fixed environment, such as the distribution  $P_z^\w$ above, are 
called {\sl quenched}.  When also the environment is averaged out,
the notions are called {\sl averaged}, or also {\sl annealed}. 
In particular,  the averaged distribution  $P_z(d x_{0,\infty})$ of the walk 
is the marginal of the joint  distribution 
$P_z(d x_{0,\infty},d\w)=P_z^\w(d{x}_{0,\infty})\P(d\w)$
on paths and environments. 

Several excellent expositions on RWRE exist,
and we refer the reader to the lectures \cite{tenlectures},
\cite{rwrelectures} and 
\cite{stflour}.  
 We turn to  the specialized
assumptions imposed on the model in this paper. 

The main assumption is  {\sl non-nestling} (N) which guarantees a
drift  uniformly over the 
environments.   The terminology 
was introduced by Zerner 
\cite{zernerldp}.  

\newtheorem*{hypothesisN}{\sc Hypothesis (N)}
\begin{hypothesisN}
There exists a vector $\uhat\in\Z^d\setminus\{0\}$ and a constant 
$\delta>0$ such that
\[\P\Bigl\{\w\,:\, \sum_{z\in\Z^d} z\cdot\uhat\,\pi_{0,z}(\w)
\geq\delta\Bigr\}=1.\]
\end{hypothesisN}

There is no harm in assuming $\uhat\in\Z^d$, and this
is convenient. 
We utilize  two auxiliary assumptions: an
exponential moment bound (M) on the steps of
the walk, and some regularity (R) on the environments.

\newtheorem*{hypothesisM}{\sc Hypothesis (M)}
\begin{hypothesisM}
There exist positive constants $M$ and ${s}_0$
such that 
\[\P\Bigl\{\w\,:\, \sum_{z\in\Z^d}e^{{s}_0|z|}\pi_{0,z}(\w)\leq e^{{s}_0 M}\Bigr\}=1.\]
\end{hypothesisM}

\newtheorem*{hypothesisWE}{\sc Hypothesis (R)}
\begin{hypothesisWE}
There exists a  constant $\kappa>0$ 
such that 
\be \P\Bigl\{\w: \, \sum_{z:\,z \cdot\uhat=1}  \pi_{0,z}(\w)\ge\kappa \,\Bigr\} =1. 
\label{level-ell}\ee
Let ${\mathcal J}=\{z:\E\pi_{0,z}>0\}$ be the set of admissible steps
under $\P$.  Then    
\be
\P\{\forall z:\pi_{0,0}+\pi_{0,z}<1\}>0  \quad\text{and}\quad
{\mathcal J}\not\subset\R u \ \text{ for all $u\in\R^d$.}
\label{Yell5}
\ee  
\end{hypothesisWE}

Assumption \eqref{level-ell} above is
 stronger than needed.   In the proofs it is
actually used in the form \eqref{level-ell-1}  [Section
\ref{intersectbound}] that permits backtracking before 
hitting the level $x\cdot\uhat=1$.  At the expense of additional
technicalities  in  Section
\ref{intersectbound}  quenched assumption \eqref{level-ell} can be 
replaced by an averaged requirement. 

Assumption \eqref{Yell5}
 is used in Lemma  \ref{Yapplm1}.   
It  is necessary for the  quenched CLT 
%for process  $B_n$ defined in \eqref{Bndef} below. 
%This 
as was discovered already in the simpler forbidden direction
case we studied  in 
\cite{forbidden-alea}  and \cite{forbidden-qclt}. 
Note that assumption \eqref{Yell5} rules out the case $d=1$.   
However, the issue is not whether the walk is genuinely $d$-dimensional, 
but whether the walk can explore its environment thoroughly enough to 
suppress the fluctuations of the quenched mean. 
  Most work on RWRE takes uniform
ellipticity and  nearest-neighbor jumps as standing assumptions, 
which of course imply Hypotheses (M) and (R). 

These assumptions are more than strong enough to imply
a law of large numbers: there exists a  velocity $v\ne 0$ such that 
\be
P_0\bigl\{\,\lim_{n\to\infty}n^{-1}X_n=v\bigr\}=1.
\label{lln1}\ee
  Representations for $v$ are given in \eqref{defv} 
and Lemma \ref{velocity}. 
Define the (approximately)
 centered and diffusively scaled process 
\begin{align}
B_n(t)=\frac{X_{[nt]}-[nt]v}{\sqrt{n}}. 
%\quad\text{and}\quad
%\Btil_n(t)=\frac{X_{[nt]}-E_0^\w(X_{[nt]})}{\sqrt{n}}\,,\quad   t\in[0,\infty).
\label{Bndef}
\end{align}
As usual  $[x]=\max\{n\in\Z: n\le x\}$ is the integer part of a real $x$.
Let $D_{\R^d}[0,\infty)$
be the standard Skorohod space of $\R^d$-valued cadlag paths
(see \cite{EK} for the basics).
Let $Q_n^\w=P^\w_0(B_n\in\cdot\,)$ 
%and  $\Qtil_n^\w=P^\w_0(\Btil_n\in\cdot\,)$ 
denote the quenched distribution
of the process $B_n$ on $D_{\R^d}[0,\infty)$.

The results of this paper
 concern the limit of the process $B_n$
% and $\Btil_n$ 
as $n\to\infty$.  
As expected, the limit process is a Brownian motion with 
correlated coordinates.  
 For a symmetric, non-negative definite 
$d\times d$ matrix $\mathfrak D$,
a {\sl Brownian motion with diffusion matrix $\mathfrak D$} is the $\R^d$-valued
process $\{{B}(t):t\geq0\}$ with 
 continuous paths, independent increments,
and such that  for $s<t$ the $d$-vector ${B}(t)-{B}(s)$ has Gaussian distribution
with mean zero and covariance matrix
$(t-s)\mathfrak D$.
The matrix $\mathfrak D$  is {\sl degenerate} in direction $u\in\R^d$
if $u^t\mathfrak D u=0$.   Equivalently,
$u\cdot {B}(t)=0$ almost surely.

Here is the main result.  

\begin{theorem}
\label{main}
Let $d\geq2$ and
consider a random walk in an i.i.d.\ product 
random environment that satisfies 
  non-nestling {\rm(N)}, 
 the exponential moment hypothesis {\rm(M)}, and
the  regularity in {\rm(R)}. 
Then for $\P$-almost every   
$\w$  distributions $Q_n^\w$  %% and  $\Qtil_n^\w$ 
converge weakly on $D_{\R^d}[0,\infty)$  to the distribution of a Brownian 
motion with a diffusion matrix $\kD$ that is independent of 
$\w$.  $u^t\kD u=0$ iff $u$ is orthogonal to the span of 
$\{x-y: \E(\pi_{0x})\E(\pi_{0y})>0\}$. 
\end{theorem}

Eqn \eqref{defkD} gives the expression for the diffusion matrix 
$\kD$, familiar for example from \cite{slowdown+clt}. 
Before turning to the proofs we discuss briefly the current
situation in this area of probability and the place of this
work in this context. 

Several different  approaches can be identified in
recent work on quenched central limit theorems for
multidimensional RWRE.
 (i) Small perturbations of classical random walk 
have been studied by many authors. The most significant 
results include the early work of Bricmont and Kupiainen \cite{BK}
and more recently Sznitman and Zeitouni  \cite{szni-zeit-06}
for small perturbations of Brownian motion in
dimension $d\ge 3$.
(ii) An averaged CLT can be turned into a quenched CLT
by  bounding certain variances through the control of intersections
of two independent paths.  This idea was introduced
by Bolthausen and Sznitman in \cite{dynstat}
and more recently applied by Berger and Zeitouni in \cite{berg-zeit-07-}. 
Both 
utilize high dimension to handle the intersections. 
   (iii) Our approach is based on the subdiffusivity 
of the quenched mean of the walk.  That is, we show that  
the variance of $E^\w_0(X_n)$ is of order $n^{2\alpha}$ for some
$\alpha<1/2$. 
We also achieve this through intersection bounds. Instead 
of high dimension we assume  strong enough drift.  
We introduced this line of reasoning in \cite{qclt-spacetime}
 and later applied it to the 
case of walks with a forbidden direction in \cite{forbidden-qclt}. 
The significant advance taken in the present paper over 
 \cite{qclt-spacetime} and  \cite{forbidden-qclt} 
 is the elimination of restrictions  on the
admissible steps of the  walk. 
Theorem \ref{RS} below summarizes the general principle for
application in this paper.

As the reader will see, the arguments in this paper are based on
quenched exponential bounds that flow from Hypotheses (N), (M) and (R). 
It is common in this field to look for an invariant measure $\P_\infty$
for the environment process that is mutually absolutely continuous
with the original $\P$,  at least on the 
 part of the space $\Omega$ to which the drift points.  In this
paper we do things a little differently: instead of the 
absolute continuity, we use bounds on the variation
distance between $\P_\infty$ and $\P$.  This distance will
decay exponentially in the direction $\uhat$.   

 In the case of
nearest-neighbor, uniformly elliptic  non-nestling 
 walks in dimension $d\ge 4$
the quenched CLT  has been  proved  earlier:  
first by  Bolthausen and Sznitman \cite{dynstat} under  
a small noise assumption, 
and recently   by Berger and Zeitouni \cite{berg-zeit-07-} without the 
small noise assumption.
Berger and Zeitouni \cite{berg-zeit-07-}
 go beyond non-nestling to more general ballistic walks.
The method in these two papers  utilizes high 
dimension crucially. Whether their argument can work in $d=3$ 
is not presently clear.   The  approach of the present paper
should  work for  more general ballistic walks in all dimensions $d\ge 2$, 
as the main 
technical step that reduces the variance estimate to an
intersection estimate is generalized (Section 
\ref{pathintersections} in the present paper).  

We turn to  the proofs. The next section collects some
preliminary material and finishes with an outline of the rest
of the paper.

\section{Preliminaries for the proof.} 
\label{prelim}
As mentioned, we can assume
that $\uhat\in\Z^d$.  This is convenient 
because then the lattice $\Z^d$ decomposes into {\sl levels}
identified by the integer value   $x\cdot\uhat$. 

Let us summarize  notation for the reader's convenience.
Constants whose exact values are not important and can change
from line to line are often denoted by $C$ and $s$. 
The set of nonnegative integers is  $\N=\{0,1,2,\dotsc\}$.
  Vectors and 
sequences are abbreviated $x_{m,n}=(x_m, x_{m+1},\dotsc,x_n)$
and $x_{m,\infty}=(x_m, x_{m+1}, x_{m+2}, \dotsc)$.  Similar
notation is used for finite and infinite random paths:
 $X_{m,n}$ $=$ $(X_m$, $X_{m+1},$ $\dotsc,$ $X_n)$
and $X_{m,\infty}$ $=$ $(X_m,$ $X_{m+1},$ $ X_{m+2}, \dotsc)$.
$X_{[0,n]}=\{X_k:0\leq k\leq n\}$ denotes the set of sites visited
by the walk. 
 ${\mathfrak D}^t$ is the transpose of a vector or matrix
${\mathfrak D}$. An element of $\R^d$ is regarded as a $d\times 1$ 
 column vector.
 The left shift on the path 
space  $(\Z^d)^\N$ is 
 $(\theta^kx_{0,\infty})_n = x_{n+k}$. 

 $\E$, $E_0$, and $E_0^\w$ denote
expectations under, respectively, $\P$, $P_0$, and $P_0^\w$.
$\P_\infty$ will denote an invariant measure on $\Omega$, with
expectation $\E_\infty$. We abbreviate 
$P^\infty_0(\cdot)=\E_\infty P^\w_0(\cdot)$
and $E^\infty_0(\cdot)=\E_\infty E^\w_0(\cdot)$ to indicate
that the environment of a quenched expectation is averaged
under $\P_\infty$. 
A family of  $\sigma$-algebras  on $\Omega$ that 
in a sense look towards the future  is defined by 
${\kS}_\ell=\sigma\{\w_x: x\cdot \uhat\geq \ell\}$.

Define the {\sl drift} 
\[D(\w)=E_0^\w(X_1)=\sum_z z\pi_{0z}(\w).\]
The {\sl environment process}  is the   Markov chain   
on $\Omega$ with transition kernel
\[\Pi(\w,A)=P_0^\w(T_{X_1}\w\in A).\]

The proof of the  quenched CLT  Theorem \ref{main}  utilizes
crucially the environment process and its invariant 
distribution. A preliminary part of the proof  is summarized in 
the next theorem quoted from  \cite{qclt-spacetime}.  This  Theorem \ref{RS} 
was proved by  applying the arguments  
 of   Maxwell and Woodroofe \cite{MW} and Derriennic and Lin
\cite{DL} to the environment process. 

%\newtheorem*{theoremRS}{\sc Theorem RS}
%\begin{theoremRS}
\begin{theorem}  {\rm \cite{qclt-spacetime}}
Let $d\geq1$.  Suppose the  probability measure  $\P_\infty$  on
$(\Omega,{\kS})$ is invariant and ergodic for the Markov
transition  $\Pi$. Assume that
$\sum_z|z|^2\E_\infty(\pi_{0z})<\infty$ and  that there
exists an $\alpha<1/2$ such that as $n\to\infty$ 
\begin{align}
\E_\infty\bigl[\,\abs{E_0^\w(X_n)-n\E_\infty(D)}^2\,\bigr]=\Ord(n^{2\alpha}).
\label{cond}
\end{align}
Then  as $n\to\infty$ the following weak limit happens 
 for $\P_\infty$-a.e.\ $\w$:  distributions  $Q_n^\w$ 
%%%and  $\Qtil_n^\w$ 
converge weakly  on the space
 $D_{\R^d}[0,\infty)$ to the distribution of a Brownian motion
with a symmetric, non-negative definite diffusion matrix
$\mathfrak D$  
that is independent of $\w$.
\label{RS}\end{theorem}
%\end{theoremRS}

Another  central tool  for the development that follows is provided  by the 
{\sl Sznitman-Zerner 
 regeneration times} \cite{szlln} that we now define.  
 For $\ell\ge0$ let 
$\lambda_\ell$ be the first time the walk reaches
level $\ell$  relative to the initial level:
\[\lambda_\ell=\min\{n\ge0:X_n\cdot\uhat-X_0\cdot\uhat\ge\ell\}.\]
Define  $\beta$ to be the first backtracking
time:
\[\beta=\inf\{n\ge0:X_n\cdot\uhat<X_0\cdot\uhat\}.\]
Let $M_n$ be the maximum level, relative to the starting level, 
 reached by time $n$:
\[M_n=\max\{X_k\cdot\uhat-X_0\cdot\uhat:0\le k\le n\}.\]
For $a>0$, and when $\beta<\infty$,
consider the first time by which the walker reaches level
$M_\beta+a$:
\[\lambda_{M_\beta+a}=\inf\{n\ge\beta:X_n\cdot\uhat-X_0\cdot\uhat\ge M_\beta+a\}.\]
Let $S_0=\lambda_a$ and, as long as $\beta\circ\theta^{S_{k-1}}<\infty$, 
define
 $S_k=S_{k-1}+\lambda_{M_\beta+a}\circ\theta^{S_{k-1}}$   for $k\ge1$.
Finally, let the first regeneration time be
\begin{align}
\label{tau}
\tau_1^{(a)}=\sum_{\ell\ge0}S_\ell
\one\{\beta\circ\theta^{S_k}<\infty
\text{ for }0\le k<\ell\text{ and }
\beta\circ\theta^{S_\ell}=\infty\}.
\end{align}
Non-nestling  guarantees 
that $\tau^{(a)}_1$ is finite,
and in fact gives moment bounds uniformly in $\w$ as  we see
in Lemma \ref{exponential} below. 
Consequently we can iterate to  define  $\tau_0^{(a)}=0$, and
for $k\ge 1$
\be
\tau_k^{(a)}=\tau_{k-1}^{(a)}+\tau_1^{(a)}\circ\theta^{\tau_{k-1}^{(a)}}.
\label{tauk}\ee

When the value of $a$ is not important we simplify the 
notation to $\tau_k=\tau^{(a)}_k$. 
Sznitman and Zerner \cite{szlln}  proved that the {\sl regeneration slabs} 
\be
\begin{split}
\cS_k=&\bigl( \tau_{k+1}-\tau_k,\,
 (X_{\tau_k+n}-X_{\tau_k})_{0\le n\le \tau_{k+1}-\tau_k},\,\\
&\qquad\qquad 
\{\w_{X_{\tau_k}+z}: 0\le z\cdot\uhat<(X_{\tau_{k+1}}-X_{\tau_k})\cdot\uhat\}
\bigr)  \end{split} 
\label{regenslab}\ee
are   i.i.d.\ for $k\ge 1$, each distributed   as
$\bigl( \tau_{1},\,
 (X_{n})_{0\le n\le \tau_{1}},\,
\{\w_{z}: 0 \le z\cdot\uhat<X_{\tau_{1}}\cdot\uhat\}\bigr)$
under $P_0(\,\cdot\,\vert\,\beta=\infty)$. 
Strictly speaking, uniform ellipticity and  nearest-neighbor jumps
 were standing assumptions in \cite{szlln}, but these assumptions
are not needed for the proof of the i.i.d.\ structure. 

From the renewal structure and moment estimates 
a law of large numbers \eqref{lln1} and an averaged functional central
limit theorem follow, along the lines of Theorem 2.3 
in \cite{szlln} and
Theorem 4.1 in \cite{slowdown+clt}.  These references treat walks that 
satisfy Kalikow's condition,  considerably more general than the 
non-nestling walks we study.   
   The limiting velocity for the 
law of large numbers  is
\be
v=\frac{E_0(X_{\tau_1}\vert\beta=\infty)}{E_0({\tau_1}\vert\beta=\infty)}.
\label{defv}\ee
The averaged CLT states that the distributions  $P_0\{B_n\in\,\cdot\,\}$ 
converge to the distribution of a Brownian motion with 
diffusion matrix
\be
\kD=\frac{E_0\bigl[(X_{\tau_1}-\tau_1v)(X_{\tau_1}-\tau_1v)^t\vert\beta=\infty\bigr]}
{E_0[{\tau_1}\vert\beta=\infty]}.
\label{defkD}
\ee

Once we know 
 that the $\P$-a.s.\ quenched CLT holds with a constant diffusion matrix,
 this diffusion matrix must be the same $\kD$  as for the 
averaged CLT.  
We give here  the argument for the degeneracy statement of
Theorem \ref{main}. 

\begin{lemma}  Define  $\kD$  by {\rm \eqref{defkD}} and let $u\in\R^d$. 
Then 
 $u^t\kD u=0$ iff $u$ is orthogonal to the span of 
$\{x-y: \E(\pi_{0x})\E(\pi_{0y})>0\}$.
\end{lemma} 

\begin{proof} The argument is a minor embellishment of that given
for a similar degeneracy  statement 
on p.~123--124  of \cite{forbidden-alea} for the forbidden-direction
 case where $\pi_{0,z}$ is supported by $z\cdot\uhat\ge 0$.
We spell out enough  of the argument  to show 
 how to adapt that proof to the present case.

Again, the intermediate step is to show that  $u^t\kD u=0$ 
iff $u$ is orthogonal to the span of 
$\{x-v: \E(\pi_{0x})>0\}$.  The argument from orthogonality
to  $u^t\kD u=0$  goes as in \cite[p.~124]{forbidden-alea}. 

Suppose $u^t\kD u=0$ which is the same as 
\[ P_0(X_{\tau_1}\cdot u
=\tau_1 v\cdot u \,\vert\,\beta=\infty)=1.\] 
Suppose $z$ is such that $\E\pi_{0,z}>0$ and $z\cdot\uhat< 0$. 
By non-nestling there must exist $w$ such that  $\E\pi_{0,z}\pi_{0,w}>0$
 and $w\cdot\uhat>0$.  Pick $m>0$ so that $(z+mw)\cdot\uhat>0$ but 
$(z+(m-1)w)\cdot\uhat\le 0$. Take $a=1$ in the definition 
\eqref{tau} of regeneration.  Then  
\begin{align*}
&P_0[X_{\tau_1}=z+2mw, \tau_1=2m+1\,\vert\,\beta=\infty]\\
&\ge
\E\Bigl[\,\Bigl(\;\prod_{i=0}^{m-1}\pi_{iw,(i+1)w}\Bigr)
\pi_{mw,z+mw}\Bigl(\;\prod_{j=0}^{m-1}\pi_{z+(m+j)w,z+(m+j+1)w}\Bigr)
P^\w_{z+2mw}(\beta=\infty)\Bigr]>0.
\end{align*}
Consequently 
\be (z+2mw)\cdot u=(1+2m)v\cdot u. \label{degenaux1}\ee

In this manner, by replacing $\sigma_1$ with $\tau_1$ 
and by adding in the no-backtracking probabilities, the arguments
 in \cite[p.~123]{forbidden-alea} can be repeated 
to show that if $\E\pi_{0x}>0$ then $x\cdot u=v\cdot u$ for $x$ such
that $x\cdot\uhat\ge 0$. 
In particular the very first step on p.~123  of \cite{forbidden-alea} 
gives $w\cdot u=v\cdot u$. 
This combines with \eqref{degenaux1} above 
to give $z\cdot u=v\cdot u$.   Now simply follow the 
proof in \cite[p.~123--124]{forbidden-alea} to its conclusion. 
\end{proof}

Here is  an outline of the proof of
Theorem \ref{main}.  It all goes via Theorem \ref{RS}.

\smallskip

(i) After some basic estimates in Section \ref{nonnestling},
we prove in Section \ref{Pinfty_exists}
 the existence of the ergodic equilibrium $\P_\infty$ required
for Theorem \ref{RS}.  $\P_\infty$ is not convenient to work with
 so we still need  to do  computations with
$\P$. For this purpose
Section \ref{Pinfty_exists} proves that 
 in the direction $\uhat$ 
the measures $\P_\infty$ and $\P$ come exponentially close in variation 
distance and that the environment process satisfies
 a $P_0$-a.s.\ ergodic theorem. 
 In Section \ref{substitution}  we show that $\P_\infty$
and $\P$  are interchangeable both in the hypotheses that
need to be checked 
and in the conclusions obtained. 
In particular, the $\P_\infty$-a.s.\ quenched CLT
coming from Theorem \ref{RS} holds also $\P$-a.s.  Then we know
that the diffusion matrix $\kD$ is the one in \eqref{defkD}. 

\smallskip

 The bulk of the work goes towards verifying condition
 \eqref{cond}, but under $\P$ instead of $\P_\infty$. 
There are two main stages to this argument. 

\smallskip

(ii) By a decomposition into martingale increments
the proof of \eqref{cond}  reduces to bounding the number of common
points of two independent walks in a common environment
(Section \ref{pathintersections}).

\smallskip

(iii) The intersections are controlled by introducing levels 
at which both walks regenerate. These 
common regeneration levels are reached fast enough and 
the progression from one
common regeneration level to the next is a Markov chain.
When this Markov chain drifts away from the origin it 
 can be approximated well enough by a symmetric random
walk.  This approximation enables us to control the growth 
of the Green function of the Markov chain, and thereby the number of 
common points. This is in Section \ref{intersectbound} and 
in an Appendix devoted to  the Green function bound.

\section{Basic estimates for non-nestling RWRE}
\label{nonnestling}
This section contains estimates that follow from
  Hypotheses  (N) and (M), all collected in the following
lemma.  These will be used repeatedly.
In addition to the stopping  times 
already defined, let \[H_z=\min\{n\ge1:X_n=z\}\]  be the first hitting
time of site $z$.

\begin{lemma}
\label{exponential}
If $\P$ satisfies Hypotheses {\rm(N)} and {\rm(M)}, 
then there exist positive
constants $\eta$, $\gamma$, $\kappa$,
$(C_p)_{p\ge1}$, and ${s}_1\le{s}_0$, 
possibly depending on $M$, ${s}_0$, and $\delta$, such that
 for all $x\in\Z^d$, $n\ge0$, ${s}\in[0,{s}_1]$, $p\ge1$, 
$\ell\ge1$, for $z$ such that $z\cdot\uhat\ge0$, $a\ge1$, 
and for $\P$-a.e.\ $\w$,
\begin{align}
&E_x^\w(e^{-{s} X_n\cdot\uhat})\leq 
e^{-{s} x\cdot\uhat}(1-{s}\delta/2)^n,
\label{exp-bound}\\
&E_x^\w(e^{{s}|X_n-x|})\leq e^{{s} Mn},
\label{bounded-step}\\
&P_x^\w(X_1\cdot\uhat\ge x\cdot\uhat+\gamma)\ge\kappa,
\label{E}\\
&E_x^\w(\lambda_\ell^p)\leq C_p\ell^p,
\label{si-bound}\\
&E_x^\w(|X_{\lambda_\ell}-x|^p)\leq C_p\ell^p,
\label{X-si-bound}\\
&E_0^\w[(M_{H_z}-z\cdot\uhat)^p\one\{H_z<n\}]\le C_p\ell^p P_0^\w(H_z<n)+
C_p s^{-p}e^{-{s}\ell/2},
\label{Hz-bound}\\
&P_x^\w(\beta=\infty)\ge \eta,
\label{beta-bound}\\
&E_x^\w(|\tau_1^{(a)}|^p)\leq C_p\, a^p,
\label{tau-bound}\\
&E_x^\w(|X_{\tau_1^{(a)}+n}-X_n|^p)\leq C_q\, a^q,\text{ for all }q>p.
\label{X-tau-bound}
\end{align}
\end{lemma}

The particular point in \eqref{tau-bound}--\eqref{X-tau-bound}
is to make the dependence on $a$ explicit. 
Note that \eqref{beta-bound}--\eqref{tau-bound} give 
\be
E_0(\tau_j-\tau_{j-1})^p<\infty
\label{tau-bound2}
\ee
for all $j\ge 1$.  
In Section \ref{Pinfty_exists}
 we construct an ergodic invariant measure $\P_\infty$
for the environment chain in a way that preserves the conclusions 
of this lemma under $\P_\infty$. 

\begin{proof}
Replacing $x$ by $0$ and $\w$ by $T_x\w$ allows us to assume
 that $x=0$. Then for all 
${s}\in[0,{s}_0/2]$
\begin{align*}
\big| E_0^\w(e^{-{s} X_1\cdot\uhat})-1+{s} E_0^\w(X_1\cdot\uhat)\big|
&\le|\uhat|^2E_0^\w(|X_1|^2 e^{{s}_0|X_1|/2})\frac{{s}^2}2 \\
&\le
(2|\uhat|/{s}_0)^2e^{{s}_0 M}{s}^2=c{s}^2,\end{align*}
where we used moment assumption (M).
Then by  the non-nestling assumption (N)
\begin{align*}
E_0^\w(e^{-{s} X_n\cdot\uhat}|X_{n-1})&=
e^{-{s} X_{n-1}\cdot\uhat}E_{X_{n-1}}^\w(e^{-{s} (X_1-X_0)\cdot\uhat})
\leq e^{-{s} X_{n-1}\cdot\uhat}(1-{s}\delta+c{s}^2).
\end{align*}
Taking now the quenched expectation of both sides and iterating the procedure
proves \eqref{exp-bound}, provided ${s}_1$ is small enough.
To prove \eqref{bounded-step} one can instead show that
\[E_0^\w(e^{{s}\sum_{k=1}^n|X_k-X_{k-1}|})\le e^{{s} n M}.\]
This can be proved by induction as for  \eqref{exp-bound},
using only Hypothesis (M) and H\"older's inequality (to switch to ${s}_0$). 

Concerning \eqref{E}, we have
\[P_0^\w(X_1\cdot\uhat\ge\gamma)\ge(1-e^{\gamma{s}}(1-{s}\delta/2))
\mathop{\longrightarrow}_{\gamma\to0}{s}\delta/2.\]
So taking $\gamma$ small enough and $\kappa$ slightly smaller than 
${s}\delta/2$ does the job.

Notice next that  $P_0^\w(\lambda_1<\infty)=1$ 
 due to \eqref{exp-bound}.
$\P$-a.s. Then 
\begin{align*}
E_0^\w(\lambda_1^p)&\le \sum_{n\ge0}(n+1)^p P_0^\w(\lambda_1>n)
\le \sum_{n\ge0}(n+1)^p P_0^\w(X_n\cdot\uhat\le 1)\\
&\le e^{{s}}\sum_{n\ge0}(n+1)^p E_0^\w(e^{-{s} X_n\cdot\uhat}).
\end{align*}
The last expression  is bounded if ${s}$ is small enough. 
Therefore,
\begin{align*}
E_0^\w(\lambda_\ell^p)&\le E_0^\w\bigg[\,\Big|\sum_{i=1}^{[\ell]+1}
(\lambda_i-\lambda_{i-1})\Big|^p\,\bigg]\le
([\ell]+1)^{p-1}\sum_{i=1}^{[\ell]+1} E_0^\w\big[E_{X_{\lambda_{i-1}}}^\w
(\lambda_1^p)\,\big]\le 
C_p\ell^p.
\end{align*}
Bound \eqref{X-si-bound} is proved similarly: by  the
Cauchy-Schwarz inequality,   Hypothesis (M) and \eqref{exp-bound},
\begin{align*}
E_0^\w(|X_{\lambda_1}|^p)&\le\sum_{n\ge1}E_0^\w(|X_n|^{2p})^{1/2}
P_0^\w(X_{n-1}\cdot\uhat<1)^{1/2}\\
&\le \bigl([2p]!\,{s}_0^{-[2p]}e^{{s}_0 M}e^s\bigr)^{1/2}
\,\sum_{n\ge1}(1-s\delta/2)^{(n-1)/2}n^p\le C_p.
\end{align*}

To prove \eqref{Hz-bound}, write
\begin{align*}
&E_0^\w[(M_{H_z}-z\cdot\uhat)^p\one\{H_z<n\}]\\
&\qquad\le C_p\sum_{\ell>\ell_0}\ell^{p-1} P_0^\w(M_{H_z}-z\cdot\uhat\ge\ell,H_z<n)+
C_p\ell_0^p P_0^\w(H_z<n)\\
&\qquad\le C_p\sum_{\ell>\ell_0}\sum_{k\ge0}\ell^{p-1}
E_0^\w[P_{X_{\lambda_{z\cdot\uhat+\ell}}}^\w(X_k\cdot\uhat-X_0\cdot\uhat\le-\ell)]+
C_p\ell_0^p P_0^\w(H_z<n)\\
&\qquad\le C_p\sum_{\ell>\ell_0}\ell^{p-1} e^{-{s}\ell}+
C_p\ell_0^p P_0^\w(H_z<n)
\le C_p s^{-p}e^{-{s}\ell_0/2}+C_p\ell_0^p P_0^\w(H_z<n).
\end{align*}

To prove \eqref{beta-bound}, note that 
 Chebyshev inequality and \eqref{exp-bound} give, 
 for ${s}>0$ small enough, $\ell\ge1$, and $\P$-a.e.\ $\w$
\begin{align*}
P_0^\w(\lambda_{-\ell+1}<\infty)\le \sum_{n\ge0} P_0^\w(X_n\cdot\uhat\le-(\ell-1))
\le 2({s}\delta)^{-1}e^{-{s}(\ell-1)}.
\end{align*}
On the other hand, for an integer $\ell\ge2$ we have
\begin{align*}
P_0^\w(\lambda_\ell<\beta)\ge\sum_x P_0^\w(\lambda_{\ell-1}<\beta,X_{\lambda_{\ell-1}}=x)
P_x^\w(\lambda_{-\ell+1}=\infty).
\end{align*}
Therefore, taking $\ell$ to infinity one has, for $\ell_0$ large enough,
\[P_0^\w(\beta=\infty)\ge
P_0^\w(\lambda_{\ell_0}<\beta)\prod_{\ell\ge\ell_0}(1-2({s}\delta)^{-1}e^{-{s}\ell}).
\]
 Markov property and \eqref{E} give 
$P_0^\w(\lambda_{\ell_0}<\beta)\ge\kappa^{\ell_0/\gamma+1}>0$ and 
 \eqref{beta-bound} is proved.

Now we will bound the quenched expectation of 
$\lambda_{M_\beta+a}^p\one\{\beta<\infty\}$ uniformly in $\w$.
To this end, for $p_1>p$ and $q_1=p_1(p_1-p)^{-1}$, we have by \eqref{beta-bound}
\begin{align*}
E_0^\w(\lambda_{M_\beta+a}^p\one\{\beta<\infty\})
&\le\sum_{n\ge1}E_0^\w(\lambda_{M_n+a}^p\one\{\beta=n\})\\
&\le\sum_{n\ge1}\Big(E_0^\w(\lambda_{M_n+a}^{p_1})\Big)^{p/p_1}
\Big(P_0^\w(\beta=n)\Big)^{1/q_1}.
\end{align*}
By \eqref{si-bound} one has, for $p_2>p_1>p$ and $q_2=p_2(p_2-p_1)^{-1}$,
\begin{align*}
E_0^\w(\lambda_{M_n+a}^{p_1})
&\le\sum_{m\ge0}\Big(E_0^\w(\lambda_{m+1+a}^{p_2})\Big)^{p_1/p_2}
\Big(P_0^\w([M_n]=m)\Big)^{1/q_2}\\
&\le C_p\sum_{m\ge0}(m+1+a)^{p_1}
\Big(\sum_{i=0}^n P_0^\w(X_i\cdot\uhat\ge m)\Big)^{1/q_2},
\end{align*}
where $C_p$ really depends on $p_1$ and $p_2$, but these are chosen 
arbitrarily, as long as they satisfy $p_2>p_1>p$.
Using \eqref{bounded-step} one has
\[P_0^\w(X_i\cdot\uhat\ge m)\le\left\{
\begin{matrix}
1\hfill&\text{if }m< 2M|\uhat|i,\\
e^{-{s} m}e^{M|\uhat|{s} i}\hfill&\text{if }m\ge 2M|\uhat|i.
\end{matrix}\right.\]
Hence, 
\begin{align*}
E_0^\w(\lambda_{M_n+a}^{p_1})&\le 
C_p
\sum_{m\ge0}(m+1+a)^{p_1}(n\one\{m<2Mn|\uhat|\}+e^{-{s} m/2})^{1/q_2}\\
&\le C_p n(n+a)^{p_1}n^{1/q_2}+C_p\sum_{m\ge0}(m+1)^{p_1}e^{-sm/2q_2}
+C_pa^{p_1}\sum_{m\ge0}e^{-sm/2q_2}\\
&\le C_p n^{1+1/q_2} (n+a)^{p_1}.
\end{align*}
Since $\{\beta=n\}\subset\{X_n\cdot\uhat\le0\}$, one can use
\eqref{exp-bound} to conclude that
\[E_0^\w(\lambda_{M_\beta+a}^p\one\{\beta<\infty\})\le 
C_p
\sum_{n\ge1}n^{p/p_1+p/(p_1q_2)}(n+a)^p(1-{s}\delta/2)^{n/q_1}\le C_p a^p.\]
In the last inequality we have used the fact that $a\ge1$.
Using, \eqref{beta-bound},
the definition of the times $S_k$, and the Markov property, one has
\begin{align*}
&E_0^\w[S_\ell^p\one\{\beta\circ\theta^{S_k}<\infty
\text{ for }0\le k<\ell\text{ and }
\beta\circ\theta^{S_\ell}=\infty\}]\\
&\qquad\le(\ell+1)^{p-1}\Big(E_0^\w[\lambda_a^p
\one\{\beta\circ\theta^{S_k}<\infty
\text{ for }0\le k<\ell\}]\\
&\qquad\qquad\qquad\qquad\qquad+
\sum_{j=0}^{\ell-1} E_0^\w[\lambda_{M_\beta+a}^p\circ\theta^{S_j}
\one\{\beta\circ\theta^{S_k}<\infty
\text{ for }0\le k<\ell\}]\Big)\\
&\qquad\le (\ell+1)^{p-1}\Big(C_p a^p(1-\eta)^\ell
+\sum_{j=0}^{\ell-1}(1-\eta)^j C_p a^p (1-\eta)^{\ell-j-1}\Big)\\
&\qquad\le C_p (\ell+1)^p(1-\eta)^{\ell-1}a^p.
\end{align*}
%\[E_0^\w(|\tau_1^{(a)}|^p)
%\le C_pa^p+C_p\sum_{\ell\ge1}(\ell+1)(1-\eta)^{\ell-1}a^p.\]
Bound \eqref{tau-bound} follows then from \eqref{tau}. 
To prove \eqref{X-tau-bound} let $q>p$ and write
\begin{align*}
E_0^\w(|X_{\tau^{(a)}_1+n}-X_n|^p)&\le
\sum_{k\ge0}E_0^\w(|X_{k+1+n}-X_{k+n}|^p |\tau_1^{(a)}|^{p-1} 
\one\{k<\tau_1^{(a)}\})\\
&\le\sum_{k\ge0}k^{-{1-q+p}}E_0^\w(|X_{k+1+n}-X_{k+n}|^p
|\tau_1^{(a)}|^q)\\
&\le C_p\sum_{k\ge0}k^{-{1-q+p}}E_0^\w(|\tau_1^{(a)}|^{2q})^{1/2}
\le C_q\, a^q,
\end{align*}
where we have used Hypothesis (M) along with the Cauchy-Schwarz inequality
in the second to last inequality and \eqref{tau-bound} in the last.
%Alternatively, for $n=0$, one can use
%\eqref{X-si-bound} instead of \eqref{si-bound} and
%\begin{align*}
%X_{\tau_1^{(a)}}=\sum_{\ell\ge0}X_{S_\ell}
%\one\{\beta\circ\theta^{S_k}<\infty,
%\text{ for }0\le k<\ell,
%\beta\circ\theta^{S_\ell}=\infty\}.
%\end{align*}
%instead of \eqref{tau}.
This completes the proof of the lemma.
\end{proof}

\section{Invariant measure and ergodicity}
\label{Pinfty_exists}
For $\ell\in\Z$ define the $\sigma$-algebras 
${\kS}_\ell=\sigma\{\w_x: x\cdot \uhat\geq \ell\}$ on $\Omega$. 
 Denote the restriction of the measure $\P$ to the
 $\sigma$-algebra ${\kS}_\ell$ by $\P_{\vert{\kS}_\ell}$. 
In this section we prove the next two theorems.  The 
variation distance of two probability measures is
$\vard(\mu,\nu)=\sup\{\mu(A)-\nu(A)\}$ with the supremum taken
over measurable sets $A$. 

\begin{theorem}
\label{Th_exist}
Assume $\P$ is product 
non-nestling {\rm(N)} and satisfies the moment hypothesis {\rm(M)}.
Then 
there exists a probability measure $\P_\infty$ on $\Omega$ with 
these properties.  
\begin{enumerate}
\item[{\rm (a)}] $\P_\infty$ is invariant and ergodic
 for the Markov  transition kernel $\Pi$.
\item[{\rm (b)}] There exist 
constants $0<c,C<\infty$ 
%$c=c(M,\delta)>0$ and $C=C(M,\delta)<\infty$ 
such that  for all $\ell\ge 0$ 
\begin{align}
\label{vard-exp}
\vard({\P_\infty}_{\vert{\kS}_\ell},\P_{\vert{\kS}_\ell})\le Ce^{-c\ell}.
\end{align}
\item[{\rm (c)}]  Hypotheses {\rm (N)} and {\rm (M)}
and the  
 conclusions of Lemma {\rm \ref{exponential}} hold
 $\P_\infty$-almost surely. 
\end{enumerate}
\end{theorem}

Along the way we also establish this ergodic theorem under the 
original environment measure.  $\E_\infty$
denotes expectation under $\P_\infty$. 

\begin{theorem} 
\label{erg-thm}  Assumptions as in Theorem {\rm\ref{Th_exist}} above. 
Let $\Psi$ be a bounded $\kS_{-a}$-measurable function 
on $\Omega$, for some $0<a<\infty$.  Then 
\be
\lim_{n\to\infty} 
 n^{-1}\sum_{j=0}^{n-1} \Psi(T_{X_j}\w) 
 =  \E_\infty \Psi
\quad\text{$P_0$-almost surely.}
\label{erg-P}\ee
\end{theorem} 

The ergodic theorem tells us that there is a unique invariant $\P_\infty$ 
in a natural relationship to $\P$, and that 
 $\P_\infty\ll\P$ on each
$\sigma$-algebra  $\kS_{-a}$.  Limit  \eqref{erg-P} cannot hold
for all bounded measurable $\Psi$  on $\Omega$ because this would 
imply the absolute continuity  $\P_\infty\ll\P$ on the entire space
$\Omega$.  A  counterexample that satisfies (N) and (M)
but where the quenched walk is degenerate
was given by Bolthausen and Sznitman \cite[Proposition 1.5]{dynstat}.
Whether regularity  assumption (R) or ellipticity
will make a difference here is not presently clear. 
For the simpler case of  space-time walks (see description of
model  in \cite{qclt-spacetime}) 
 with nondegenerate
$P^\w_0$ absolute continuity $\P_\infty\ll\P$ does hold on the 
entire space. Theorem 3.1 in \cite{dynstat} proves this 
for nearest-neighbor jumps with some weak ellipticity. 
 The general case is no harder. 

\begin{proof}[Proof of Theorems \ref{Th_exist} and \ref{erg-thm}] 
Let  $\P_n(A)=P_0(T_{X_n}\w\in A)$. A computation shows
that \[f_n(\w)=\frac{d\P_n}{d\P}(\w)=\sum_x P_x^\w(X_n=0).\]

By hypotheses (M) and  (N) 
  we can replace the state space $\Omega
=\mathcal P^{\Z^d}$ 
with  the smaller space 
$\Omega_0={\mathcal P}_{0}^{\Z^d}$ where 
\be {\mathcal P}_{0}=
\{(p_z)\in\mathcal P: 
\text{$\sum_z e^{{s}_0|z|}p_z\le e^{{s}_0 M}$ and    
  $\sum_{z} z\cdot\uhat\,p_z\geq\delta$ } \}.
\label{defP0}\ee
%${\mathcal P}_{0}$  is compact under the $\ell^1$ metric.  
Fatou's lemma shows that the exponential bound is preserved by
pointwise convergence in $\cP_0$. 
%which is a consequence of $\ell^1$-convergence in $\cP_0$.
Then the exponential bound shows that
the non-nestling property is also preserved.  
Thus ${\mathcal P}_{0}$  is compact, and 
then  $\Omega_0$ is compact  under the  product topology.

Compactness gives  a subsequence $\{n_j\}$ along which the averages
${n_j}^{-1}\sum_{m=1}^{n_j}\P_m$ converge weakly to a probability measure
$\P_\infty$ on $\Omega_0$.   Hypotheses (N) and (M)  transfer to 
$\P_\infty$ by virtue of having been included in the 
state space $\Omega_0$. 
Thus the proof of Lemma \ref{exponential} can be repeated for 
$\P_\infty$-a.e. $\w$.  We have verified part (c) of Theorem \ref{Th_exist}. 

Next 
we check that $\P_\infty$ is invariant under $\Pi$. 
Take a  bounded,  continuous  local function $F$ on $\Omega_0$
that  depends only on environments
$(\w_x:|x|\le K)$.  For $\w,\bar\w\in\Omega_0$ 
\begin{align*}
&\bigl\lvert \Pi F(\w)-\Pi F(\bar\w) \bigr\rvert  = 
\bigl\lvert 
E_0^\w[F(T_{X_1}\w)]-E_0^{\bar\w}[F(T_{X_1}\bar\w)]\bigr\rvert \\
&\quad\le 
 \sum_{\abs{z}\le C} \Bigl\lvert \pi_{0,z}(\w)F(T_z\w)
- \pi_{0,z}(\bar\w)F(T_z\bar\w) \Bigr\rvert
+ \norm{F}_\infty 
\sum_{\abs{z}> C} \bigl(\pi_{0,z}(\w)+\pi_{0,z}(\bar\w)\bigr). 
\end{align*}
From this we see that $\Pi F$ is continuous.  
For let $\bar\w\to\w$ in $\Omega_0$ so that  
%Then $\bar\w_x\to\w_x$ in $\ell^1$ for each $x$, and consequently also 
$\bar\w_{x,z}\to\w_{x,z}$ 
 at each coordinate. Since the last term above is controlled by
the uniform exponential tail bound imposed on $\mathcal P_0$,
continuity of 
$\Pi F$ follows.  
Consequently the weak limit 
${n_j}^{-1}\sum_{m=1}^{n_j}\P_m\to\P_\infty$   together with 
$\P_{n+1}=\P_n\Pi$ implies the 
 $\Pi$-invariance of $\P_\infty$.

We show the exponential bound \eqref{vard-exp} on the variation 
distance next because the ergodicity proof depends on it.
 On metric spaces  total variation distance
can be  characterized  in terms of continuous functions: 
\[\vard(\mu,\nu)=\frac12
\sup\Big\{\int f d\mu-\int f d\nu:f\text{ continuous},\ \sup|f|\le1\Big\}.\]  This makes $\vard(\mu,\nu)$ lower semicontinuous
which we shall find convenient below. 

Fix $\ell>0$. Then 
\be
\begin{split}
\frac{d{\P_n}_{|{\kS}_\ell}}{d{\P}_{|{\kS}_\ell}}&=
%\sum_x\E[P_x^\w(X_n=0)|{\kS}_\ell]\\
\E\big[\sum_x P_x^\w(X_n=0,\max_{j\leq n} X_j\cdot\uhat\leq \ell/2)
\big\vert {\kS}_\ell\big]\\
&\qquad
+\sum_x\E[P_x^\w(X_n=0,\max_{j\leq n} X_j\cdot\uhat>\ell/2)|{\kS}_\ell].
\end{split} \label{RNder}\ee
The $L^1(\P)$-norm of the second term is bounded by 
\[I_{n,\ell}=
P_0(\max_{j\leq n} X_j\cdot\uhat>X_n\cdot\uhat+\ell/2)\]
and \eqref{exp-bound} tells us that
\begin{align}
I_{n,\ell}\leq\sum_{j=0}^n
e^{-{s} \ell/2}(1-{s}\delta/2)^{n-j}\leq Ce^{-{s} \ell/2}.
\label{Ink}
\end{align}
The integrand in the first term of \eqref{RNder}
 is measurable with
respect to $\sigma(\w_x:x\cdot\uhat\le \ell/2)$ and therefore independent of 
$\kS_{\ell}$. The distance between the whole first term and 1 is then 
$\Ord(I_{n,\ell})$.   Thus for  large enough $\ell$,
\begin{align*}
\vard({\P_n}_{|{\kS}_\ell},{\P}_{|{\kS}_\ell})\leq
\int\Bigl\lvert\frac{d{\P_n}_{|{\kS}_\ell}}{d{\P}_{|{\kS}_\ell}}-1\Bigr\rvert d\P
\leq 2I_{n,\ell}\leq Ce^{-c \ell}.
\end{align*}

By the construction of $\P_\infty$ as the Ces\`aro limit  
and by the lower semicontinuity and convexity 
 of the variation distance 
\[\vard({\P_\infty}_{|{\kS}_\ell},{\P}_{|{\kS}_\ell})\leq
\linf_{j\rightarrow\infty} 
n_j^{-1}\sum_{m=1}^{n_j}\vard({\P_m}_{|{\kS}_\ell},{\P}_{|{\kS}_\ell})
\leq Ce^{-c\ell}.
\]
Part (b) has been verified. 

As the last point 
we prove the ergodicity. 
  Recall  the notation $E^\infty_0=\E_\infty E^\w_0$.  
Let $\Psi$ be a bounded local function on $\Omega$. It suffices
to prove that for some constant $\conb$  
\be
\lim_{n\to\infty} 
E^\infty_0 \Bigl\lvert n^{-1}\sum_{j=0}^{n-1} \Psi(T_{X_j}\w) 
-\conb\,\Bigr\rvert = 0.
\label{ergaux1}
\ee
By an approximation it follows from this that for all $F\in L^1(\P_\infty)$
\be
n^{-1}\sum_{j=0}^{n-1} \Pi^jF(\w)  \rightarrow \E_\infty F 
\quad\text{in $L^1(\P_\infty)$.}
\label{erqaux2}
\ee
  By standard theory (Section IV.2 in \cite{rosenblatt})
this is equivalent to ergodicity of $\P_\infty$ for the transition $\Pi$. 

We combine the proof of Theorem \ref{erg-thm} with the proof of
 \eqref{ergaux1}.  For this purpose let $\Psi$ be $\kS_{-a+1}$-measurable
with $a<\infty$.   Take $a$ to be the parameter in 
the regeneration times \eqref{tau}.    Let 
\[
\varphi_i=\sum_{j=\tau_{i}}^{\tau_{i+1}-1} \Psi(T_{X_j}\w).
\]
From the  i.i.d.\ regeneration slabs and the moment 
bound \eqref{tau-bound2}  follows the limit 
\be
\lim_{m\to\infty} 
 m^{-1}\sum_{j=0}^{\tau_m-1}   \Psi(T_{X_j}\w) 
=
\lim_{m\to\infty} 
 m^{-1}\sum_{i=0}^{m-1} \varphi_i 
=\cona \qquad\text{$P_0$-almost surely,}
\label{ergaux3} \ee
where the constant $\cona$ is defined by the limit.  

To justify this more precisely,  recall the 
 definition of regeneration slabs given in \eqref{regenslab}.
  Define  a function  $\Phi$
of the regeneration slabs  by
\[
\Phi(\cS_0,\cS_1,\cS_2,\dotsc)=\sum_{j=\tau_1}^{\tau_{2}-1} 
\Psi(T_{X_{j}}\w).
%\sum_{j=0}^{\tau_{2}-\tau_{1}-1} \Psi(T_{X_{\tau_1+j}-X_{\tau_1}}(T_{X_{\tau_1}}\w)).
\]
 Since each regeneration slab has thickness in $\uhat$-direction
 at least $a$, the $\Psi$-terms in the sum do not read the environments
below level zero and consequently the sum is a function of 
$(\cS_0,\cS_1,\cS_2,\dotsc)$.
Next one can check for $k\ge 1$ that 
\[
\Phi(\cS_{k-1},\cS_k,\cS_{k+1},\dotsc) =
\sum_{j=\tau_1(X_{\tau_{k-1}+\,\cdot\,}-X_{\tau_{k-1}})}^{\tau_{2}
(X_{\tau_{k-1}+\,\cdot\,}-X_{\tau_{k-1}})-1} 
\Psi\bigl(T_{X_{\tau_{k-1}+j}-X_{\tau_{k-1}}}(T_{X_{\tau_{k-1}}}\w)\bigr)
=\varphi_k. 
\]
%%Now $\varphi_k=\Phi(\cS_{k-1},\cS_k,\cS_{k+1})$ and 
Now the sum
of $\varphi$-terms in \eqref{ergaux3} can be decomposed into
\[
\varphi_0+\varphi_1+ 
\sum_{k=1}^{m-2}\Phi(\cS_{k},\cS_{k+1},\cS_{k+2},\dotsc). 
\]
The limit \eqref{ergaux3} follows because the slabs  $(\cS_k)_{k\ge 1}$ are 
i.i.d.\ and the finite initial terms 
 $\varphi_0+\varphi_1$ are eliminated
 by the $m^{-1}$ factor. 

Let $\alpha_n=\inf\{k: \tau_k\ge n\}$. 
 Bounds \eqref{beta-bound}--\eqref{tau-bound} give finite moments
of all orders  to the increments
 $\tau_k-\tau_{k-1}$ and this 
 implies that $n^{-1}(\tau_{\alpha_n-1}-\tau_{\alpha_n})
\to 0$ $P_0$-almost surely. Consequently \eqref{ergaux3} 
yields the next limit,   for another
constant $\conb$: 
\be
\lim_{n\to\infty} 
 n^{-1}\sum_{j=0}^{n-1}   \Psi(T_{X_j}\w) 
= \conb \qquad\text{$P_0$-almost surely.}
\label{ergaux5}
\ee
By boundedness this limit 
is valid also in  $L^1(P_0)$ and the initial
point of the walk is immaterial by shift-invariance of $\P$.
Let $\ell>0$ and  choose a small 
$\e_0>0$.    Abbreviate 
\[
G_{n,x}(\w)=E^\w_x\Bigl[\;
\Bigl\lvert n^{-1}\sum_{j=0}^{n-1} \Psi(T_{X_j}\w) 
-\conb\,\Bigr\rvert \one\bigl\{\,\inf_{j\ge 0}X_j\cdot\uhat\ge 
X_0\cdot\uhat-\e_0\ell/2\bigr\}\,\Bigr].
\]
  Let \[\Ic=\{x\in\Z^d: x\cdot\uhat\ge \e_0 \ell,
\abs{x}\le A\ell\}\] for some constant $A$. 
Use the bound \eqref{vard-exp} on the variation
distance and the fact that  the functions $G_{n,x}(\w)$ are 
uniformly bounded over all $x,n,\w$, and, if $\ell$ is large enough
relative to $a$ and $\e_0$,   for $x\in\Ic$ 
the function $G_{n,x}$ is $\kS_{\e_0\ell/3}$-measurable.
\begin{align*}
&\P_\infty\Bigl\{ \;\sum_{x\in \Ic}P^\w_0[X_\ell=x] G_{n,x}(\w)\ge \e_1
\Bigr\}
\le \sum_{x\in \Ic} \P_\infty\{ G_{n,x}(\w)\ge \e_1/(C\ell^d)\} \\
&\le  C\ell^d \e_1^{-1}  \sum_{x\in \Ic} \E_\infty G_{n,x} 
\le  C\ell^d \e_1^{-1}  \sum_{x\in \Ic} \E G_{n,x} 
+  C\ell^{2d} \e_1^{-1} e^{-c\e_0\ell/3}.
\end{align*}
By \eqref{ergaux5} $\E G_{n,x}\to 0$ for any fixed $x$.
Thus from above we get for any fixed $\ell$, 
\be
\lim_{n\to\infty} E^\infty_0\bigl[ \,\one\{X_\ell\in\Ic\}
G_{n,X_\ell} \bigr] \le \e_1 + C\ell^{2d} \e_1^{-1} e^{-c\e_0\ell/3}.
\label{ergaux8} \ee
The reader should bear in mind that the constant $C$ is 
changing from line to line.  Finally, we write 
\begin{align*}
&\varlimsup_{n\to\infty} 
E^\infty_0 \Bigl\lvert n^{-1}\sum_{j=0}^{n-1} \Psi(T_{X_j}\w) 
-\conb\,\Bigr\rvert\\
&\le \varlimsup_{n\to\infty} 
E^\infty_0 \Bigl[ \one\{X_\ell\in\Ic\}\,
\Bigl\lvert n^{-1}\sum_{j=\ell}^{n+\ell-1} \Psi(T_{X_j}\w) 
-\conb\,\Bigr\rvert \one\bigl\{\inf_{j\ge \ell}X_j\cdot\uhat\ge 
X_\ell\cdot\uhat-\e_0\ell/2\bigr\} \,\Bigr]\\
&\quad\qquad  +\; C P^\infty_0 \{X_\ell\notin\Ic\}
\;+\; C P^\infty_0 \bigl\{\inf_{j\ge \ell}X_j\cdot\uhat<
X_\ell\cdot\uhat-\e_0\ell/2\bigr\}\\
&\le \varlimsup_{n\to\infty} E^\infty_0\bigl[ \,\one\{X_\ell\in\Ic\}
G_{n,X_\ell} \bigr] \;+\;
C P^\infty_0 \{X_\ell\cdot\uhat<\e_0\ell\}\\
&\quad\qquad  +\;C P^\infty_0 \{\,\lvert X_\ell\rvert >A\ell\}
\;+\;
C E^\infty_0 P^\w_{X_\ell} 
\bigl\{\inf_{j\ge0}X_j\cdot\uhat< X_0\cdot\uhat-\e_0\ell/2\bigr\}.
\end{align*} 
As pointed out, $\P_\infty$ satisfies Lemma \ref{exponential}
 because hypotheses
(N) and (M) were built into the space $\Omega_0$ that supports
$\P_\infty$. This enables us to make the error probabilities
above small.   Consequently, if we first pick $\e_0$ and $\e_1$ 
small enough, $A$ large enough, then $\ell$ large, and apply \eqref{ergaux8},
we will have shown \eqref{ergaux1}.  Ergodicity of $\P_\infty$
has been shown.  
This concludes the proof of Theorem \ref{Th_exist}. 

Thereom \ref{erg-thm} has also been established. It follows from
the combination of \eqref{ergaux1} and \eqref{ergaux5}.
\end{proof}

\section{Change of measure}
\label{substitution}
There are several stages in the proof where we need to 
check that a desired conclusion is not affected by
choice between  $\P$ and  $\P_\infty$.  We collect 
all instances of such transfers in this section. 
The standing assumptions of this section are that 
 $\P$ is an i.i.d.\ product measure that   satisfies
 Hypotheses {\rm(N)} and {\rm(M)}, and that
 $\P_\infty$ is the measure given by Theorem {\rm \ref{Th_exist}}.
We  show first that  $\P_\infty$ can be 
 replaced with $\P$  in the key condition \eqref{cond} of
Theorem \ref{RS}.

\begin{lemma}
\label{velocity}
The velocity $v$ defined by {\rm\eqref{defv}} satisfies
$v=\E_\infty(D)$.  
There exists a constant C such that
\begin{align}
\label{velocity-bound}
|E_0(X_n)-n\E_\infty(D)|\le C \qquad\text{for all $n\ge 1$.} 
\end{align}
\end{lemma}

\begin{proof} We start by showing $v=\E_\infty(D)$.  
The uniform exponential tail in the definition
\eqref{defP0} of $\cP_0$ makes the function $D(\w)$ bounded and continuous on
$\Omega_0$.  By the Ces\`aro definition of $\P_\infty$,
\[
\E_\infty(D) = \lim_{j\to\infty} \frac1{n_j}\sum_{k=0}^{n_j-1} \E_k(D)
 = \lim_{j\to\infty} \frac1{n_j}\sum_{k=0}^{n_j-1} E_0[D(T_{X_k}\w)].
\]
The moment bounds \eqref{beta-bound}--\eqref{X-tau-bound} imply
that the law of large numbers $n^{-1}X_n\to v$ holds also in $L^1(P_0)$.
From this and the Markov property 
\[
v=\lim_{n\to\infty} \frac1n \sum_{k=0}^{n-1} E_0(X_{k+1}-X_{k})
=\lim_{n\to\infty} \frac1n \sum_{k=0}^{n-1} E_0[D(T_{X_k}\w)].
\]
We have proved  $v=\E_\infty(D)$.  

 The variables 
$(X_{\tau_{j+1}}-X_{\tau_{j}}, \tau_{j+1}-\tau_{j})_{j\ge 1}$ are
i.i.d.  with sufficient moments by \eqref{beta-bound}--\eqref{X-tau-bound}.
With $\alpha_n= \inf\{j\ge 1: \tau_{j}-\tau_1\ge n\}$ Wald's identity gives
\[
E_0 (X_{\tau_{\alpha_n}}-X_{\tau_1})=E_0(\alpha_n)
E_0(X_{\tau_1}\vert\beta=\infty) 
\quad\text{and}\quad
E_0 ( \tau_{\alpha_n}-\tau_1)
=E_0(\alpha_n) E_0({\tau_1}\vert\beta=\infty).
\] 
Consequently, by the definition \eqref{defv} of $v$,
\begin{align*}
E_0(X_n) - nv = 
v E_0(\tau_{\alpha_n}-\tau_1-n)- 
E_0 (X_{\tau_{\alpha_n}}-X_{\tau_1}-X_{n}).
\end{align*}

It remains to show that   $E_0(\tau_{\alpha_n}-\tau_1-n)$  and
$E_0(X_{\tau_{\alpha_n}}-X_{\tau_1}-X_n)$ are bounded by constants. 
We do this with a simple renewal argument.
Let $Y_j=\tau_{j+1}-\tau_{j}$ for $j\ge 1$ and $V_0=0$, 
 $V_m=Y_1+\dotsm+Y_m$.  
The quantity to bound is  the forward recurrence
time  
$B_n=\min\{k\ge 0: n+k\in \{V_m\}\}$ because 
$\tau_{\alpha_n}-\tau_1-n=B_n$. 

We can write
\[
B_n = (Y_1-n)^+  +   \sum_{k=1}^{n-1} \one\{Y_1=k\} B_{n-k}\circ\theta
\]
where $\theta$ shifts the sequence $\{Y_k\}$ 
and makes  $B_{n-k}\circ\theta$  independent of $Y_1$. 
The two main terms on the right multiply to zero,  so for any integer 
$p\ge 1$
\[
B_n^p = ((Y_1-n)^+)^p  +
          \sum_{k=1}^{n-1} \one\{Y_1=k\} (B_{n-k}\circ\theta)^p.
\]
Set $z(n) = E_0((Y_1-n)^+)^p$.  Moment bounds
 \eqref{beta-bound}--\eqref{tau-bound} give 
  $E_0(Y_1^{p+1})<\infty$ which 
implies $\sum z(n) < \infty$.  Taking expectations and using independence
gives the  discrete renewal equation 
\[
E_0B_n^p = z(n)  +   \sum_{k=1}^{n-1} P_0(Y_1=k) E_0B_{n-k}^p.
\]
Induction on $n$ shows  that   $E_0B_n^p \le \sum_{k=1}^n z(k) \le C(p)$
for all $n$.
In particular,   $E_0(\tau_{\alpha_n}-\tau_1-n)^p$ is bounded  by a constant 
uniformly over $n$.
To extend this to $E_0\lvert X_{\tau_{\alpha_n}}-X_{\tau_1}-X_n\rvert^p$ 
apply an argument 
like the one given for \eqref{X-tau-bound} at the end of Section 
\ref{nonnestling}. 
\end{proof}

\begin{proposition}
\label{exchange}
Assume that there exists an $\bar\alpha<1/2$ such that
\begin{align}
\E\left(\abs{E_0^\w(X_n)-E_0(X_n)}^2\right)=\Ord(n^{2\bar\alpha}).
\label{cond1}
\end{align}
Then condition \eqref{cond} is satisfied for some $\alpha<1/2$.
\end{proposition}

\begin{proof}
By \eqref{velocity-bound} assumption \eqref{cond1} 
turns into 
\begin{align}
\label{cond1-modified}
\E\left(\abs{E_0^\w(X_n)-nv}^2\right)=\Ord(n^{2\bar\alpha}).
\end{align}
In the rest of this proof we use the conclusions of Lemma \ref{exponential}
under $\P_\infty$ instead of $\P$. This is justified by
part (c) of Theorem \ref{Th_exist}. 

For $k\geq1$, recall that 
$\lambda_k=\inf\{n\geq0:(X_n-X_0)\cdot\uhat\geq k\}$.
Take $k=[n^\rho]$ for a small enough $\rho>0$.   The point of the
proof is to let the walk run up to a high level $k$ so 
that expectations under $\P_\infty$ can be profitably related to 
expectations under $\P$ through the variation distance bound
\eqref{vard-exp}.   Estimation is needed to remove the
dependence on the environment on low levels.  
 First 
compute as follows.  
\begin{align}
&\E_\infty\left[\abs{E_0^\w(X_n-nv)}^2\right]
=\E_\infty\left[\abs{E_0^\w(X_n-nv,\lambda_k\leq n)+
E_0^\w(X_n-nv,\lambda_k>n)}^2\right]
\nn\\
\begin{split}
&\!\leq2\E_\infty\!\left[\,\abs{E_0^\w(X_n-X_{\lambda_k}-(n-\lambda_k)v,\lambda_k\leq n)
-E_0^\w(\lambda_k v,\lambda_k\leq n)+E_0^\w(X_{\lambda_k},\lambda_k\le n)^2}
\,\right]
\\
&\qquad+\Ord\bigl(n^2\E_\infty[P_0^\w(\lambda_k>n)]\bigr)
\end{split}
\nn\\
\begin{split}
&\!\leq8\E_\infty\biggl[\;\Bigl\lvert\;
\sum_{\substack{0\leq m\leq n\\ x\cdot\uhat\geq k}}
P_0^\w(X_m=x,\lambda_k=m)E_x^\w\{X_{n-m}-x-(n-m)v\}
\Bigr\rvert^2\,\biggr]
\\
&\qquad+\Ord(k^2+n^2e^{{s} k}(1-{s}\delta/2)^n).
\end{split}
\label{part1}
\end{align}
The last error term above is $\Ord(n^{2\rho})$. 
We  used the Cauchy-Schwarz inequality and Hypothesis (M) to get the second
term in the first inequality, and then \eqref{exp-bound},
\eqref{si-bound}, and \eqref{X-si-bound} in 
the last inequality. 

To handle the expectation on line  \eqref{part1}  we introduce a spanning
set of vectors that satisfy the main assumptions that $\uhat$ does. 
Namely, let  $\{\uhat_i\}_{i=1}^d$ span $\R^d$ and satisfy 
these conditions: 
 $|\uhat-\uhat_i|\le \delta/(2M)$, where $\delta$ and $M$ are
the constants from Hypotheses (N) and (M), and  
\begin{align}
\label{vectors}
\uhat=\sum_{i=1}^d\alpha_i\uhat_i \ \text{ with }\alpha_i>0.
\end{align}
Then non-nestling (N) holds for each $\uhat_i$ with constant $\delta/2$,
and all the conclusions of Lemma \ref{exponential}
hold when $\uhat$ is replaced by $\uhat_i$ and $\delta$ by $\delta/2$.
 Define the event  $A_k=\{\inf_i X_i\cdot\uhat\ge k\}$ and the set
\begin{align*}
\Lambda=\{x\in\Z^d:\min_i x\cdot\uhat_i\ge1\}.
\end{align*}
The point of introducing  $\Lambda$ is that the number of points $x$  
in $\Lambda$ on level $x\cdot\uhat=\ell>0$ is of order $\Ord(\ell^{d-1})$.  

By  Jensen's inequality
the expectation on line  \eqref{part1} is bounded by
\begin{align}
&2\E_\infty\biggl[\;\sum_{\substack{x\in\Lambda,\,x\cdot\uhat\ge k
\\0\le m\le n}}
P_0^\w(X_m=x,\lambda_k=m)\bigl|E_x^\w\{X_{n-m}-x-(n-m)v,A_{k/2}\}\bigr|^2\biggr]
\label{main-term}\\
&\quad+\E_\infty\biggl[\;\sum_{\substack{0\leq m\leq n\\ x\not\in\Lambda}}
P_0^\w(X_m=x,\lambda_k=m)\bigl|E_x^\w\{X_{n-m}-x-(n-m)v\}\bigr|^2\biggr]\nn\\
&\quad+2\E_\infty\biggl[\;\sum_{\substack{0\leq m\leq n\\ x\cdot\uhat\geq k}}
P_0^\w(X_m=x,\lambda_k=m)\bigl|E_x^\w\{X_{n-m}-x-(n-m)v,A_{k/2}^c\}\bigr|^2\biggr].\nn
\end{align}

By Cauchy-Schwarz, 
Hypothesis (M) and \eqref{exp-bound}, the third term is
$\Ord(n^2 e^{-{s} k/2})=\Ord(1)$. 
The second term is of order
\begin{align*}
n^2\max_i \E_\infty[P_0^\w(X_{\lambda_k}\cdot\uhat_i<1)]&\le
n^2\max_i\sum_{m\ge1}\E_\infty\Big[
\bigl(P_0^\w(X_m\cdot\uhat_i<1)P_0^\w(X_m\cdot\uhat\ge k)\bigr)^{1/2}
\,\Big]\\
&\le e^{{s}/2}n^2\sum_{m\ge1}(1-{s}\delta/4)^{m/2}e^{-\mu k/2}
e^{\mu M|\uhat|m/2}\\
&=\Ord(n^2 e^{-\mu k/2})=\Ord(1),
\end{align*}
for $\mu$ small enough.  It remains to bound the term
on line \eqref{main-term}.
To this end, by Cauchy-Schwarz, 
\eqref{bounded-step} and \eqref{exp-bound},
\[P_0^\w(X_m=x,\lambda_k=m)\le 
\{e^{-{s} x\cdot\uhat/2}e^{{s} M|\uhat|m/2}\wedge1\}\times
\{e^{\mu k/2}(1-\mu\delta/2)^{(m-1)/2}\wedge1\} 
\equiv p_{x,m,k}.\]
Notice that 
\[\sum_{x\in\Lambda}\{e^{-{s} x\cdot\uhat/2}e^{{s} M|\uhat|m/2}\wedge1\}
=\Ord(m^d)\]
and
\[\sum_{m\ge1}m^d\{e^{\mu k/2}(1-\mu\delta/2)^{(m-1)/2}\wedge1\}
=\Ord(k^{d+1}).\]
Substitute these back into line \eqref{main-term} to eliminate
the quenched probability coefficients. 
The quenched expectation  in \eqref{main-term} 
 is $\kS_{k/2}$-measurable.  Consequently 
variation distance bound \eqref{vard-exp} 
allows us to switch back to $\P$ and get this upper bound
for line \eqref{main-term}: 
\[2\sum_{\substack{x\in\Lambda,\,x\cdot\uhat\ge k\\0\le m\le n}}p_{x,m,k}
\E[|E_x^\w\{X_{n-m}-x-(n-m)v,A_{k/2}\}|^2]+\Ord(k^{d+1}n^2 e^{-ck/2}).\]
The error term is again $\Ord(1)$. 

Now insert $A_{k/2}^c$ back inside the quenched
expectation, incurring  another error term of order 
$\Ord(k^{d+1}n^2 e^{-{s} k/2})=\Ord(1)$. 
Using the shift-invariance of $\P$, along with
\eqref{cond1-modified}, and collecting all of the above error terms, 
we get 
\begin{align*}
&\E_\infty\left[\,\abs{E_0^\w(X_n-nv)}^2\,\right]\\
&=\sum_{\substack{x\in\Lambda,\,x\cdot\uhat\ge k\\0\le m\le n}}p_{x,m,k}
\E[\,|E_x^\w\{X_{n-m}-x-(n-m)v\}|^2]+\Ord(n^{2\rho})\\
&=
\Ord(k^{d+1}n^{2\bar\alpha}+n^{2\rho})
= \Ord(n^{\rho(d+1)+2\bar\alpha}). 
\end{align*}
Pick  $\rho>0$  small enough so that $2\alpha=
\rho(d+1)+2\bar\alpha<1$. 
The conclusion \eqref{cond} follows. 
\end{proof}

Once  we have verified the assumptions of Theorem \ref{RS} we have 
the CLT under  $\P_\infty$-almost every $\w$.  But we want the CLT
under  $\P$-almost every $\w$. 
Thus as the final point of this section we prove 
 the transfer of the central limit theorem from $\P_\infty$ to $\P$. 
  This  is where we use the ergodic theorem, Theorem \ref{erg-thm}.
Let $\wiener$ be the probability distribution
 of the Brownian motion with diffusion
matrix $\mathfrak D$. 

\begin{lemma} 
 Suppose the weak convergence $Q^\w_n\Rightarrow W$ 
holds for $\P_\infty$-almost every $\w$.  Then the same is true
 for $\P$-almost every $\w$.  
%Same statement holds for the 
%convergence  $\Qtil_n^\w\Rightarrow W$.  
\label{clt-transfer-lm}
\end{lemma} 

\begin{proof} 
It suffices to show that for any bounded uniformly continuous
$F$ on $D_{\R^d}[0,\infty)$ and any $\delta>0$
\[
\varlimsup_{n\to\infty} E^{\w}_0[F(B_n)]\le \int F\,d\wiener+\delta \quad
\text{$\P$-a.s.} 
\] 
By considering also $-F$ 
this gives $E^{\w}_0[F(B_n)]\to \int F\,d\wiener$ {$P_0$-a.s.} 
for each such function.
A countable collection of them determines weak convergence. 

Fix such an $F$.  Let $c=\int F\,d\wiener$ and  
\[
\overline h(\w)=\lsup_{n\to\infty} E_0^\w[F(B_n)]. 
\]
 For $\ell>0$ 
define the events
\[
A_{-\ell}=\{\inf_{n\ge0} X_n\cdot\uhat\ge-\ell\}\]
and then 
\[
\overline h_\ell(\w)=\lsup_{n\to\infty} E_0^\w[F(B_n), A_{-\ell}]
\quad\text{and}\quad 
\Psi(\w)=\one\{\w\,:\, \bar{h}_\ell(\w)\le c+\tfrac12\delta\}.
\]
The assumed
quenched CLT under $\P_\infty$ gives  $\P_\infty\{\overline h=c\}=1$. 
By  \eqref{exp-bound},  and by its extension to $\P_\infty$ in
Theorem  \ref{Th_exist}(c), there are constants $0<C,s<\infty$ such
that  
\[
\lvert \overline h(\w) - \overline h_\ell(\w) \rvert \le Ce^{-s\ell}
\]
uniformly over all $\w$ that support both $\P$ and $\P_\infty$. 
Consequently  if $\delta>0$ is given, 
 $\E_\infty\Psi=1$ for large enough $\ell$.
Since $\Psi$ is $\kS_{-\ell}$-measurable 
 Theorem \ref{erg-thm} implies that 
\[n^{-1}\sum_{j=1}^n \Psi(T_{X_j}\w)\to 1 \quad \text{ $P_0$-a.s.}\]  
By increasing $\ell$ if necessary we can ensure 
that $\{\bar{h}_\ell\le c+\tfrac12\delta\}\subset\{\bar{h}\le c+\delta\}$
 and conclude that the stopping time
\[
\zeta=\inf\{n\ge 0: \bar{h}(T_{X_n}\w)\le c+\delta\}
\]
is $P_0$-a.s.\ finite. From the definitions we now have 
\[
\varlimsup_{n\to\infty} E^{T_{X_\zeta}\w}_0[F(B_n)]\le 
\int F\, dW+\delta \quad
\text{$P_0$-a.s.} 
\]
Then by bounded convergence 
\[\varlimsup_{n\to\infty}E_0^\w E_0^{T_{X_{\zeta}}\w}[F(B_n)]
\le\int F\, dW+\delta \quad \P\text{-a.s.}\]
Since $\zeta$ is a finite stopping time, the strong Markov property,
 the uniform continuity of $F$ and the exponential moment 
bound \eqref{bounded-step}  on $X$-increments  imply
\[\varlimsup_{n\to\infty}E_0^\w [F(B_n)]
\le\int F\, dW+\delta\quad \P\text{-a.s.}\]
This concludes the proof. 
\end{proof}

\section{Reduction to path intersections}
\label{pathintersections}
The preceding sections have reduced the proof of the
main result Theorem \ref{main}   to proving the estimate
\begin{align}
\E\left(\,\abs{E_0^\w(X_n)-E_0(X_n)}^2\,\right)=\Ord(n^{2\alpha})
\quad\text{for some $\alpha<1/2$.} 
\label{cond1a}
\end{align}
The next reduction takes us to 
 the expected number of intersections of the paths of two independent
walks $X$ and $\Xtil$ 
 in the same environment. 
The argument uses a decomposition into martingale differences
through an  ordering of lattice sites.  
This idea for bounding a variance is natural and has
been used  in RWRE earlier by Bolthausen and Sznitman \cite{dynstat}.

Let $P_{0,0}^\w$ be the quenched
law of the walks $(X,\Xtil)$ started at $(X_0,\Xtil_0)=(0,0)$  
 and $P_{0,0}=\int P_{0,0}^\w\,\P(d\w)$ the averaged
law with expectation operator $E_{0,0}$. The set of sites visited
by a walk is denoted by $X_{[0,n)}=\{X_k: 0\le k<n\}$ and 
 $\abs{A}$ is the number of elements in a discrete set $A$. 

\begin{proposition}
\label{intersections}
Let $\P$ be an i.i.d.\ product measure and satisfy Hypotheses {\rm(N)} and {\rm(M)}.
Assume that there exists an $\bar\alpha<1/2$ such that
\begin{align}
E_{0,0}(|X_{[0,n)}\cap\Xtil_{[0,n)}|)=\Ord(n^{2\bar\alpha}).
\label{cond2}
\end{align}
Then condition \eqref{cond1a} is satisfied. 
\end{proposition}

\def\wwtil{{\tilde{\w}}}
\begin{proof}
For $L\geq0$, define
${\mathcal B}(L)=\{x\in\Z^d:|x|\leq L\}$. 
Fix $n\geq1$, $c>|\uhat|$, and
let $(x_j)_{j\geq1}$ be some fixed ordering of ${\mathcal B}(cMn)$ satisfying
\[\forall i\geq j:x_i\cdot \uhat\geq x_j\cdot \uhat.\]
For $B\subset\Z^d$ let $\kS_B=\sigma\{\w_x: x\in B\}$. 
Let $A_j=\{x_1,\dotsc,x_{j}\}$, $\zeta_0= E_0(X_n)$,  and for $j\geq 1$
\[
\zeta_j=\E(E^\w_0(X_n)|\kS_{A_j}).
%= \int \P(d\w_{A_j^c}) E^\w_0(X_n).
\]
 $(\zeta_j-\zeta_{j-1})_{j\geq1}$ is a sequence of $L^2(\P)$-martingale
differences and we have
\begin{align}
&\E[\,|E^\w_0( X_n) -E_0(X_n)|\,^2]\\
&\qquad\le
2\E\bigl[\,|E_0(X_n)-\E\{E^\w_0(X_n)|\kS_{{\mathcal B}(cMn)}\}|^2\,\bigr]\nn\\
&\qquad\qquad+
2\E\bigl[\,\bigl|E^\w_0(X_n,\max_{i\le n} |X_i|>cMn)\nn\\
&\qquad\qquad\qquad\qquad -
\E\{E^\w_0(X_n,\max_{i\le n} |X_i|>cMn)\,\vert\,
\kS_{{\mathcal B}(cMn)}\}\bigr|^2\,\bigr]\nn\\
&\qquad\le2\sum_{j=1}^{|{\mathcal B}(cMn)|}\E(\,|\zeta_j-\zeta_{j-1}|^2\,)+
\Ord(n^3 e^{-{s} M(c-|\uhat|)n}).
\label{mgale-decomp}
\end{align}
In the last inequality we have used \eqref{bounded-step}.  The 
error is $\Ord(1)$. 
For $z\in\Z^d$ define half-spaces 
\[
{\mathcal H}(z)=\{x\in\Z^d: x\cdot\uhat > z\cdot\uhat\}.
\]
Since $A_{j-1}\subset A_j\subset {\mathcal H}({x_j})^c$, 
%\begin{align*}
%\E(|\zeta_j-\zeta_{j-1}|^2)&=\E(|\E[\E(E_0^\w(X_n)|\kS_{U^c(x_j)})|\kS_{A_j}]
%-\E\{\E[\E(E_0^\w(X_n)|\kS_{U^c(x_j)})|\kS_{A_j}]|\kS_{A_{j-1}}\}|^2)\\
%&=\E(|\E[\E(E_0^\w(X_n)|\kS_{U^c(x_j)})|\kS_{A_j}](\w_{A_j})
%-\int\P^{\w_{A_{j-1}}}_{A_j,\{x_j\}}(d\tilde{\w}_{x_j})
%\E\{\E[\E(E_0^\w(X_n)|\kS_{U^c(x_j)})|\kS_{A_j}]|\kS_{A_{j-1}}\}|^2)\\
%&\le
%\end{align*}
\begin{align}
&\E(|\zeta_j-\zeta_{j-1}|^2)  \nn\\
&\quad =
\int \P(d\w_{A_j})
\Bigl\lvert \iint \P(d\w_{A_j^c})\P(d\tilde{\w}_{x_j})
 \bigl( E^\w_0( X_n)-E^{\pp{\w,\tilde{\w}_{x_j}}}_0( X_n)\bigr)
\Bigr\rvert^2\nn\\
&\quad \leq
\iint \P(d\w_{{\mathcal H}(x_j)^c})\P(d\tilde{\w}_{x_j})
\Bigl\lvert \int \P(d\w_{{\mathcal H}(x_j)}) \bigl( E^\w_0( X_n)-E^{\pp{\w,\tilde{\w}_{x_j}}}_0( X_n)\bigr)
\Bigr\rvert^2.
\label{line-2}
 \end{align}
Above $\pp{\w,\tilde{\w}_{x_j}}$ denotes an environment obtained from $\w$
by replacing  $\w_{x_j}$ with  $\tilde{\w}_{x_j}$. 

We fix a point $z=x_j$ to develop a bound for the expression above,
and then  return to collect the estimates. 
Abbreviate $\wwtil=\pp{\w,\tilde{\w}_{x_j}}$. 
Consider two walks $X_n$ and $\Xtil_n$ starting at $0$. $X_n$ obeys
environment $\w$, while $\Xtil_n$ obeys $\wwtil$. 
We can couple the two walks so that they stay together until 
the first time they visit $z$. Until a visit to $z$ happens, the 
walks are identical. So we write
\begin{align}
&\int \P(d\w_{{\mathcal H}(z)})\bigl( E^\w_0( X_n)-E^\wwtil_0( X_n)\bigr)
\label{line-7}\\
&=
\int \P(d\w_{{\mathcal H}(z)}) \sum_{m=0}^{n-1} P^\w_0( H_z=m)
\bigl( E^\w_z( X_{n-m}-z)-E^\wwtil_z( X_{n-m}-z)  \bigr)\nn\\
\begin{split}
&=\int \P(d\w_{{\mathcal H}(z)}) \sum_{m=0}^{n-1} \sum_{\ell>0}
 P^\w_0( H_z=m, \ell-1\leq\max_{0\leq j\leq m}X_j\cdot\uhat-z\cdot\uhat<\ell  )
\\
&\qquad\qquad \times 
\bigl( E^\w_z( X_{n-m}-z)-E^\wwtil_z( X_{n-m}-z)  \bigr).
\end{split} \label{line-19}
\end{align} 
 Decompose 
${\mathcal H}(z)={\mathcal H}_\ell(z)\cup {\mathcal H}_\ell'(z)$ where
\[
{\mathcal H}_\ell(z)=\{x\in\Z^d: z\cdot\uhat < x\cdot\uhat< z\cdot\uhat+\ell\}
\; \text{and}\; 
{\mathcal H}_\ell'(z)=\{x\in\Z^d:  x\cdot\uhat \geq z\cdot\uhat+\ell\}.
\]
Take a single $(\ell,m)$ term from the sum in
\eqref{line-19} and only
 the expectation $E^\w_z( X_{n-m}-z)$. 
\begin{align}
&\int \P(d\w_{{\mathcal H}(z)}) 
P^\w_0( H_z=m, \ell-1\leq\max_{0\leq j\leq m}X_j\cdot\uhat-z\cdot\uhat<\ell)\nn\\
&\qquad\qquad \times 
E^\w_z( X_{n-m}-z)\nn\\
\begin{split}
&=\int \P(d\w_{{\mathcal H}(z)}) 
P^\w_0( H_z=m, \ell-1\leq\max_{0\leq j\leq m}X_j\cdot\uhat-z\cdot\uhat<\ell)\\
&\qquad\qquad \times 
E^\w_z( X_{\tau^{(\ell)}_1+n-m}-X_{\tau^{(\ell)}_1})
\end{split}  
\label{line-22}  \\
\begin{split}   \quad 
&+\int \P(d\w_{{\mathcal H}(z)}) 
P^\w_0( H_z=m, \ell-1\leq\max_{0\leq j\leq m}X_j\cdot\uhat-z\cdot\uhat<\ell)
\\
&\qquad\qquad \times 
E^\w_z( X_{n-m}-X_{\tau^{(\ell)}_1+n-m}+X_{\tau^{(\ell)}_1}-z)\end{split}
\label{line-23}
\end{align}
The parameter $\ell$ in the 
 regeneration time $\tau^{(\ell)}_1$ of  the walk started at $z$  
ensures that the subsequent walk $X_{\tau^{(\ell)}_1+\,\cdot}$ stays in
${\mathcal H}_\ell'(z)$.  Below we make use of this to get independence 
from the environments
in ${\mathcal H}_\ell'(z)^c$.  
By \eqref{X-tau-bound} 
 the quenched
 expectation in \eqref{line-23} can be bounded by $C_p\ell^p$, for any $p>1$.

Integral \eqref{line-22} is developed further as follows. 
\begin{align}
&\int \P(d\w_{{\mathcal H}(z)}) 
 P^\w_0( H_z=m, \ell-1\leq\max_{0\leq j\leq m}X_j\cdot\uhat-z\cdot\uhat<\ell)\nn\\
&\qquad\qquad \times 
 E^\w_z( X_{\tau^{(\ell)}_1+n-m}-X_{\tau^{(\ell)}_1})\nn\\
&=\int \P(d\w_{{\mathcal H}_\ell(z)}) 
 P^\w_0( H_z=m, \ell-1\leq\max_{0\leq j\leq m}X_j\cdot\uhat-z\cdot\uhat<\ell)\nn\\
&\qquad\qquad \times \int \P(d\w_{{\mathcal H}_\ell'(z)}) 
 E^\w_z( X_{\tau^{(\ell)}_1+n-m}-X_{\tau^{(\ell)}_1})\nn\\
&=\int \P(d\w_{{\mathcal H}_\ell(z)}) 
 P^\w_0( H_z=m, \ell-1\leq\max_{0\leq j\leq m}X_j\cdot\uhat-z\cdot\uhat<\ell)\nn\\
&\qquad\qquad \times 
 E_z( X_{\tau^{(\ell)}_1+n-m}-X_{\tau^{(\ell)}_1}\vert \kS_{{\mathcal H}_\ell'(z)^c})\nn\\
\begin{split}
&=\int \P(d\w_{{\mathcal H}_\ell(z)}) 
 P^\w_0( H_z=m, \ell-1\leq\max_{0\leq j\leq m}X_j\cdot\uhat-z\cdot\uhat<\ell)
%\label{line-30}
\\
&\qquad\qquad \times 
 E_0( X_{n-m}\vert \beta=\infty).\end{split}  \label{line-31}
\end{align}
The last equality above comes from the regeneration structure, see
Proposition 1.3 in Sznitman-Zerner \cite{szlln}.  The $\sigma$-algebra
$\kS_{{\mathcal H}_\ell'(z)^c}$ is contained in the  $\sigma$-algebra
$\cG_1$ defined by (1.22) of \cite{szlln} for the walk starting at $z$. 

The last quantity 
  \eqref{line-31}  above
reads the environment only until the first visit
to $z$, hence does not see the distinction between $\w$ and $\wwtil$.
Hence when the integral
 \eqref{line-19} is developed separately for 
$\w$ and $\wwtil$ into the sum of integrals \eqref{line-22}
and \eqref{line-23},
integrals \eqref{line-22} for $\w$ and $\wwtil$ cancel each other. 
We are left only with two instances
of integral \eqref{line-23}, one for both  $\w$ and $\wwtil$.  
The last quenched expectation in \eqref{line-23} 
 we bound by $C_p\ell^p$ as was mentioned above.

Going back to \eqref{line-7}, we get this bound:
\begin{align*}
&\Bigl\lvert 
\int \P(d\w_{{\mathcal H}(z)})\bigl( E^\w_0( X_n)-E^\wwtil_0( X_n)\bigr)
\Bigr\rvert\nn\\
&\leq C_p
\int \P(d\w_{{\mathcal H}(z)}) \sum_{\ell>0} \ell^p
 P^\w_0( H_z<n, \ell-1\leq\max_{0\leq j\leq H_z}X_j\cdot\uhat-z\cdot\uhat<\ell)
\nn\\
&\le C_p \int \P(d\w_{{\mathcal H}(z)}) 
E_0^\w[(M_{H_z}-z\cdot\uhat)^p\one\{H_z<n\}]\\
&\le C_p n^{p\e} \int \P(d\w_{{\mathcal H}(z)}) 
P_0^\w(H_z<n)+C_p s^{-p}e^{-{s} n^\e/2}. 
\end{align*}
 For the last inequality we used \eqref{Hz-bound} with $\ell=n^\e$ and 
some small  $\e, s>0$. 
Square, take $z=x_j$, integrate as in \eqref{line-2}, and use Jensen's 
inequality to bring the square inside the integral to get 
\begin{align*}
\E(\,|\zeta_j-\zeta_{j-1}|^2\,)\le 2C_p n^{2p\e}\E[\,|P_0^\w(H_{x_j}<n)|^2\,]+
2C_p s^{-2p}e^{-{s} n^\e}.
\end{align*}
Substitute these bounds into line \eqref{mgale-decomp} and note
that the error there is $\Ord(1)$. 
\begin{align*}
&\E[\,|E_0^\w(X_n)-E_0(X_n)|^2\,]\\
&\qquad\qquad\le C_p n^{2p\e}\sum_z\E[\,|P_0^\w(H_z<n)|^2\,]+
\Ord(n^d s^{-2p}e^{-{s} n^\e})+\Ord(1)\\
&\qquad\qquad= C_p n^{2p\e}\sum_z P_{0,0}(z\in X_{[0,n)}\cap\tilde X_{[0,n)})+
\Ord(1)\\
&\qquad\qquad= C_p n^{2p\e}E_{0,0}[\,|X_{[0,n)}\cap\tilde X_{[0,n)}|\,]+
\Ord(1).
\end{align*}
Utilize assumption \eqref{cond2} and take    $\e>0$ small enough 
so that  $2\alpha = 2p\e + 2\bar\alpha<1$. \eqref{cond1a}
has been verified.
\end{proof}

\section{Bound on intersections}
\label{intersectbound}
The remaining piece of the proof of Theorem \ref{main} 
is this estimate: 
\begin{align}
E_{0,0}(\,\lvert {X_{[0,n)}\cap\Xtil_{[0,n)}}\rvert\,)=\Ord(n^{2\alpha})
\quad\text{for some $\alpha<1/2$.} 
\label{cond3}
\end{align}
$X$ and $\Xtil$ are two independent walks in a common 
environment with quenched distribution
$P^\w_{x,y}[X_{0,\infty}\in A, \Xtil_{0,\infty}\in B]
= P^\w_{x}(A)P^\w_{y}(B)$ and averaged distribution 
$E_{x,y}(\cdot)=\E P^\w_{x,y}(\cdot)$.

To deduce the  sublinear bound 
 we introduce regeneration times at which both
 walks regenerate on the same level in space (but not 
necessarily at the same time). 
Intersections happen only within  the  
regeneration slabs, and the
expected number of intersections
decays exponentially in the distance between the points of
entry of the  walks in the slab.
From  regeneration to regeneration the difference of
the two walks operates
like a Markov chain. This Markov chain can be approximated 
by a symmetric random walk.  Via this preliminary work the
required estimate boils down to deriving a Green function 
bound for a Markov chain that can be suitably approximated
by a symmetric random walk.  This part is relegated to an 
appendix.  Except for the appendix, we complete the proof
of the functional central limit theorem 
in this section.    

To aid  our discussion of a pair of walks $(X,\Xtil)$ we introduce 
some new notation.  
We write  $\theta^{m,n}$ for the shift on pairs of paths:
%%$\theta^{m,n}(\mathbf{x},\mathbf{y})=(\theta^{m}\mathbf{x},\theta^{n}\mathbf{y})$.
$\theta^{m,n}(x_{0,\infty},y_{0,\infty})=(\theta^mx_{0,\infty},
\theta^ny_{0,\infty})$.
If we write separate expectations for $X$ and $\Xtil$
under $P^\w_{x,y}$, these are denoted by $E^\w_x$ and $\Etil^\w_y$. 

By a {\sl joint stopping time} we mean pair $(\alpha, \altil)$ 
 that  satisfies $\{\alpha=m,\altil=n\}\in\sigma\{X_{0,m},\Xtil_{0,n}\}$.
Under the distribution $P^\w_{x,y}$ the walks $X$ and $\Xtil$ are 
independent. Consequently if $\alpha\vee\altil<\infty$
$P^\w_{x,y}$-almost surely then  
for any events $A$ and $B$,
\begin{align*}
&P^\w_{x,y}[ (X_{0,\alpha},\Xtil_{0,\altil})\in A,\, 
 (X_{\alpha, \infty},\Xtil_{\altil,\infty})\in B]\\
&\qquad = E^\w_{x,y} \bigl[ \one\{(X_{0,\alpha},\Xtil_{0,\altil})\in A\}
P^\w_{X_{\alpha},\Xtil_{\altil}}\{ 
 (X_{0,\infty},\Xtil_{0,\infty})\in B\} \bigr].
\end{align*}
This type of joint restarting will be used without comment in the sequel. 

For this section it will be convenient to have level stopping times
 and running maxima that are not defined 
 relative to the initial level. 
\[
\gamma_\ell=\inf\{n\ge 0: X_n\cdot\uhat\ge \ell\}
\quad\text{and}\quad 
\gamma^+_\ell=\inf\{n\ge 0: X_n\cdot\uhat> \ell\}.
\]
Since  $\uhat\in\Z^d$, $\gamma^+_\ell$ is simply an abbreviation
for $\gamma_{\ell+1}$.
Let  
$M_n=\sup\{X_i\cdot\uhat:i\le n\}$ 
 be the running maximum. 
 ${\Mtil}_n$,   $\tilde\gamma_\ell$ and   $\tilde\gamma^+_\ell$  
are the corresponding 
quantities for the $\Xtil$ walk.  The first backtracking time
for the $\Xtil$ walk is  $\tilde\beta=\inf\{n\ge 1: \Xtil_n\cdot\uhat
<\Xtil_0\cdot\uhat\}$. 

Define 
\[L=\inf\{\ell>(X_0\cdot\uhat)\wedge(\Xtil_0\cdot\uhat):
X_{\gamma_\ell}\cdot\uhat=\Xtil_{\tilde\gamma_\ell}\cdot\uhat=\ell\}\]
as the first fresh
common level after at least one walk has exceeded its
starting level.
Set $L=\infty$ if there is no such common level. 
When the walks are on a common level, their difference will
lie in the hyperplane
\[
\mathbb{V}_d=\{z\in\Z^d: z\cdot\uhat=0\}.
\]
We start with exponential tail bounds on the time to reach 
the common level.   

\begin{lemma}
There  exist  constants  $0<a_1, a_2,C<\infty$ 
such that, for all $x,y\in\Z^d$,  $m\ge0$ and $\P$-a.e.\ $\w$,
\begin{align}
P_{x,y}^\w(\gamma_L\vee\tilde\gamma_L\ge m)
\le  
Ce^{a_1\lvert y\cdot\uhat-x\cdot\uhat\rvert-a_2m}. 
 \label{gamma_L-bound}
\end{align}
\label{gamlm2}\end{lemma}

For the proof we need a bound on the overshoot. 

\begin{lemma}  There exist constants $0<C,s<\infty$ 
such that, for any level $k$, any $b\ge 1$,   any $x\in\Z^d$
such that $x\cdot\uhat\le k$, and $\P$-a.e.~$\w$,
\be 
P^\w_x[ X_{\gamma_k}\cdot\uhat\ge k+ b]\le  Ce^{-sb}.
\label{overshoot}
\ee
\label{overshootlm}
\end{lemma} 

\begin{proof} 
From \eqref{exp-bound} 
 it follows that for a constant $C$, for any level $\ell$,
any $x\in\Z^d$, and $\P$-a.e.~$\w$,
\be
E^\w_x[\text{number of visits to level $\ell$}]
=\sum_{n=0}^\infty P^\w_x[X_n\cdot\uhat=\ell]\le C.
\label{greenbd-w1}
\ee
(This is certainly clear if $x\cdot\uhat=\ell$. Otherwise wait
until the process first lands on level $\ell$, if ever.) 

From this and the exponential moment hypothesis
 we deduce the required 
 bound on the
overshoots:   for any $k$,
any $x\in\Z^d$ such that $x\cdot\uhat\le k$, and $\P$-a.e.~$\w$, 
\begin{align*}
&P^\w_x[ X_{\gamma_k}\cdot\uhat\ge k+ b]
= \sum_{n=0}^\infty\sum_{z\cdot\uhat <k}
 P^\w_x[ \gamma_k>n,\, X_n=z,\, X_{n+1}\cdot\uhat \ge k+ b]\\
&=\sum_{\ell>0} \sum_{z\cdot\uhat =k-\ell} 
\sum_{n=0}^\infty P^\w_x[\gamma_k>n,\, X_n=z] 
P^\w_z[X_{1}\cdot\uhat \ge k+ b]\\
&\le \sum_{\ell>0} \sum_{z\cdot\uhat =k-\ell} 
\sum_{n=0}^\infty P^\w_x[X_n=z]  Ce^{-s(\ell+b)} \\
&\le Ce^{-sb}  \sum_{\ell>0} e^{-s\ell} \le  Ce^{-sb}.
\qedhere \end{align*}
\end{proof} 

\begin{proof}[Proof of Lemma \ref{gamlm2}]
Consider first $\gamma_L$, and 
let us restrict ourselves to the case where the initial points $x,y$ satisfy
$x\cdot\uhat<y\cdot\uhat$.  

Perform an iterative construction of stopping times $\eta_i$, $\etatil_i$
and levels $\ell(i)$, $\elltil(i)$.  
Let $\eta_0=\etatil_0=0$, $x_0=x$ and $y_0=y$.
$\ell(0)$ and $\elltil(0)$ need not be defined.  Suppose that 
the construction has been done to stage $i-1$ with
$x_{i-1}=X_{\eta_{i-1}}$, $y_{i-1}=\Xtil_{\etatil_{i-1}}$, and 
$x_{i-1}\cdot\uhat<y_{i-1}\cdot\uhat$.  Then set 
\[
\ell(i)= X_{\gamma(y_{i-1}\cdot\uhat)}\cdot\uhat,\,
\elltil(i)= \Xtil_{\gatil(\ell(i))}\cdot\uhat,\,
\eta_i= \gamma(\elltil(i))\,\text{ and }\, 
\etatil_i= \gatil(X_{\eta_i}\cdot\uhat+1).
\]
In words, starting at $(x_{i-1},y_{i-1})$ with $y_{i-1}$ 
above $x_{i-1}$, let $X$ reach the level of $y_{i-1}$ and 
let $\ell(i)$ be the level $X$ lands on; let $\Xtil$ reach the level
$\ell(i)$  and 
let $\elltil(i)$ be the level $\Xtil$ lands on. Now let $X$ try to 
establish a new common level at  $\elltil(i)$ with $\Xtil$:
in other words, follow $X$ until the  time $\eta_i$ it reaches  level  $\elltil(i)$
or above, and stop it there.   Finally, reset the 
situation by letting $\Xtil$ reach  a level strictly above 
the level of $X_{\eta_i}$, and stop it there at time $\etatil_i$. 
The starting locations for the next step are
$x_i=X_{\eta_{i}}$, $y_i=\Xtil_{\etatil_{i}}$ that satisfy 
$x_i\cdot\uhat<y_i\cdot\uhat$.

We show that 
within each  step of the iteration there is a uniform lower bound on the 
probability that a fresh common level was found.  
For this purpose we utilize assumption \eqref{level-ell} in the
 weaker  form
\be \P\{\w: \,  P_0^\w(  X_{\gamma_1}\cdot\uhat=1  )\ge\kappa \,\} =1. 
\label{level-ell-1}\ee
Pick $b$ large enough so that the bound in \eqref{overshoot}
is $<$ $1$. 
For $z,w\in\Z^d$ such that $z\cdot\uhat\ge w\cdot\uhat$ define a 
function 
\begin{align*}
\psi(z,w)&=P^\w_{z,w}[ \text{ $X_{\gamma_k}\cdot\uhat=k$ for each $k\in \{z\cdot\uhat, 
\dotsc, z\cdot\uhat+b\}$, } \\
&\qquad\qquad \Xtil_{\gatil(z\cdot\uhat)}\cdot\uhat -z\cdot\uhat  \le b\, ]  \\
&\ge \kappa^b(1-Ce^{-sb})  \equiv \kappa_2>0.
\end{align*}
The uniform lower bound comes from the independence of 
the walks, from \eqref{overshoot} and 
from iterating assumption \eqref{level-ell-1}. 
By the Markov property 
\begin{align*}
&P^\w_{x_{i-1},y_{i-1}}[ \,X_{\gamma(\elltil(i))}\cdot\uhat = \elltil(i) ]\\
&\ge 
P^\w_{x_{i-1},y_{i-1}}[ \,\Xtil_{\gatil(\ell(i))}\cdot\uhat -\ell(i)  \le b,\,
\text{ $X_{\gamma_k}\cdot\uhat=k$ for each 
$k\in\{\ell(i),\dotsc,\ell(i)+b\}$
} ]\\  
&\ge E^\w_{x_{i-1},y_{i-1}}\bigl[ \psi(X_{\gamma(y_{i-1}\cdot\uhat)}, 
y_{i-1})\bigr] \ge \kappa_2. 
%&\ge E^\w_{x,y} P_{X_{\gamma(y\cdot\uhat)}, y} [
%\Xtil_{\gatil(X_0\cdot\uhat)}\cdot\uhat -X_0\cdot\uhat  \le b, \\
%&\qquad \text{ $X_{\gamma_k}\cdot\uhat=k$ for each $k\in \{X_0\cdot\uhat, 
%\dotsc, X_0\cdot\uhat+b\}$}]\\
%&\ge \kappa_2.
\end{align*}

The first iteration on which    the
 attempt to create a common level at $\elltil(i)$  succeeds is
\[I=\inf\{i\ge1:X_{\gamma_{\elltil(i)}}\cdot\uhat=\elltil(i)\}.\]
Then $\elltil(I)$ is  a new fresh common level and
consequently $L\le \elltil(I)$. This gives
the upper bound 
\[\gamma_L\le\gamma_{\elltil(I)}.\]
We develop an exponential tail bound for $\gamma_{\elltil(I)}$,
still under the assumption  $x\cdot\uhat<y\cdot\uhat$. 

 From the uniform bound above and the Markov property we get
\[
P^\w_{x,y}[ I>i] \le (1-\kappa_2)^i.
\]

  Lemma \ref{overshootlm} gives
 an exponential bound 
\be
P^\w_{x,y}\bigl[ (\Xtil_{\etatil_{i}}-\Xtil_{\etatil_{i-1}})\cdot\uhat
 \ge b\bigr]\le Ce^{-sb} 
\label{Xtil-bd-21}
\ee
because the distance $(\Xtil_{\etatil_{i}}-\Xtil_{\etatil_{i-1}})\cdot\uhat$ is a
sum of four overshoots: 
\begin{align*}
&\bigl(\Xtil_{\etatil_{i}}-\Xtil_{\etatil_{i-1}}\bigr)\cdot\uhat
= \bigl( \Xtil_{\gatil(X_{\eta_i}\cdot\uhat+1)}\cdot\uhat- X_{\eta_i}\cdot\uhat-1\bigr)
+1 +\bigl( X_{\gamma(\elltil(i))}\cdot\uhat - \elltil(i)\bigr) \\
&\quad + 
\bigl( \Xtil_{\gatil(\ell(i))}\cdot\uhat-\ell(i) \bigr) + 
\bigl( X_{\gamma(\Xtil_{\etatil_{i-1}}\cdot\uhat)}\cdot\uhat
-\Xtil_{\etatil_{i-1}}\cdot\uhat\bigr).
\end{align*}
Next, from the exponential tail bound on
$
\bigl(\Xtil_{\etatil_{i}}-\Xtil_{\etatil_{i-1}}\bigr)\cdot\uhat
$
and from 
\[
\elltil(i)\le \Xtil_{\etatil_i}\cdot\uhat 
=
\sum_{j=1}^i \bigl(\Xtil_{\etatil_{j}}-\Xtil_{\etatil_{j-1}}\bigr)
\cdot\uhat +y\cdot\uhat
\]
we get the large deviation estimate 
\[
P^\w_{x,y}[\elltil(i) \ge bi+ y\cdot\uhat]\le e^{-sbi}
\quad\text{for $i\ge 1$ and  $b\ge b_0$,}
\]
for some constants $0<s<\infty$ (small enough) and  $0<b_0<\infty$ 
 (large enough).  Combine
this with the bound above on $I$ to write 
\begin{align*}
P^\w_{x,y}[\elltil(I)\ge a] &\le P^\w_{x,y}[I> i] 
+P^\w_{x,y}[ \elltil(i)\ge a] \\
&\le e^{-si}+e^{sy\cdot\uhat-sa} \le 2e^{sy\cdot\uhat-sa}
\end{align*}
where we assume $a\ge 2b_0+y\cdot\uhat$ and set the integer
 $i=\tfl{b_0^{-1}(a-y\cdot\uhat)}$. Recall that $0<s<\infty$ is a constant
whose value can change from line to line.

From \eqref{exp-bound}  and an exponential Chebyshev 
\[
P^\w_{x,y}[\gamma_k> m]\le P^\w_{x}[X_m\cdot\uhat\le k] 
\le e^{sk-sx\cdot\uhat-h_1m}
\]
for all $x,y\in\Z^d$, $k\in\Z$ and $m\ge 0$.  Above and in the remainder
of this proof $h_1$, $h_2$ and  $h_3$ are small positive constants.
Finally we derive
\begin{align*}
P^\w_{x,y}[\gamma_{\elltil(I)}> m]&\le 
P^\w_{x,y}[\elltil(I)\ge k+x\cdot\uhat]
+  P^\w_{x,y}[\gamma_{k+x\cdot\uhat}> m]\\
&\le  2e^{s(y-x)\cdot\uhat-sk} +  e^{sk-h_1m}
\le Ce^{s(y-x)\cdot\uhat-h_2m}.
\end{align*}
To justify the inequalities  above
assume $m\ge 4sb_0/h_1> 4s/h_1 $ and pick $k$ in the range 
\[  
\frac{h_1m}{2s}+ (y-x)\cdot\uhat \le   k\le \frac{3h_1m}{4s}
+ (y-x)\cdot\uhat.
\]  

To summarize, at this point we have 
\be
P^\w_{x,y}[\gamma_{L}> m]\le Ce^{s(y-x)\cdot\uhat-h_2m}
\quad\text{for $x\cdot\uhat<y\cdot\uhat$.} 
\label{gamaux9}
\ee

To extend this estimate to the case 
$x\cdot\uhat\ge y\cdot\uhat$, simply allow $\Xtil$ to go
above $x$ and then  apply  \eqref{gamaux9}.  
By an application of the overshoot bound \eqref{overshoot}
and \eqref{gamaux9} at the point $(x, \Xtil_{\gatil(x\cdot\uhat+1)})$ 
\begin{align*}
&P^\w_{x,y}[\gamma_{L}> m]
\le E^\w_{x,y}  P^\w_{x, \Xtil_{\gatil(x\cdot\uhat+1)}}[\gamma_{L}> m]\\
&\le P^\w_{x,y}[\Xtil_{\gatil(x\cdot\uhat+1)}>x\cdot\uhat+\e m ]
+ Ce^{s\e m-h_2m}
\le Ce^{-h_3m}
\end{align*}
if we take $\e>0$ small enough.   

We have proved the lemma for $\gamma_L$, and the same argument
works for $\gatil_L$. 
\end{proof}

 Assuming
that  $X_0\cdot\uhat=\Xtil_0\cdot\uhat$
define the joint stopping times
\begin{align*}
(\rho,\tilde\rho)=
(\gamma^+_{M_{\beta\wedge\tilde\beta}\vee{\Mtil}_{\beta\wedge\tilde\beta}}
\,,
\tilde\gamma^+_{M_{\beta\wedge\tilde\beta}\vee
{\Mtil}_{\beta\wedge\tilde\beta}}\,)
\end{align*}
and 
\be
(\nu_1,\tilde\nu_1)=\begin{cases}
(\rho,\tilde\rho)
+(\gamma_L,\tilde\gamma_L)\circ\theta^{\rho, \rhotil} 
&\text{if $\rho\vee\rhotil<\infty$}\\
\infty &\text{if $\rho=\rhotil=\infty$.} 
\end{cases}\label{defnunutil}\ee 
Notice that $\rho$ and $\rhotil$ are finite or 
infinite together, and they are infinite 
iff neither walk backtracks below its initial level 
($\beta=\tilde\beta=\infty$). 
Let $\nu_0=\tilde\nu_0=0$ and for $k\ge0$ define
\begin{align*}
&(\nu_{k+1},\tilde\nu_{k+1})=(\nu_k,\tilde\nu_k)
+(\nu_1,\tilde\nu_1)\circ\theta^{\nu_k,\nutil_k}.
\end{align*}
Finally let
$(\nu,\tilde\nu)=(\gamma_L,\tilde\gamma_L)$,
$K=\sup\{k\ge0:\nu_k\vee\tilde\nu_k<\infty\}$, and
\be  (\mu_1,\mutil_1)
=(\nu,\tilde\nu)+(\nu_K,\tilde\nu_K)\circ\theta^{\nu,\nutil}.
\label{defmumutil}\ee 
These represent the first {\sl common regeneration times}
 of the two paths.  Namely, $X_{\mu_1}\cdot\uhat=
\Xtil_{\mutil_1}\cdot\uhat$ and for all $n\ge 1$,
\[ \text{$X_{\mu_1-n}\cdot\uhat< X_{\mu_1}\cdot\uhat
\le X_{\mu_1+n}\cdot\uhat$ and 
$\Xtil_{\mutil_1-n}\cdot\uhat< \Xtil_{\mutil_1}\cdot\uhat
\le \Xtil_{\mutil_1+n}\cdot\uhat$.}\]

Next we extend the exponential tail bound to the regeneration times. 

\begin{lemma}
There exist  constants $0<C<\infty$ and $\bar\eta\in(0,1)$  
such that, for all $x,y\in\mathbb{V}_d=\{z\in\Z^d: z\cdot\uhat=0\}$,
$k\ge0$, and $\P$-a.e.\ $\w$, we have
\be
P_{x,y}^\w(\mu_1\vee\mutil_1\ge k)
\le C(1-\bar\eta)^k.\label{mu-bound}
\ee
\label{mu-lemma}
\end{lemma}

\begin{proof}
We prove  geometric tail bounds successively
for $\gamma_1^+$, $\gamma_\ell^+$, $\gamma_{M_r}^+$, $\rho$, $\nu_1$, $\nu_k$,
and finally for $\mu_1$. 
To begin, \eqref{exp-bound} implies that
\[P_0^\w(\gamma_1^+\ge n)\le P_0^\w(X_{n-1}\cdot\uhat\le1)\le 
e^{s_2}(1-\eta_1)^{n-1}\]
with $\eta_1=s_2\delta/2$, for some small $s_2>0$.
By summation by parts 
\[E_0^\w(e^{s_3\gamma_1^+})\le e^{s_2}J_{s_3},\]
for a  small enough $s_3>0$ and $J_s=1+(e^s-1)/(1-(1-\eta_1)e^s)$. 
By the Markov property for  
$\ell\ge1$,
\[E_0^\w(e^{s_3\gamma_\ell^+})\le
\sum_{x\cdot\uhat>\ell-1}
E_0^\w(e^{s_3\gamma_{\ell-1}^+},X_{\gamma_{\ell-1}^+}=x)
E_x^\w(e^{s_3\gamma_\ell^+}).\]
But if $x\cdot\uhat>\ell-1$, then
$E_x^\w(e^{s_3\gamma_\ell^+})\le E_0^{T_x\w}(e^{s_3\gamma_1^+})$.
Therefore by induction
\begin{align}
\label{lambda-bound}
E_0^\w(e^{s_3\gamma_\ell^+})\le (e^{s_2}J_{s_3})^\ell\qquad
\text{for any integer $\ell\ge 0$}.
\end{align}

Next for an integer $r\ge1$, 
\begin{align*}
E_0^\w(e^{s_4\gamma_{M_r}^+})
&=\sum_{\ell=0}^\infty E_0^\w(e^{s_4\gamma_\ell^+},M_r=\ell)
\le\sum_{\ell=0}^\infty
E_0^\w(e^{2s_4\gamma_\ell^+})^{1/2}P_0^\w(M_r=\ell)^{1/2}\\
&\le C\sum_{\ell=0}^\infty
(e^{s_2}J_{2s_4})^{\ell/2}(\one\{\ell<3Mr|\uhat|\}+e^{-s_5\ell})^{1/2}
\le C (e^{s_2}J_{2s_4})^{Cr},
\end{align*}
for some $C$ and for positive but  small enough $s_2$, $s_4$, and $s_5$.
In the last inequality above we used the fact that $e^{s_2}J_{2s_4}$ converges
to $1$ as first $s_4\searrow 0$ and  then $s_2\searrow 0$.
In the second-to-last inequality we used
\eqref{bounded-step} to get the bound
\[\sum_{i=1}^r P_0^\w(X_i\cdot\uhat\ge \ell)\le
\sum_{i=1}^r e^{-s \ell}e^{M|\uhat|s i}  \le Ce^{-s_5\ell} 
\quad \text{if }\ell\ge 3M|\uhat|r.
%\left\{ \begin{matrix} 1\hfill&\text{if }\ell< 2M|\uhat|i,\\
%e^{-s \ell}e^{M|\uhat|s i}\hfill&\text{if }\ell\ge 2M|\uhat|i. \end{matrix}\right.
\]
Above we assumed that the walk $X$ starts at $0$. 
Same bounds work for any  $x\in\mathbb{V}_d$ because
a shift  orthogonal to $\uhat$ does not alter levels, 
in particular 
$P_x^\w(M_r=\ell)=P_0^{T_x\w}(M_r=\ell)$. 

By this same observation we 
 show that for all $x,y\in\mathbb{V}_d$
\[E_{x,y}^\w(e^{s_4\gamma_{{\Mtil}_r}^+})\le C (e^{s_2}J_{2s_4})^{Cr}\]
by repeating the earlier series of inequalities.

Using \eqref{exp-bound} and these estimates gives 
%Using \eqref{exp-bound} and the fact that 
%$E_x^\w(e^{s_4\gamma^+_{M_r}})=E_0^{T_x\w}(e^{s_4\gamma^+_{M_r}})$, we get
for $x,y\in\mathbb{V}_d$
\begin{align*}
P_{x,y}^\w(\rho\ge n,\beta\wedge\tilde\beta<\infty)
&=\sum_{r=1}^\infty
P_{x,y}^\w(\gamma^+_{M_r\vee{\Mtil}_r}\ge n,\beta\wedge\tilde\beta=r)\\
&\le e^{-s_4n/2}\sum_{1\le r\le\e n}
E_x^\w(e^{s_4\gamma^+_{M_r}})^{1/2}
E_{x,y}^\w(e^{s_4\gamma^+_{{\Mtil}_r}})^{1/2}\\
&\quad+\sum_{r>\e n}
\bigl(P_x^\w\{X_r\cdot\uhat<x\cdot\uhat\}+
P_y^\w\{\Xtil_r\cdot\uhat<y\cdot\uhat\}\,\bigr)\\
&\le C\e n e^{-s_4 n/2}(e^{s_2}J_{2s_4})^{C\e n}+C(1-s_6\delta/2)^{\e n}.
\end{align*}
Taking $\e>0$ small enough shows the existence
of a constant $\eta_2>0$ such that for all $x,y\in\mathbb{V}_d$, 
 $n\ge1$, and $\P$-a.e.\ $\w$,
\begin{align*}
P_{x,y}^\w(\rho\ge n,\beta\wedge\tilde\beta<\infty)\le C(1-\eta_2)^n.
%\label{rho-bound}
\end{align*}

Same bound works for $\rhotil$ also.  
We combine this with \eqref{gamma_L-bound} 
to get  a geometric tail bound for 
$\nu_1\one\{\beta\wedge\tilde\beta<\infty\}$.  Recall
 definition \eqref{defnunutil} and take $\e>0$ small. 
\begin{align*}
&P^\w_{x,y}[ \nu_1\ge k,\,  \beta\wedge\tilde\beta<\infty]\\
&\le P^\w_{x,y}[ \rho\ge k/2,\,  \beta\wedge\tilde\beta<\infty]
+P^\w_{x,y}[ 
 \beta\wedge\tilde\beta<\infty,\,
\lvert X_\rho\cdot\uhat-\Xtil_{\rhotil}\cdot\uhat\rvert> \e k  ]\\
&\qquad + P^\w_{x,y}[ \gamma_L\circ\theta^{\rho,\rhotil}\ge k/2,\,
 \beta\wedge\tilde\beta<\infty,\,
\lvert X_\rho\cdot\uhat-\Xtil_{\rhotil}\cdot\uhat\rvert\le \e k \, ].
\end{align*}
On the right-hand side above we have an exponential bound
for each of the three probabilities: the first probability gets it 
 from the estimate immediately above, the 
second  from a combination of that and \eqref{bounded-step},
and the third from  \eqref{gamma_L-bound}: 
\begin{align*}
&P^\w_{x,y}[ \gamma_L\circ\theta^{\rho,\rhotil}\ge k/2,\,
 \beta\wedge\tilde\beta<\infty,\,
\lvert X_\rho\cdot\uhat-\Xtil_{\rhotil}\cdot\uhat\rvert\le \e k \, ]\\
&=
E^\w_{x,y}\bigl[ \one\{
 \beta\wedge\tilde\beta<\infty,\,
\lvert X_\rho\cdot\uhat-\Xtil_{\rhotil}\cdot\uhat\rvert\le \e k \}\,
P^\w_{X_\rho\,,\,\Xtil_{\rhotil}} \{
 \gamma_L\ge k/2\}\,   \bigr]\\
&\le Ce^{a_1\e k- a_2k/2}. 
\end{align*} 
The constants in the  last bound above are those from 
 \eqref{gamma_L-bound}, and we choose $\e< a_2/(2a_1)$. 
We have thus established  that 
\[E_{x,y}^\w(e^{s_7\nu_1}\one\{\beta\wedge\tilde\beta<\infty\})
\le {\bar J}_{s_7}\]
for a small enough $s_7>0$,  with
 ${\bar J}_s=C(1-(1-\eta_3)e^s)^{-1}$ and 
$\eta_3>0$.

%On the other hand, 
%\[E_{x,y}^\w(e^{s_3\nu_1})=E_x^\w(e^{s_7\gamma_{\ell_0}^+})\le 
%J_{s_3}^{(y\cdot\uhat-x\cdot\uhat)^+}.\]
To move from $\nu_1$ to $\nu_k$ 
use the Markov property and induction:
\begin{align*}
&E_{x,y}^\w(e^{s_7\nu_k}\one\{\nu_k\vee\tilde\nu_k<\infty\})\\
&=\sum_{z,\tilde z}
E_{x,y}^\w(e^{s_7\nu_{k-1}}\one\{\nu_{k-1}\vee\tilde\nu_{k-1}<\infty,
X_{\nu_{k-1}}=z,
\Xtil_{\tilde\nu_{k-1}}=\tilde z\})\\
&\qquad\times \;E_{z,\tilde z}^\w(e^{s_7\nu_1}
\one\{\beta\wedge\tilde\beta<\infty\})\\
&\le {\bar J}_{s_7}E_{x,y}^\w(e^{s_7\nu_{k-1}}\one\{\nu_{k-1}\vee\tilde\nu_{k-1}<\infty\})
\le \dotsm \le {\bar J}_{s_7}^k.
\end{align*}

Next, use the Markov property at
 the joint stopping times $(\nu_k,\tilde\nu_k)$,
\eqref{gamma_L-bound},
\eqref{beta-bound}, and induction to derive 
\begin{align*}
&P_{x,y}^\w(K\ge k)\le P_{x,y}^\w(\nu_k\vee\tilde\nu_k<\infty)\\
&\qquad\le
\sum_{z,\tilde z}P_{x,y}^\w(\nu_{k-1}\vee\tilde\nu_{k-1}<\infty,
X_{\nu_{k-1}}=z,\Xtil_{\tilde\nu_{k-1}}=\tilde z)
P_{z,\tilde z}^\w(\beta\wedge\tilde\beta<\infty)\\
&\qquad\le
(1-\eta^2)P_{x,y}^\w(\nu_{k-1}\vee\tilde\nu_{k-1}<\infty)
\le(1-\eta^2)^k.
\end{align*}
Finally  use the Cauchy-Schwarz and Chebyshev inequalities
to write
\begin{align*}
P_{x,y}^\w(\nu_K\ge n)
&=\sum_{k\ge1}P_{x,y}^\w(\nu_k\ge n,K=k)\\
&\le\sum_{k>\e n}(1-\eta^2)^k+e^{-s_7n}\sum_{1\le k\le\e n}
E_{x,y}^\w(e^{s_7\nu_k}\one\{\nu_k\vee\tilde\nu_k<\infty\})\\
&\le C(1-\eta^2)^{\e n}+C\e n e^{-s_7n} {\bar J}_{s_7}^{\e n}.
\end{align*}
Looking at the definition \eqref{defmumutil} of $\mu_1$
we see that an exponential tail bound follows by applying
 \eqref{gamma_L-bound} to the $\nu$-part and by taking 
 $\e>0$ small enough in the last calculation above.
Repeat the same argument for $\tilde\mu_1$ to
conclude the proof of \eqref{mu-bound}.
\end{proof}

After these preliminaries define the sequence of
common regeneration times by $\mu_0=\mutil_0=0$ and 
\be
(\mu_{i+1},\tilde\mu_{i+1})=(\mu_i,\tilde\mu_i)
+(\mu_1,\tilde\mu_1)\circ\theta^{\mu_i, \tilde\mu_i}.
\label{defmukmutilk}
\ee
The next tasks are to identify suitable Markovian 
structures and to develop a coupling.

\begin{proposition} 
The process  $(\Xtil_{\mutil_i}-X_{\mu_i})_{i\ge 1}$ is a Markov chain
on $\mathbb{V}_d$ with transition probability
\be
q(x,y)=P_{0,x}[  \Xtil_{\mutil_1}-X_{\mu_1}= y\,\vert\, 
\beta=\tilde\beta=\infty].
\label{defqxy}
\ee
\label{Ymcprop}\end{proposition}

Note that the time-homogeneous 
Markov chain does not start from $\Xtil_{0}-X_{0}$ because the transition
  to $\Xtil_{\mutil_1}-X_{\mu_1}$ does not include the condition
$\beta=\tilde\beta=\infty$. 

\begin{proof} Express the iteration of the common regeneration times
as
\[
(\mu_i,\mutil_i)= (\mu_{i-1},\mutil_{i-1})\,+\,
 \bigl( (\nu,\nutil) + (\nu_{K},\nutil_{K})
\circ\theta^{\nu,\nutil}\bigr)\circ\theta^{\mu_{i-1},\mutil_{i-1}},
\quad i\ge 1.
\]
Let $K_i$ be the value of $K$ at the $i$th iteration:
\[ K_i= K\circ\theta^{\nu,\nutil}\circ\theta^{\mu_{i-1},\mutil_{i-1}}.
\]
%Let  
%\[ L_i =X_{\mu_i}\cdot\uhat =\Xtil_{\mutil_i}\cdot\uhat \]
%be the level where the $i$th common regeneration happens. 

Let $n\ge 2$ and $z_1,\dotsc, z_n\in\mathbb{V}_d$.  
Write 
\begin{align}
&P_{0,z}[  \Xtil_{\mutil_i}-X_{\mu_i}=z_i\,\text{ for $ 1\le i\le n$}]
\label{temp-B-1}\\
&=\sum_{
(k_i, m_i,\mtil_i,v_i,\vtil_i)_{1\le i\le n-1}\in \Psi}
P_{0,z}\bigl[ K_i=k_i,\, \mu_i=m_i,\,\mutil_i=\mtil_i,\, \nn\\
&\qquad\qquad\qquad
\qquad   X_{m_i}=v_i\,\text{ and }\,\Xtil_{\mtil_i}=\vtil_i
\,\text{ for  $1\le i\le n-1$},\; \nn\\[3pt]
&\qquad \qquad \qquad\qquad\qquad
(\Xtil_{\mutil_1}-X_{\mu_1})\circ\theta^{m_{n-1},\mtil_{n-1}}=z_n \bigr]. 
\nn\end{align}
Above $\Psi$ is the set of 
vectors  $(k_i, m_i,\mtil_i,v_i,\vtil_i)_{1\le i\le n-1}$ 
such that   $k_i$ is nonnegative and $m_i$, $\mtil_i$, 
$v_i\cdot\uhat$, and  $\vtil_i\cdot\uhat$ are all
positive and  strictly increasing in $i$, 
 and $\vtil_i-v_i=z_i$.

Define the events
\[
A_{k,b,\btil}=\{ \nu+\nu_k\circ\theta^{\nu,\nutil}=b
\,,\, \nutil+\nutil_k\circ\theta^{\nu,\nutil}=\btil\}
%A_k=\{ (\nu_j\circ\theta^{\nu,\nutil})\vee(\nutil_j\circ\theta^{\nu,\nutil})
%<\infty \,\text{ for $ 0\le i\le k$}\}
\]
and 
\[
B_{b,\btil} =\{ X_j\cdot\uhat\ge X_0\cdot\uhat \,\text{ for $ 1\le j\le b$}\,,\;
\Xtil_j\cdot\uhat\ge \Xtil_0\cdot\uhat \,\text{ for $ 1\le j\le \btil$}\}.
\]
Let $m_0=\mtil_0=0$, 
 $b_i=m_i-m_{i-1} $ and $\btil_i=\mtil_i-\mtil_{i-1} $. 
Rewrite  the sum from above as
\begin{align*}
&\sum_{
(k_i, m_i,\mtil_i,v_i,\vtil_i)_{1\le i\le n-1}\in\Psi}
E_{0,z}\Bigl[\;  \prod_{i=1}^{n-1} \one\{A_{k_i, b_i, \btil_i}\}\circ\theta^{m_{i-1},\mtil_{i-1}} \\
&\qquad \times 
 \prod_{i=2}^{n-1} \one\{B_{b_i,\,\btil_i }\}\circ\theta^{m_{i-1},\mtil_{i-1}},
%%\\ &\qquad\qquad\qquad\qquad 
 \, X_{m_i}=v_i\,\text{ and }\,\Xtil_{\mtil_i}=\vtil_i\,\text{ for  $1\le i\le n-1$},\; \\[3pt]
&\qquad \qquad \qquad \beta\circ\theta^{m_{n-1}}=\tilde\beta\circ\theta^{\mtil_{n-1}}=\infty\,,\,
(\Xtil_{\mutil_1}-X_{\mu_1})\circ\theta^{m_{n-1},\mtil_{n-1}}=z_n \Bigr].\\
\intertext{Next restart the walks at times $(m_{n-1},\mtil_{n-1})$ to turn
the sum into the following.
}  
&\sum_{(k_i, m_i,\mtil_i,v_i,\vtil_i)_{1\le i\le n-1}\in\Psi}
\E \biggl\{ \; E^\w_{0,z}\Bigl[\;  \prod_{i=1}^{n-1} 
\one\{A_{k_i, b_i, \btil_i}\}\circ\theta^{m_{i-1},\mtil_{i-1}} \\
&\qquad \times 
 \prod_{i=2}^{n-1} \one\{B_{b_i,\,\btil_i }\}\circ\theta^{m_{i-1},\mtil_{i-1}},
 \, X_{m_i}=v_i\,\text{ and }\,\Xtil_{\mtil_i}=\vtil_i\,\text{ for  $1\le i\le n-1$}\Bigr] \\[3pt]
&\qquad \qquad \qquad 
\times P^\w_{v_{n-1}\,,\,\vtil_{n-1}}\Bigl[
\beta=\tilde\beta=\infty\,,\,
\Xtil_{\mutil_1}-X_{\mu_1}=z_n \Bigr]  \;\biggr\}.
\end{align*}
Inside the outermost braces 
the events in the first quenched expectation force  the level
\[ \ell= 
X_{m_{n-1}}\cdot\uhat=v_{n-1}\cdot\uhat=\Xtil_{\mtil_{n-1}}\cdot\uhat
=\vtil_{n-1}\cdot\uhat\]  to be 
a new maximal level for both walks.  Consequently   the first 
quenched expectation is a function of 
$\{\w_x: x\cdot\uhat<\ell\}$ while the last quenched
probability is a function of 
$\{\w_x: x\cdot\uhat\ge \ell\}$.
By independence of the environments, the sum becomes
\begin{align}
&\sum_{(k_i, m_i,\mtil_i,v_i,\vtil_i)_{1\le i\le n-1}\in\Psi}
E_{0,z}\Bigl[\;  \prod_{i=1}^{n-1} 
\one\{A_{k_i, b_i, \btil_i}\}\circ\theta^{m_{i-1},\mtil_{i-1}} 
\label{temp-B0}\\
&\qquad \times 
 \prod_{i=2}^{n-1} \one\{B_{b_i,\,\btil_i }\}
\circ\theta^{m_{i-1},\mtil_{i-1}},
 \, X_{m_i}=v_i\,\text{ and }\,\Xtil_{\mtil_i}=\vtil_i\,
\text{ for  $1\le i\le n-1$}\Bigr] \nn            \\[2pt]
&\qquad \qquad \qquad 
\times P_{v_{n-1}\,,\,\vtil_{n-1}}\Bigl[
\beta=\tilde\beta=\infty\,,\,
\Xtil_{\mutil_1}-X_{\mu_1}=z_n \Bigr].\nn
\end{align}
By a shift and a conditioning the 
last probability transforms as follows.  
\begin{align*}
&P_{v_{n-1}\,,\,\vtil_{n-1}}\bigl[
\beta=\tilde\beta=\infty\,,\,
\Xtil_{\mutil_1}-X_{\mu_1}=z_n \bigr]\\
&=P_{0, z_{n-1}}\bigl[
\Xtil_{\mutil_1}-X_{\mu_1}=z_n \,\big\vert\,  
\beta=\tilde\beta=\infty  \,\bigr]
P_{v_{n-1}\,,\,\vtil_{n-1}}\bigl[ \beta=\tilde\beta=\infty \bigr]
\\
&=q(z_{n-1},z_n) 
P_{v_{n-1}\,,\,\vtil_{n-1}}\bigl[ \beta=\tilde\beta=\infty \bigr].
\end{align*}
Now reverse the above use of independence to put the probability 
\[ P_{v_{n-1}\,,\,\vtil_{n-1}}[ \beta=\tilde\beta=\infty]\]
back together with the  expectation 
 \eqref{temp-B0}. 
Inside this
expectation  this furnishes  the event 
$\beta\circ\theta^{m_{n-1}}=\tilde\beta\circ\theta^{\mtil_{n-1}}=\infty$
and with this the union of the entire collection of events turns back into 
$\Xtil_{\mutil_i}-X_{\mu_i}=z_i$  for $1\le i\le n-1$.
Going back to the beginning on line \eqref{temp-B-1}
we see that we have now shown 
\begin{align*}
&P_{0,z}[  \Xtil_{\mutil_i}-X_{\mu_i}=z_i\,\text{ for $ 1\le i\le n$}]\\
&=P_{0,z}[  \Xtil_{\mutil_i}-X_{\mu_i}=z_i\,\text{ for $ 1\le i\le n-1$}]
q(z_{n-1},z_n).
\end{align*} 
Continue by induction. 
\end{proof}

The Markov chain $Y_k=\Xtil_{\mutil_k}-X_{\mu_k}$ will be 
compared to a random walk obtained by performing the same 
construction of joint regeneration times to two
independent walks in independent environments.   
To indicate the difference in construction we change
notation. Let the 
pair of walks $(X,\Xbar)$ obey $P_0\otimes P_z$ with $z\in\mathbb{V}_d$,
and denote the first backtracking time of the $\Xbar$ walk
by $\betabar=\inf\{n\ge 1: \Xbar_n\cdot\uhat<\Xbar_0\cdot\uhat\}$.
Construct the 
common regeneration times $(\rho_k,\rhobar_k)_{k\ge 1}$
for $(X,\Xbar)$  by the same recipe
[\eqref{defnunutil}, \eqref{defmumutil} and \eqref{defmukmutilk}] 
 as was used to construct
$(\mu_k,\mutil_k)_{k\ge 1}$ for $(X,\Xtil)$.
Define $\Ybar_k=\Xbar_{\rhobar_k}-X_{\rho_k}$.  An analogue
of the previous
proposition, which we will not spell out, 
shows that  $(\Ybar_k)_{k\ge 1}$ is a Markov chain with transition
\be
\qbar(x,y)=
P_0\otimes P_x[  \Xbar_{\rhobar_1}-X_{\rho_1}= y\,\vert\, 
\beta=\betabar=\infty].
\label{defqbar}\ee

 In the next two proofs we make use of the 
following decomposition. 
Suppose  $x\cdot\uhat=y\cdot\uhat=0$, and let
 $(x_1,y_1)$ be another pair of points on a common, higher
 level: $x_1\cdot\uhat=y_1\cdot\uhat=\ell>0$.  Then we can write
\be\begin{split}
&\{ (X_0,\Xtil_0)=(x,y),\,\beta=
\tilde{\beta}=\infty, \, 
(X_{\mu_1},\Xtil_{\mutil_1})=(x_1,y_1)\}\\
&\quad =
\bigcup_{(\gamma,\gatil)} 
\{ X_{0,n(\gamma)}=\gamma,\, \Xtil_{0,n(\gatil)}=\gatil, \,
\beta\circ\theta^{n(\gamma)}=
\tilde\beta\circ\theta^{n(\gatil)}=\infty\}.
\end{split}\label{pathdecomp1}\ee
Here  $(\gamma,\gatil)$ range over all pairs of paths that connect
$(x,y)$ to $(x_1,y_1)$, that stay between levels $0$ and $\ell-1$
before the final points, and 
for which a common regeneration fails at all levels before $\ell$.  
 $n(\gamma)$ is the index of the final point 
 along the path, so for example
$\gamma=(x=z_0,z_1,\dotsc, z_{n(\gamma)-1},  z_{n(\gamma)}=x_1)$.  

\begin{proposition} The process 
$(\Ybar_k)_{k\ge 1}$ is a symmetric random walk on $\mathbb{V}_d$ 
and its   transition probability satisfies
\[
\qbar(x,y)=\qbar(0,y-x)=\qbar(0,x-y)=
P_0\otimes P_0[  \Xbar_{\rhobar_1}-X_{\rho_1}= y-x\,\vert\, 
\beta=\betabar=\infty].
\]
\label{Ybarprop2}
\end{proposition}

\begin{proof}   It remains  to show that for independent $(X,\Xbar)$
 the transition \eqref{defqbar}
 reduces to  a symmetric random walk. This becomes
obvious once probabilities are decomposed into sums over paths
because the events of interest are insensitive to shifts by 
$z\in\mathbb{V}_d$.
\be\begin{split}
&P_0\otimes P_x[\beta=\betabar=\infty\,,\,
  \Xbar_{\rhobar_1}-X_{\rho_1}= y]\\
&=\sum_w P_0\otimes P_x[\beta=\betabar=\infty\,,\,
  X_{\rho_1}= w\,,\,\Xbar_{\rhobar_1}=y+w ]\\
&=\sum_w  \sum_{(\gamma, \gabar)} 
P_0[ X_{0,n(\gamma)}=\gamma,\, \beta\circ\theta^{n(\gamma)}=\infty]  
P_x[ X_{0,n(\gabar)}=\gabar, \,
\beta\circ\theta^{n(\gabar)}=\infty]\\   
&=\sum_w  \sum_{(\gamma, \gabar)} 
P_0[ X_{0,n(\gamma)}=\gamma] P_x[ X_{0,n(\gabar)}=\gabar]  
\bigl(P_0[\beta=\infty]\bigr)^2. 
\end{split}  
\label{temp-gam-7}
\ee

Above we used the decomposition idea from \eqref{pathdecomp1}.
Here  $(\gamma, \gabar)$ range  over the appropriate
class of pairs of
paths in $\Z^d$  such that $\gamma$ goes from $0$ to $w$ and  
$\gabar$ goes from $x$ to $y+w$.
% the paths do not backtract below
%level $0$ and their last points represent new maximal levels for both. 
% and the pair realizes the event $ \{ \beta=\betabar=\infty\,,\,
% X_{\rho_1}= w\,,\, \Xbar_{\rhobar_1}=y+w \}$.
The independence for the last equality above comes from 
noticing that the quenched probabilities
$
\text{$P^\w_0[X_{0,n(\gamma)}=\gamma]$ and $P^\w_w[\beta=\infty]$ }
$
depend on independent collections of environments. 
  
The probabilities on the last line of 
\eqref{temp-gam-7}   are not changed if each pair $(\gamma,\gabar)$ 
is replaced
by $(\gamma,\gamma')=(\gamma,\gabar-x)$.  These pairs connect
 $(0,0)$ to $(w,y-x+w)$.
  Because $x\in\mathbb{V}_d$ satisfies 
$x\cdot\uhat=0$, 
 the shift has not changed regeneration levels. 
This shift turns $P_x[ X_{0,n(\gabar)}=\gabar]$ on   the last line
of \eqref{temp-gam-7} into $P_0[ X_{0,n(\gamma')}=\gamma']$.
We can reverse  the steps in 
\eqref{temp-gam-7}   to arrive at the probability 
\[
P_0\otimes P_0[\beta=\betabar=\infty\,,\,
  \Xbar_{\rhobar_1}-X_{\rho_1}= y-x].
 \]
This proves $\qbar(x,y)=\qbar(0,y-x)$.

Once both walks start at $0$ 
 it is immaterial which is labeled 
$X$ and which $\Xbar$, hence symmetry holds. 
\end{proof}

It will be useful to know that $\qbar$ inherits all possible
transitions from $q$.

\begin{lemma} If $q(z,w)>0$ then also $\qbar(z,w)>0$. 
\label{qqbarlm3}
\end{lemma} 

\begin{proof} 
By the decomposition from \eqref{pathdecomp1}
 we can express
\[
P_{x,y}[  (X_{\mu_1},\Xtil_{\mutil_1})=(x_1,y_1) \vert \beta=
\tilde{\beta}=\infty ] 
=
\sum_{(\gamma,\gatil)} 
\frac{\E P^\w[\gamma]P^\w[\gatil]P^\w_{x_1}[\beta=\infty]
P^\w_{y_1}[\beta=\infty ]}{P_{x,y}[\beta=
\tilde{\beta}=\infty ]}.
\]
If this probability is positive, then at least one pair 
$(\gamma,\gatil)$ satisfies $\E P^\w[\gamma]P^\w[\gatil]>0$.
This implies that $ P[\gamma]P[\gatil]>0$ so that also 
\[
P_{x}\otimes P_{y}[  (X_{\mu_1},\Xtil_{\mutil_1})=(x_1,y_1) \vert \beta=
\tilde{\beta}=\infty ] >0. 
\qedhere
\]
\end{proof} 

In the sequel we detach  the notations $Y=(Y_k)$ and $\Ybar=(\Ybar_k)$ 
from their original definitions  in terms of the walks
$X$, $\Xtil$ and $\Xbar$,  
 and use  $(Y_k)$ and $(\Ybar_k)$  to denote canonical Markov chains with
transitions $q$ and $\qbar$. 
Now we construct  a coupling.

\begin{proposition} The single-step transitions $q(x,y)$ for $Y$ and 
$\qbar(x,y)$ for $\Ybar$ 
can be coupled in such a way that, when the processes start
from a common state $x$, 
\[
P_{x,x}[Y_1\ne\Ybar_1] \le Ce^{-\alpha_1\abs{x}}
\]
for all $x\in\mathbb{V}_d$.  Here $C$ and $\alpha_1$ are finite
positive constants independent of $x$. 
\label{qqbarprop4}\end{proposition} 

\begin{proof} We start by constructing a coupling of three walks $(X,\Xtil,\Xbar)$
such that the pair $(X,\Xtil)$ has distribution $P_{x,y}$ and 
 the pair $(X,\Xbar)$ has distribution $P_{x}\otimes P_y$.

First let $(X,\Xtil)$ be two independent walks in a common environment
$\w$ as before. Let $\ombar$ be an environment independent
of $\w$.   Define the walk $\Xbar$ as follows. 
Initially $\Xbar_0=\Xtil_0$.
On the sites $\{X_k:0\le k<\infty\}$ $\Xbar$ obeys environment 
$\ombar$, and on all other sites $\Xbar$ obeys $\omega$.  
$\Xbar$ is coupled to agree with $\Xtil$ until the time
\[
T=\inf\{ n\ge 0: \Xbar_n\in \{X_k:0\le k<\infty\}\,\}
\]
it hits the path of $X$. 

The coupling between $\Xbar$ and $\Xtil$  can be achieved simply as
follows. Given $\w$ and $\ombar$,
  for each $x$ create two independent  i.i.d.~sequences  
$(z^x_k)_{k\ge 1}$  and $(\zbar^x_k)_{k\ge 1}$ with distributions
\[
Q^{\w,\ombar}[z^x_k=y]=\pi_{x,x+y}(\w)
\quad\text{and}\quad
Q^{\w,\ombar}[\zbar^x_k=y]=\pi_{x,x+y}(\ombar).
\]
Do this independently at each $x$. 
Each time the $\Xtil$-walk visits state $x$,
it uses a new $z^x_k$ variable as its next step, and never reuses the same
$z^x_k$ again.  The $\Xbar$ walk operates the same way except that
it uses the variables $\zbar^x_k$ when $x\in\{X_k\}$ and 
the $z^x_k$ variables when  $x\notin\{X_k\}$.  Now $\Xbar$ and $\Xtil$
follow the same steps $z^x_k$ until $\Xbar$ hits the set $\{X_k\}$.

It is intuitively obvious that the walks $X$ and $\Xbar$ are 
independent because they never use the same environment.  
The following calculation verifies this. 
Let $X_0=x_0=x$ and $\Xtil=\Xbar=y_0=y$ be the initial states, and 
$\PP_{x,y}$  the joint measure created by the coupling. 
Fix finite vectors $x_{0,n}=(x_0,\dotsc,x_n)$ and 
$y_{0,n}=(y_0,\dotsc,y_n)$ and recall also the notation
 $X_{0,n}=(X_0,\dotsc,X_n)$. 

The description of the coupling tells us to start as follows.
\begin{align*}
&\PP_{x,y}[X_{0,n}=x_{0,n}, \Xbar_{0,n}=y_{0,n}]
=\int\P(d\w)\int\P(d\ombar) \int P_x^\w(d{z_{0,\infty}})
\one\{z_{0,n}=x_{0,n}\} \\
&\qquad \times 
\prod_{i: y_i\notin\{z_k:\,0\le k<\infty\}}\pi_{y_i,y_{i+1}}(\w)\cdot
\prod_{i: y_i\in\{z_k:\,0\le k<\infty\}}\pi_{y_i,y_{i+1}}(\ombar)\\
\intertext{[by dominated convergence]}
&=\lim_{N\to\infty}\int\P(d\w)\int\P(d\ombar) \int P_x^\w(dz_{0,N})
\,\one\{z_{0,n}=x_{0,n}\} \\
&\qquad \times 
\prod_{i: y_i\notin\{z_k:\,0\le k\le N\}}\pi_{y_i,y_{i+1}}(\w)\cdot
\prod_{i: y_i\in\{z_k:\,0\le k\le N\}}\pi_{y_i,y_{i+1}}(\ombar)\\
&=\lim_{N\to\infty} \sum_{z_{0,N}: z_{0,n}=x_{0,n}}
\int\P(d\w) \, P_x^\w[X_{0,N}=z_{0,N}]
\prod_{i: y_i\notin\{z_k:\,0\le k\le N\}}\pi_{y_i,y_{i+1}}(\w)\\
&\qquad\qquad \times 
\int\P(d\ombar) 
\prod_{i: y_i\in\{z_k:\,0\le k\le N\}}\pi_{y_i,y_{i+1}}(\ombar)\\
\intertext{[by independence of the two functions of $\w$]}
&=\lim_{N\to\infty} \sum_{z_{0,N}: z_{0,n}=x_{0,n}}
\int\P(d\w) \, P_x^\w[X_{0,N}=z_{0,N}] 
\int \P(d\w) 
\prod_{i: y_i\notin\{z_k:\,0\le k\le N\}}\pi_{y_i,y_{i+1}}(\w)\\
&\qquad\qquad \times 
\int\P(d\ombar) 
\prod_{i: y_i\in\{z_k:\,0\le k\le N\}}\pi_{y_i,y_{i+1}}(\ombar)\\
&=P_x[X_{0,n}=x_{0,n}]\cdot P_y[X_{0,n}=y_{0,n}].
\end{align*}

Thus at this point the coupled pairs $(X,\Xtil)$ and $(X,\Xbar)$
have the desired marginals $P_{x,y}$ and $P_x\otimes P_y$. 

 Next construct the common regeneration
times $(\mu_1,\mutil_1)$ for  $(X,\Xtil)$ and
$(\rho_1,\rhobar_1)$ for  $(X,\Xbar)$ by the earlier 
recipes. 
Define  two pairs of walks stopped at their
common regeneration times:
\be
(\Gamma,\Gammabar)\equiv\bigl( (X_{0,\,\mu_1},\Xtil_{0,\,\mutil_1}),
(X_{0,\,\rho_1},\Xbar_{0,\,\rhobar_1})\bigr).
\label{defGaGa}
\ee

Suppose the sets $X_{[0,\,\mu_1\vee\rho_1)}$ and 
$\Xtil_{[0,\,\mutil_1\vee\rhobar_1)}$ do not intersect. Then the 
construction implies that the path $\Xbar_{0,\,\mutil_1\vee\rhobar_1}$ 
 agrees with 
$\Xtil_{0,\,\mutil_1\vee\rhobar_1}$, and this forces the equalities
$(\mu_1,\mutil_1)=(\rho_1,\rhobar_1)$ 
 and $(X_{\mu_1},\Xtil_{\mutil_1})=(X_{\rho_1},\Xbar_{\rhobar_1})$.
We insert an estimate on this event.

\begin{lemma}
There exist  constants $0<C,s<\infty$ 
such that, for all $x,y\in\mathbb{V}_d$
 and $\P$-a.e.\ $\w$, 
\be
P_{x,y}^\w(X_{[0,\,\mu_1\vee\rho_1)}\cap
\Xtil_{[0,\,\mutil_1\vee\rhobar_1)}\ne\emptyset)\le 
Ce^{-s|x-y|}.
\label{capbound}
\ee
\end{lemma}

\begin{proof}
Write 
\begin{align*}
P_{x,y}^\w(X_{[0,\,\mu_1\vee\rho_1)}\cap\Xtil_{[0,\,\mutil_1\vee\rhobar_1)}\ne\emptyset)
&\le P_{x,y}^\w(\mu_1\vee\mutil_1\vee\rho_1\vee\rhobar_1>\e|x-y|)\\
&\qquad+P_x^\w(\,\max_{1\le i\le\e|x-y|}|X_i-x|\ge|x-y|/2)\\
&\qquad+P_y^\w(\,\max_{1\le i\le\e|x-y|}|X_i-y|\ge|x-y|/2).
\end{align*}
By \eqref{mu-bound} and its analogue for $(\rho_1,\rhobar_1)$
 the first term on the right-hand-side decays 
exponentially in $|x-y|$.
Using \eqref{bounded-step} the second and third terms are bounded by 
$\e|x-y|e^{-s|x-y|/2}e^{\e s|x-y|M}$, for $s>0$ small enough. 
Choosing $\e>0$ small enough finishes the proof.
\end{proof}

From \eqref{capbound} 
 we obtain 
\be
\PP_{x,y}\bigl[\,(X_{\mu_1},\Xtil_{\mutil_1})\ne
(X_{\rho_1},\Xbar_{\rhobar_1})\,\bigr]
\le  \PP_{x,y}\bigl[\,\Gamma\ne\Gammabar\bigr]
\le Ce^{-s\abs{x-y}}.
\label{XXtil-goal5}
\ee

  But we are not finished yet: it remains
 to include the conditioning on no backtracking.
For this purpose generate an i.i.d.~sequence 
$(X^{(m)},\Xtil^{(m)},\Xbar^{(m)})_{m\ge 1}$, each triple constructed
as above.  Continue to write $\PP_{x,y}$ for  the
probability measure of the entire sequence.
 Let $M$ be the first $m$ such that 
the paths  $(X^{(m)},\Xtil^{(m)})$ do not backtrack,
which means that 
\[
\text{$X^{(m)}_k\cdot\uhat\ge X^{(m)}_0\cdot\uhat$ 
and 
$\Xtil^{(m)}_k\cdot\uhat\ge \Xtil^{(m)}_0\cdot\uhat$ for all $k\ge 1$.}
\]
Similarly define  $\Mbar$ for $(X^{(m)},\Xbar^{(m)})_{m\ge 1}$.  $M$ and 
$\Mbar$ are stochastically bounded by geometric random variables 
 by \eqref{beta-bound}. 

The pair of walks 
$(X^{(M)},\Xtil^{(M)})$ is now 
distributed as a pair of walks under the measure 
$P_{x,y}[\,\cdot\,\vert \beta=\tilde\beta=\infty]$, 
while  
$(X^{(\Mbar)},\Xbar^{(\Mbar)})$
 is distributed as a pair of walks under
$P_{x}\otimes P_y[\,\cdot\,\vert \beta=\betabar=\infty]$.

  Let also again 
\[\Gamma^{(m)}=
(X^{(m)}_{0\,,\,\mu^{(m)}_1},\Xtil^{(m)}_{0\,,\,\mutil^{(m)}_1})
 \quad\text{and}\quad  
\Gammabar^{(m)}=
(X^{(m)}_{0\,,\,\rho^{(m)}_1},\Xbar^{(m)}_{0\,,\,\rhobar^{(m)}_1})\]
  be the 
pairs of paths 
run up to  their common regeneration times. 
Consider  the two pairs 
of paths $(\Gamma^{(M)}, \Gammabar^{(\Mbar)})$
chosen by the random indices $(M,\Mbar)$. 
We insert one more lemma.

\begin{lemma}
For $s>0$ as above, and
 a new constant $0<C<\infty$, 
\be
\PP_{x,y}\bigl[\,\Gamma^{(M)}\ne\Gammabar^{(\Mbar)}\bigr]
\le Ce^{-s\abs{x-y}/2}.
\label{XXtil-goal7}
\ee
\end{lemma}
\begin{proof} 
Let $\Ac_m$ be the event that  
the walks $\Xtil^{(m)}$ and $\Xbar^{(m)}$ agree up to the maximum 
$\mutil^{(m)}_1\vee\rhobar^{(m)}_1$ of their regeneration times. 
The equalities $M=\Mbar$ and  
$\Gamma^{(M)}=\Gammabar^{(\Mbar)}$ are a consequence 
of the event $\Ac_1\cap\dotsm\cap \Ac_M$, for the following reason.
As pointed out earlier, on the event $\Ac_m$  we have the equality
of the regeneration times  $\mutil^{(m)}_1=\rhobar^{(m)}_1$
and of the  stopped paths 
$\Xtil^{(m)}_{0\,,\,\mutil^{(m)}_1}=
\Xbar^{(m)}_{0\,,\,\rhobar^{(m)}_1}$.  By definition, these walks
do not backtrack after the regeneration time. 
  Since the walks  $\Xtil^{(m)}$ and $\Xbar^{(m)}$ agree
up to this time, they must backtrack or fail to backtrack
together.  If this is true for 
each $m=1,\dotsc,M$, it forces  $\Mbar=M$, since the other factor 
in deciding  $M$ and $\Mbar$ are the paths $X^{(m)}$ that are common
to both.   And since the paths agree up to the regeneration times, 
we have  $\Gamma^{(M)}=\Gammabar^{(\Mbar)}$. 

Estimate \eqref{XXtil-goal7}  follows: 
\begin{align*}
&\PP_{x,y}\bigl[\,\Gamma^{(M)}\ne\Gammabar^{(\Mbar)}\,\bigr]
\le \PP_{x,y}\bigl[\,\Ac_1^c\cup\dotsm\cup \Ac_M^c\,\bigr]\\
&\le \sum_{m=1}^\infty \PP_{x,y}[M\ge m,\, \Ac_m^c\,]
\le \sum_{m=1}^\infty \bigl(\PP_{x,y}[M\ge m]\bigr)^{1/2}  
\bigl(\PP_{x,y}[ \Ac_m^c]\bigr)^{1/2}\\
&\le Ce^{-s\abs{x-y}/2}.
\end{align*}
The last step comes from the estimate in \eqref{capbound}
for each $\Ac_m^c$ and the geometric bound on $M$. 
\end{proof}

We are ready to finish the proof of Proposition 
\ref{qqbarprop4}. 
To create initial conditions 
 $Y_0=\Ybar_0=x$ take initial
states $(X^{(m)}_0,\Xtil^{(m)}_0)=(X^{(m)}_0,\Xbar^{(m)}_0)=(0,x)$.  
 Let the final outcome of
the coupling be the pair 
\[
(Y_1,\Ybar_1)=\bigl( \Xtil^{(M)}_{\mutil^{(M)}_1} \;-\;
X^{(M)}_{\mu^{(M)}_1}\,,\,
\Xbar^{(\Mbar)}_{\rhobar^{(\Mbar)}_1} \;-\;
X^{(\Mbar)}_{\rho^{(\Mbar)}_1}\bigr) 
\]
under the measure $\PP_{0,x}$. The marginal distributions
of $Y_1$ and $\Ybar_1$ are correct
[namely, given by the transitions 
\eqref{defqxy} and  \eqref{defqbar}]  because, as argued above,
the pairs of walks themselves have the right marginal distributions.
The event $\Gamma^{(M)}=\Gammabar^{(\Mbar)}$ implies 
$Y_1=\Ybar_1$, so 
estimate \eqref{XXtil-goal7} gives the bound claimed in 
Proposition \ref{qqbarprop4}.
 \end{proof}

The construction of the Markov chain is complete, and we return to
the main development of the proof.  It remains to prove a sublinear
bound on the expected number  
$E_{0,0}\lvert X_{[0,n)}\cap \Xtil_{[0,n)}\rvert $ 
of common points of two independent walks in a common environment. 
Utilizing the common regeneration times,  write
\be
E_{0,0}\lvert X_{[0,n)}\cap \Xtil_{[0,n)}\rvert
\le\sum_{i=0}^{n-1} E_{0,0}\lvert X_{[\mu_i,\mu_{i+1})}
\cap \Xtil_{[\mutil_i,\mutil_{i+1})}\rvert.
\label{capbd7}\ee

The term $i=0$ is a finite constant by bound \eqref{mu-bound}
 because the number of common points is
bounded by the number $\mu_1$ of steps.  
For each  $0<i<n$ 
apply a  decomposition
into pairs of paths from $(0,0)$ 
to given points $(x_1,y_1)$   in the style of \eqref{pathdecomp1}: 
$(\gamma,\gatil)$ are the pairs of paths with the property that 
\begin{align*}
&\bigcup_{(\gamma,\gatil)}
\{X_{0,n(\gamma)}=\gamma,\, \Xtil_{0,n(\gatil)}=\gatil,\,
\beta\circ\theta^{n(\gamma)}=
\tilde\beta\circ\theta^{n(\gatil)}=\infty\}\\
&\qquad
=\{ X_0=\Xtil_0=0,\,  X_{\mu_i}=x_1,\, \Xtil_{\mutil_i}=y_1\}.
\end{align*}
Each term $i>0$ in \eqref{capbd7} we rearrange as follows. 
\begin{align*}
&E_{0,0}\lvert X_{[\mu_i,\mu_{i+1})}
\cap \Xtil_{[\mutil_i,\mutil_{i+1})}\rvert\\
&=\sum_{x_1,y_1}\sum_{(\gamma,\gatil)} \E
P^\w_{0,0}[ X_{0,n(\gamma)}=\gamma,\, \Xtil_{0,n(\gatil)}=\gatil] \\
&\qquad \times
E^\w_{x_1,y_1}(\one\{\beta=\tilde{\beta}=\infty\}
 \lvert X_{[0\,,\, \mu_{1})}
\cap \Xtil_{[0\,,\,\mutil_{1})}\rvert\,)\\
&=\sum_{x_1,y_1}\sum_{(\gamma,\gatil)} \E
P^\w_{0,0}[ X_{0,n(\gamma)}=\gamma,\, \Xtil_{0,n(\gatil)}=\gatil]
P^\w_{x_1,y_1}[\beta=\tilde{\beta}=\infty] \\
&\qquad \times E^\w_{x_1,y_1}(\, 
 \lvert X_{[0\,,\,\mu_1)}
\cap \Xtil_{[0\,,\,\mutil_1)}\rvert  \,\vert\,
\beta=\tilde{\beta}=\infty\,)\\
&=\sum_{x_1,y_1} \E
P^\w_{0,0}[ X_{\mu_i}=x_1,\, \Xtil_{\mutil_i}=y_1]
 E^\w_{x_1,y_1}(\, 
 \lvert X_{[0\,,\,\mu_1)}
\cap \Xtil_{[0\,,\,\mutil_1)}\rvert  \,\vert\,
\beta=\tilde{\beta}=\infty\,).
\end{align*}
The last conditional quenched expectation above is handled by estimates
\eqref{beta-bound}, \eqref{mu-bound}, \eqref{capbound}
  and Schwarz inequality:
\begin{align*}
&E^\w_{x_1,y_1}(\, 
 \lvert X_{[0\,,\,\mu_1)}
\cap \Xtil_{[0\,,\,\mutil_1)}\rvert  \,\vert\,
\beta=\tilde{\beta}=\infty\,)\le \eta^{-2}  E^\w_{x_1,y_1}(\, 
 \lvert X_{[0\,,\,\mu_1)}
\cap \Xtil_{[0\,,\,\mutil_1)}\rvert\,)\\
&\qquad \le \eta^{-2}  E^\w_{x_1,y_1}(\mu_1\cdot \one\{
X_{[0\,,\,\mu_1)}\cap \Xtil_{[0\,,\,\mutil_1)}\ne\emptyset\}\,)\\
&\qquad\le \eta^{-2}  \bigl(E^\w_{x_1,y_1}[\mu_1^2]\bigr)^{1/2} 
\bigl( P^\w_{x_1,y_1}\{
X_{[0\,,\,\mu_1)}\cap \Xtil_{[0\,,\,\mutil_1)}\ne\emptyset\}\,\bigr)^{1/2}\\
&\qquad\le Ce^{-s\abs{x_1-y_1}/2}.
\end{align*}
Define  $h(x)=Ce^{-s\abs{x}/2}$, 
insert the last bound back up, and appeal to the Markov property
established in  Proposition \ref{Ymcprop}:
\begin{align*}
E_{0,0}\lvert X_{[\mu_i,\mu_{i+1})}
\cap \Xtil_{[\mutil_i,\mutil_{i+1})}\rvert
&\le  E_{0,0} \bigl[h( \Xtil_{\mutil_i}-X_{\mu_i})\bigr]\\
&=\sum_{x} P_{0,0}[\Xtil_{\mutil_1}-X_{\mu_1}=x]
\sum_yq^{i-1}(x,y)h(y).
\end{align*}

In order to apply Theorem \ref{greenthm1} from the Appendix, we check
its hypotheses in the next lemma.  Assumption \eqref{Yell5} enters
here for the first and only time.

\begin{lemma} The Markov chain $(Y_k)_{k\ge 0}$ with transition $q(x,y)$ 
%% $Y_k=\Xtil_{\mutil_k}-X_{\mu_k}$
and the symmetric random walk %%%$\Ybar_k=\Xbar_{\rhobar_k}-X_{\rho_k}$
$(\Ybar_k)_{k\ge 0}$ with transition $\qbar(x,y)$ 
satisfy assumptions {\rm (A.i)}, {\rm (A.ii)} and  {\rm (A.iii)}
stated in the beginning of the Appendix. 
\label{Yapplm1}\end{lemma}

\begin{proof} 
From Lemma \ref{mu-lemma} and \eqref{bounded-step}
 we get moment bounds
\[E_{0,x}\lvert \Xbar_{\rhobar_k}\rvert^m
\;+\; E_{0,x}\lvert X_{\rho_k}\rvert^m <\infty 
\]
for any power $m<\infty$.  This gives assumption (A.i), namely 
that $E_0\lvert \Ybar_1\rvert^3<\infty$.  
  The second part of assumption (A.ii) 
comes from Lemma \ref{qqbarlm3}.  Assumption (A.iii) comes
from Proposition \ref{qqbarprop4}.   

The only part that needs work is the  first part of assumption (A.ii).
We show that it 
follows from part \eqref{Yell5} of Hypothesis (R).
By \eqref{Yell5}  and non-nestling (N)  there exist two non-zero
vectors $y\ne z$ such that $z\cdot\uhat>0$ 
and $\E\pi_{0,y}\pi_{0,z}>0$.
 Now we have a number of  cases to consider.  
In each case we should describe an event 
that gives  $Y_1-Y_0$ a particular nonzero value  and whose
probability is bounded away from zero, uniformly over $x=Y_0$. 

\smallskip

{\bf Case 1:}  $y$ is noncollinear with $z$. The sign of $y\cdot\uhat$
gives three subcases. We do the trickiest one explicitly. 
Assume  $y\cdot\uhat<0$.   Find the smallest 
positive integer $b$ such that 
$
(y+bz)\cdot\uhat>0. 
$
Then find the minimal positive integers $k,m$ such that 
$k(y+bz)\cdot\uhat=m z\cdot\uhat$.  Below $P_x$ is
the path measure of the Markov chain $(Y_k)$ and then $P_{0,x}$
the measure of the walks $(X,\Xtil)$ as before. 
\begin{align*}
&P_x\{ Y_1-Y_0=ky+(kb-m)z\} \\
&\ge 
P_{0,x}\bigl\{ \Xtil_{\mutil_1}=x+ky+(k+1)bz\,,\, X_{\mu_1}=(m+b)z\,,\, 
\beta=\tilde{\beta}=\infty\bigr\}\\
&\ge \E\Bigl[  P^{T_x\w}_0\{ \text{$X_{i(b+1)+1}=i(y+bz)+z, \dotsc, 
X_{i(b+1)+b}=i(y+bz)+bz$,}\\
&\qquad\qquad\qquad \text{ $X_{(i+1)(b+1)}= (i+1)(y+bz)$ 
\ \  for $0\le i\le k-1$, \  and then  }\\
&\qquad\qquad\qquad
\text{ $X_{k(b+1)+1}=k(y+bz)+z\,,\dots,\,X_{k(b+1)+b}
= k(y+bz)+bz$ }\}\\
&\qquad\qquad \times 
  P^\w_0\{\, X_{1}=z\,,\, X_2=2z\,, \dotsc,\,X_{m+b} =(m+b)z\,\}\\
&\qquad\qquad\times
P^\w_{x+ky+(k+1)bz}\{\beta=\infty\}P^\w_{(m+b)z}\{\beta=\infty\}\Bigr]. 
\end{align*}
Regardless of possible
intersections of the paths, 
assumption \eqref{Yell5} and inequality \eqref{beta-bound} imply 
that  the quantity above has a  positive lower bound that is
independent of $x$.  The assumption that $y, z$
are nonzero and 
 noncollinear ensures
that $ky+(kb-m)z\ne 0$.  

\smallskip

{\bf Case 2:} $y$ is collinear with $z$.
Then there is a vector $w\not\in\R z$ such that $\E\pi_{0,w}>0$. 
If $w\cdot\uhat\le0$, then by Hypothesis (N) there exists
 $u$ such that $u\cdot\uhat>0$ and $\E\pi_{0,w}\pi_{0,u}>0$.
If $u$ is collinear with $z$, then replacing
$z$ by $u$ and $y$ by $w$ puts us back in Case 1. 
So, replacing $w$ by $u$ if necessary,
we can assume that $w\cdot\uhat>0$. We have four subcases, depending
on whether $x=0$ or not and $y\cdot\uhat<0$ or not. 

\smallskip

{\bf (2.a)} The case
$x\ne0$ is resolved simply by taking paths consisting of only $w$-steps
for one walk and only $z$-steps for the other, until they  meet on a 
common level and then never backtrack. 

\smallskip

{\bf (2.b)} The case $y\cdot\uhat>0$ corresponds to Case 3 in the proof of 
\cite[Lemma 5.5]{forbidden-qclt}.

\smallskip

{\bf (2.c)}
The only case left is $x=0$ and $y\cdot\uhat<0$. 
Let $b$ and $c$ be the smallest positive integers
such that $(y+bw)\cdot\uhat\ge0$
 and  $(y+cz)\cdot\uhat>0$. Choose minimal positive 
integeres $m\ge b$  and $n>c$ such that $m(w\cdot\uhat)=n(z\cdot\uhat)$.
Then,
\begin{align*}
&P_0\{Y_1-Y_0=nz-mw\}\\
&\ge P_{0,0}\{\Xtil_{\tilde\mu_1}=y+bw+nz,X_{\mu_1}=y+(b+m)w\}\\
&\ge\E\Big[
P_0^\w\{X_i=iw\text{ for }1\le i\le b 
        \text{ and }X_{b+1+j}=y+(b+j)w\text{ for }0\le j\le m\}\\
&\qquad\qquad
\times P_0^\w\{X_i=iw\text{ for }0\le i\le b, X_{b+1}=bw+z\text{ and then}\\
&\qquad\qquad\qquad\qquad  X_{b+1+j}=y+bw+jz\text{ for }1\le j\le n\}\\
&\qquad\qquad
\times P_{y+(b+m)w}^\w(\beta=\infty)P_{y+bw+nz}^\w(\beta=\infty)\Big].
\end{align*}
Since $w$ and $z$ are noncollinear, $mw\ne nz$.   For the same reason,
$w$-steps are always taken at points not visited before.
This makes the above lower bound positive. 
By the choice of $b$ and $z\cdot\uhat>0$, neither walk dips 
below level 0.

We can see that the first common regeneration
level for the two paths is $(y+bw+nz)\cdot\uhat$. 
The first walk backtracks from level $bw\cdot\uhat$ so this
is not a  common regeneration
level. 
The second walk  splits from the first walk 
at $bw$, takes a $z$-step up, and then backtracks using a $y$-step.
So the common regeneration level can only be at or above level 
$(y+bw+(c+1)z)\cdot\uhat$.
The fact that $n>c$ ensures that $(y+bw+nz)\cdot\uhat$ is high enough.
The minimality of $n$ ensures that this is the first such
level.
\end{proof}  

Now that  the assumptions have been checked, Theorem  \ref{greenthm1}
gives constants  $0<C<\infty$ and 
$0<\eta<1$ 
such that 
\[
\sum_{i=1}^{n-1} \sum_yq^{i-1}(x,y)h(y) \le Cn^{1-\eta}
\quad\text{ for all $x\in\mathbb{V}_d$ and $n\ge 1$.}
\]
Going back to \eqref{capbd7} and collecting the bounds along
the way gives the final estimate 
\[
E_{0,0}\lvert X_{[0,n)}\cap \Xtil_{[0,n)}\rvert \le Cn^{1-\eta}
\]
for all $n\ge 1$.  
This is \eqref{cond2} which was earlier shown to imply 
condition \eqref{cond} required by Theorem \ref{RS}. 
Previous work in Sections \ref{prelim} and \ref{substitution}
convert the CLT from Theorem \ref{RS} into the main result
Theorem \ref{main}.  The entire proof is complete, except for
the Green function estimate furnished by the Appendix.  

\appendix
\section{A Green function type bound}
Let us write a $d$-vector
in terms of coordinates  as $x=(x^1,\dotsc,x^d)$,
and similarly for random vectors  $X=(X^1,\dotsc,X^d)$. 

Let $Y=(Y_k)_{k\ge 0}$ be a Markov chain  on $\Z^d$ with 
transition probability $q(x,y)$, and let  $\Ybar=(\Ybar_k)_{k\ge 0}$ be a
 symmetric  random walk  on $\Z^d$ with 
transition probability $\qbar(x,y)=\qbar(y,x)=\qbar(0,y-x)$. 
Make the following assumptions.

\smallskip

(A.i) A third moment bound
$E_0\lvert \Ybar_1\rvert^3<\infty$. 

\smallskip

(A.ii) Some uniform  nondegeneracy: 
there is at least one index $j\in\{1,\dotsc,d\}$
and a constant $\kappa_0$  such that the coordinate $Y^j$ 
satisfies 
\be
P_x\{ Y^j_1-Y^j_0 \ge 1\}\ge \kappa_0>0 \quad\text{for all $x$.}
\label{Y-ell-ass}
\ee
(The inequality $\ge 1$ can be replaced by $\le -1$, the point is 
to assure that a cube is exited fast enough.) 
Furthermore, for every $i\in\{1,\dotsc,d\}$, 
if  the one-dimensional random walk 
$\Ybar^i$ is degenerate in the sense that $\qbar(0,y)=0$ for $y^i\ne 0$,
then so is the process  $Y^i$
in the sense that $q(x,y)=0$ whenever $x^i\ne y^i$.  In other words, 
any coordinate that can move in the $Y$ chain somewhere in space
can also move in the $\Ybar$ walk.

\smallskip

(A.iii) Most importantly, assume that for any initial state $x$  the transitions 
$q$ and $\qbar$ can be coupled  so that 
\[
P_{x,x}[Y_1\ne\Ybar_1]\le Ce^{-\alpha_1\lvert x\rvert}
\] 
where $0<C,\alpha_1<\infty$ are constants independent of $x$. 

\smallskip

Throughout the section $C$ will change value but $\alpha_1$ remains
the constant in the assumption above.  
Let $h$ be a function on $\Z^d$
 such that $0\le h(x)\le Ce^{-\alpha_2|x|}$ for constants 
$0<\alpha_2,C<\infty$. This section is devoted
to proving the following Green function type bound on the Markov chain. 

\begin{theorem} There are constants $0<C,\eta<\infty$ such 
that  
\[
\sum_{k=0}^{n-1} E_zh(Y_k) = \sum_y h(y) \sum_{k=0}^{n-1} P_0(Y_k=y)   
\le Cn^{1-\eta}\quad\text{ for all $n\ge 1$ and $z\in\Z^d$.} 
\]
\label{greenthm1}
\end{theorem}

To prove the estimate, we begin by
 discarding terms outside a cube of side $r=c_1 \log n$. 
Bounding probabilities crudely by 1  gives  
\begin{align*}
&\sum_{|y|> c_1 \log n} h(y) \sum_{k=0}^{n-1} P_z(Y_k=y)   \le 
n \sum_{|y|> c_1 \log n}  h(y)  \le  Cn\sum_{k > c_1 \log n} k^{d-1} e^{-\alpha_2k} \\
&\le Cn \sum_{k > c_1 \log n}  e^{-(\alpha_2/2)k} 
\le  Cn e^{-(\alpha_2/2)c_1 \log n} \le Cn^{1-\eta} 
\end{align*}
as long as $n$ is large enough so that 
$k^{d-1}\le e^{\alpha_2k/2}$, and this works for any $c_1$. 

Let \[B=[-c_1\log n, c_1\log n]^d.\]
Since $h$ is bounded,   it now remains to show that 
\be
\sum_{k=0}^{n-1} P_z(Y_k\in B) \le Cn^{1-\eta}. 
\label{goal-Y-1}
\ee
For this we can assume $z\in B$ since accounting for the time 
to enter $B$ for the first time can only improve the estimate. 

Bound \eqref{goal-Y-1}  will be achieved in two stages. First we show that the 
Markov chain $Y$ does not stay in $B$ longer  than a time
whose mean 
is a power of the size of $B$.  Second, we show that often
enough $Y$ 
 follows the random walk $\Ybar$ during its excursions outside $B$.
The random walk excursions are long and thereby we obtain \eqref{goal-Y-1}.
Thus our first task is to construct a suitable coupling
of $Y$ and $\Ybar$. 

\begin{lemma}  Let $\zeta=\inf\{n\ge 1: \Ybar\in A\}$ be the 
first entrance time of $\Ybar$ into some set $A\subseteq\Z^d$.
Then we can couple $Y$ and $\Ybar$ so that 
\[
P_{x,x}[\text{ $Y_k\ne\Ybar_k$ for some $1\le k\le \zeta$ }]
\le C E_x \sum_{k=0}^{\zeta-1} e^{-\alpha_1\lvert\Ybar_k\rvert}.
\]
\label{YYbarlm1}
\end{lemma} 

The proof shows that the statement works also
if  $\zeta=\infty$ is possible, but we will not need this case.

\begin{proof} For each state $x$ create an i.i.d.~sequence
$(Z^x_k, \Zbar^x_k)_{k\ge 1}$ such that $Z^x_k$ has 
distribution $q(x,x+\,\cdot\,)$,  $\Zbar^x_k$ has 
distribution $\qbar(x,x+\,\cdot\,)=\qbar(0,\,\cdot\,)$, and 
each pair $(Z^x_k, \Zbar^x_k)$ is coupled so that 
$P(Z^x_k\ne \Zbar^x_k) \le Ce^{-\alpha_1\abs{x}}$.  
For distinct $x$ these sequences are independent. 

Construct the process $(Y_n,\Ybar_n)$ as follows: with 
counting measures 
\[L_n(x)=\sum_{k=0}^n \one\{Y_k=x\}
\quad\text{and}\quad 
 \Lbar_n(x)=\sum_{k=0}^n \one\{\Ybar_k=x\} \quad(n\ge 0) 
\]
and with initial point $(Y_0,\Ybar_0)$ given, define for $n\ge 1$ 
\[
Y_n=Y_{n-1}+ Z^{Y_{n-1}}_{L_{n-1}(Y_{n-1})}
\quad\text{and}\quad
\Ybar_n=\Ybar_{n-1}+ \Zbar^{\Ybar_{n-1}}_{\Lbar_{n-1}(\Ybar_{n-1})}.
\]

In words, every time the chain $Y$ visits a state $x$, it 
reads its next jump from a new variable $Z^x_k$ which is then
discarded and never used again.  And similarly for $\Ybar$.   
This construction has the property that, if $Y_k=\Ybar_k$ for 
$0\le k\le n$ with $Y_{n}=\Ybar_{n}=x$, 
then the next joint step  is 
$(Z^x_k,\Zbar^x_k)$ for $k=L_{n}(x)=\Lbar_{n}(x)$.  In other
words, given that the processes agree up to the present 
and reside together at $x$, the probability that they separate in
the next step is bounded by $Ce^{-\alpha_1\abs{x}}$.

Now follow self-evident steps. 
\begin{align*}
&P_{x,x}[\text{ $Y_k\ne\Ybar_k$ for some $1\le k\le \zeta$ }]\\
&\le \sum_{k=1}^\infty
 P_{x,x}[\text{ $Y_j=\Ybar_j\in A^c$ for $1\le j<k$,   $Y_k\ne\Ybar_k$ } ]\\
&\le \sum_{k=1}^\infty
 E_{x,x}\bigl[\one\{\text{ $Y_j=\Ybar_j\in A^c$ for $1\le j<k$ }\}
P_{Y_{k-1},\Ybar_{k-1}}( Y_1\ne\Ybar_1) \,  \bigr]\\
&\le C\sum_{k=1}^\infty
 E_{x,x}\bigl[\one\{\text{ $Y_j=\Ybar_j\in A^c$ for $1\le j<k$ }\}
e^{-\alpha_1\lvert \Ybar_{k-1}\rvert} \,   \bigr]\\
&\le C E_x \sum_{m=0}^{\zeta-1} e^{-\alpha_1\lvert\Ybar_m\rvert}.
\qedhere
\end{align*}
\end{proof}

For the remainder of this section
 $Y$ and $\Ybar$ are always coupled
in the manner that satisfies Lemma \ref{YYbarlm1}.

\begin{lemma} Let $j\in\{1,\dotsc,d\}$ be such that the one-dimensional
random walk $\Ybar^j$ is not degenerate.  
Let $r_0$ be a positive integer and  $\wbar=\inf\{n\ge 1: \Ybar_n^j\le r_0\}$ 
 the first time the random walk $\Ybar$
 enters the half-space 
$\Hc=\{x: x^j\le r_0\}$. Couple $Y$ and $\Ybar$ starting from a
common initial state $x\notin\Hc$.
Then there is a constant $C$ independent of $r_0$ such that
\[
\sup_{x\notin\Hc } P_{x,x}[\text{ $Y_k\ne\Ybar_k$ for some
 $k\in\{1,\dotsc,\wbar\}$ }] \le  Ce^{-\alpha_1r_0}
\quad \text{ for all $r_0\ge 1$.} 
\]
The same result holds for $\Hc=\{x: x^j\ge -r_0\}$.
\label{YY-aux-lm-1}
\end{lemma}

\begin{proof} By Lemma \ref{YYbarlm1} 
\begin{align*}
&P_{x,x}[\text{ $Y_k\ne\Ybar_k$ for some
 $k\in\{1,\dotsc,\wbar\}$ }]
\le  CE_x\biggl[\; \sum_{k=0}^{\wbar-1} 
e^{-\alpha_1 \lvert \Ybar_k\rvert}\,\biggr]\\
&\qquad\qquad\qquad 
\le CE_{x^j}\biggl[\; \sum_{k=0}^{\wbar-1} 
e^{-\alpha_1  \Ybar_k^j}\,\biggr]
=C\sum_{t=r_0+1}^\infty e^{-\alpha_1 t} g(x^j,t) 
\end{align*}
where  for $s,t\in [r_0+1,\infty)$
\[
g(s,t)= \sum_{n=0}^\infty P_s[\Ybar_n^j=t\,,\,\wbar>n] 
\]
is the Green function of the half-line $(-\infty, r_0]$ for 
the one-dimensional random walk $\Ybar^j$.
This is the expected
number of visits to $t$ before entering $(-\infty,r_0]$,
defined on p.~209 in Spitzer \cite{spitzer}.  The development
in Sections 18 and 19 in \cite{spitzer} gives the bound 
\be
g(s,t)\le C(1+(s-r_0-1)\wedge(t-r_0-1))\le C(t-r_0),\quad 
s,t\in [r_0+1,\infty). 
\label{spitz-1}\ee

Here is some more detail. Shift $r_0+1$ to the origin to 
match the setting in \cite{spitzer}.  Then    P19.3 on p.~209
gives
\[
g(x,y)= \sum_{n=0}^{x\wedge y} u(x-n)v(y-n)\qquad \text{for $x,y>0$}
\]
where the functions $u$ and $v$ are defined on p.~201. 
For a symmetric random walk $u=v$. 
P18.7 on p.~202 implies that
\[
v(m)=\frac1{\sqrt{c}} \sum_{k=0}^\infty 
\mathbf{P}[\mathbf{Z}_1+\dotsm+\mathbf{Z}_k=m]
\]
where $c$ is a certain constant
and  $\{\mathbf{Z}_i\}$ are i.i.d.~strictly positive, integer-valued
ladder variables for the underlying random walk.  
Now one can show inductively that $v(m)\le v(0)$ for each $m$
so the quantities $u(m)=v(m)$ are bounded. This 
justifies \eqref{spitz-1}. 

Continuing from further above we get the estimate claimed in the 
statement: 
\[
E_x\biggl[\; \sum_{k=0}^{\wbar-1} 
e^{-\alpha_1 \lvert \Ybar_k\rvert}\,\biggr]
\le C\sum_{t>r_0} (t-r_0) e^{-\alpha_1 t}
\le C e^{-\alpha_1r_0}. 
\qedhere
\]
\end{proof} 

For the next lemmas abbreviate $B_r=[-r,r]^d$ for $d$-dimensional
centered cubes. 

\begin{lemma} 
With $\alpha_1$ given in the coupling hypothesis {\rm (A.iii)},
fix any positive constant $\kappa_1> 2\alpha_1^{-1}$. 
Consider large positive integers $r_0$ and $r$ that  satisfy 
 \[2\alpha_1^{-1}\log r\le  r_0\le \kappa_1\log r< r.\]  Then there exist
a positive integer $m_0$ and  a constant 
$0<\alpha_3 <\infty$ such that, for large enough $r$,  
\be
\inf_{x\in B_r\smallsetminus B_{r_0}} 
P_x[ \text{without entering $B_{r_0}$ chain $Y$ exits $B_r$ by time
 $r^{m_0}$}]
\ge \frac{\alpha_3}r.
\label{YY-aux-1.5}
\ee
\label{YY-aux-lm-2}\end{lemma}

\begin{proof} 
We consider first the case where 
 $x\in B_r\smallsetminus B_{r_0}$ has a coordinate $x^j$
that satisfies $x^j\in[-r, -r_0-1]\cup[r_0+1,r]$ and 
$\Ybar^j$ is nondegenerate.  For this case we can take $m_0=4$.
A higher $m_0$ may be needed to move a suitable coordinate 
out of the interval $[-r_0,r_0]$.  This is done in the second
step of the proof. 

The same argument works for
  both  $x^j\in[-r, -r_0-1]$
and $x^j\in[r_0+1,r]$.  We  treat the case
$x^j\in[r_0+1,r]$.  
One way to realize the event in \eqref{YY-aux-1.5}  is this:  starting at 
$x^j$,
the $\Ybar^j$ walk exits  $[r_0+1,r]$ by time $r^4$ through the right 
boundary into $[r+1,\infty)$, 
and $Y$ and $\Ybar$ stay coupled together throughout this time. 
Let $\zetabar$ be the time $\Ybar^j$ exits $[r_0+1,r]$ and $\wbar$ the 
time $\Ybar^j$ enters $(-\infty, r_0]$.  Then $\wbar\ge\zetabar$. 
 Thus the 
complementary probability of \eqref{YY-aux-1.5} is bounded by 
\be
\begin{split}
&P_{x^j}\{\text{ $\Ybar^j$  exits  $[r_0+1,r]$ into $(-\infty, r_0]$ }\}\\
&\qquad + \; P_{x^j}\{ \zetabar >r^4\} \; +  \;  
P_{x,x}\{\text{ $Y_k\ne\Ybar_k$ for some
 $k\in\{1,\dotsc,\wbar\}$ }\}.
\end{split} 
\label{YY-aux-3}\ee

We treat the terms one at a time. 
From the development on p.~253-255  in \cite{spitzer} we get the 
bound
\be
P_{x^j}\{\text{ $\Ybar^j$  exits  $[r_0+1,r]$ into $(-\infty, r_0]$ }\}
 \le 1-\frac{\alpha_4}r
\label{YY-aux-4}
\ee
for some constant $\alpha_4>0$. In some more detail:
P22.7 on p.~253, 
the inequality in the third display  of p.~255, and the third moment
assumption on the steps of $\Ybar$  give a lower bound 
\be
P_{x^j}\{\text{ $\Ybar^j$  exits  $[r_0+1,r]$ into $[r+1,\infty)$ }\}
\ge \frac{x^j-r_0-1-c_1}{r-r_0-1}
\label{spitz-8}
\ee
for the probability of exiting to the right.
Here $c_1$ is a constant that comes from the term denoted
in \cite{spitzer}  by
$  M\sum_{s=0}^N(1+s)a(s) $
whose finiteness follows from the third moment assumption.  
The text  on p.~254-255 suggests that these
steps need the aperiodicity assumption. This need for
aperiodicity  can be traced back via P22.5 to 
P22.4 which is used to assert the boundedness of 
$u(x)$ and $v(x)$.  But as we observed above in the derivation
of \eqref{spitz-1} boundedness of $u(x)$ and $v(x)$   is true
without any additional assumptions.  
 
 To go forward from \eqref{spitz-8} 
fix any $m>c_1$ so that the numerator
above is positive for $x^j=r_0+1+m$. 
The probability in \eqref{spitz-8}
 is minimized at $x^j=r_0+1$, and from $x^j=r_0+1$
there is a fixed positive
probability $\theta$ to take $m$ steps to the right to get
past the point $x^j=r_0+1+m$.  Thus for all $x^j\in[r_0+1,r]$ we get the lower bound 
\[
P_{x^j}\{\text{ $\Ybar^j$  exits  $[r_0+1,r]$ into $[r+1,\infty)$ }\}
\ge \frac{\theta(m-c_1)}{r-r_0-1} \ge \frac{\alpha_4}{r}
\] 
and \eqref{YY-aux-4} is verified. 

As in  \eqref{spitz-1} let $g(s,t)$ be the Green function
of the random walk $\Ybar^j$  for the 
half-line $(-\infty, r_0]$,  and 
let $\tilde{g}(s,t)$ be the Green function for the complement
of the interval 
$[r_0+1,r]$. Then $\tilde{g}(s,t)\le g(s,t)$, and by \eqref{spitz-1}
we get this moment bound:
\begin{align*}
 E_{x^j}[\,\zetabar\,]=\sum_{t=r_0+1}^r \tilde{g}(x^j,t) 
\le \sum_{t=r_0+1}^r g(x^j,t) 
\le Cr^2.
\end{align*}
Consequently, 
 uniformly over $x^j\in [r_0+1,r]$, 
\be
 P_{x^j}[ \zetabar >r^4] \le \frac{C}{r^2}.
\label{YY-aux-5} \ee

From Lemma \ref{YY-aux-lm-1}  
\be
P_x[\text{ $Y_k\ne\Ybar_k$ for some
 $k\in\{1,\dotsc,\wbar\}$ }] \le  Ce^{-\alpha_1r_0}.
\label{YY-aux-6}\ee

Putting bounds \eqref{YY-aux-4}, \eqref{YY-aux-5}
and \eqref{YY-aux-6} together gives an upper bound of 
\[
1\;-\;\frac{\alpha_4}r \;+\; \frac{C}{r^2} \;+\;   Ce^{-\alpha_1r_0}
\]
for the sum in \eqref{YY-aux-3} which bounds the complement of the 
probability in \eqref{YY-aux-1.5}.  By assumption 
$r_0>2\alpha_1^{-1}\log r$, so for large enough
$r$  the sum above is not more than $1-\alpha_3/r$ for some
constant $\alpha_3>0$.  

\medskip 

The lemma is now proved for those   $x\in B_r\smallsetminus B_{r_0}$ 
for which  some 
\[ j\in J\equiv \{1\le j\le d: \text{ the one-dimensional 
walk $\Ybar^j$ is nondegenerate}\}\]
satisfies  $x^j\in[-r, -r_0-1]\cup[r_0+1,r]$.
Now suppose  $x\in B_r\smallsetminus B_{r_0}$  but all $j\in J$ 
satisfy $x^j\in[-r_0,r_0]$. 
Let 
\[
T=\inf\{n\ge 1: \text{$Y^j_n\notin[-r_0,r_0]$ for some $j\in J$}\}.
\]
The first part of the proof gives $P_x$-almost surely 
\[
P_{Y_T}[ 
\text{without entering $B_{r_0}$ chain $Y$ exits $B_r$ by time
 $r^{4}/2$}]
\ge \frac{\alpha_3}r.
\]
Replacing $r^4$ by $r^4/2$ only affects the constant in \eqref{YY-aux-5}. 
It can of course happen that $Y_T\notin B_r$ but then we interpret the above
probability as one. 

By the Markov property it remains to show that  for a suitable $m_0$ 
\be
\inf\bigl\{\; P_x[ T\le r^{m_0}/2]:  \text{$x\in B_r\smallsetminus B_{r_0}$  but
$x^j\in[-r_0,r_0]$ for  all $j\in J$}
\,\bigr\} 
\label{YY-aux-10}
\ee
is bounded below by a positive constant.  Hypothesis \eqref{Y-ell-ass} implies
that for some constant $b_1$,
 $E_xT\le b_1^{r_0}$ uniformly over the relevant  $x$.  This is because 
one way to realize $T$ is to wait until some coordinate $Y^j$ 
takes $2r_0$ successive identical steps. By hypothesis
  \eqref{Y-ell-ass} this random time is stochastically bounded by a 
 geometrically distributed
random variable. 
 
It is also necessary for this argument
 that during time $[0, T]$ the chain $Y$
does not enter $B_{r_0}$.   Indeed, under the present assumptions 
 the chain never enters $B_{r_0}$.  This is because for
 $x\in B_r\smallsetminus B_{r_0}$  some 
coordinate $i$ must satisfy  $x^i\in[-r, -r_0-1]\cup[r_0+1,r]$.
But now this coordinate $i\notin J$, and so by hypothesis
(A.ii) the one-dimensional process $Y^i$ is constant, 
$Y^i_n=x^i\notin[-r_0,r_0]$ for all $n$. 
  
Finally, the required positive lower  bound for   \eqref{YY-aux-10}
comes by Chebychev.  Take $m_0\ge \kappa_1\log b_1 + 1$  where $\kappa_1$
comes from the assumptions of the lemma.  Then, by the hypothesis
$r_0\le \kappa_1\log r$,  
\[
P_x[ T> r^{m_0}/2]\le 2r^{-m_0} b_1^{r_0}\le 2r^{\kappa_1\log b_1-m_0}\le \tfrac12
\]
for $r\ge 4$.  
\end{proof}

We come to one of the main auxiliary lemmas of this development. 

\begin{lemma}  Let $U=\inf\{n\ge 0: Y_n\notin B_r\}$  be 
the first exit time from $B_r$ for the Markov chain $Y$. 
 Then there exist finite positive constants $C_1, m_1$ such that 
\[\sup_{x\in B_r}E_x(U)\le C_1r^{m_1}\quad\text{ for all $1\le r<\infty$.}\] 
\label{YY-aux-lm-6}
\end{lemma} 

\begin{proof} First observe that $\sup_{x\in B_r} E_x(U)<\infty$
by assumption \eqref{Y-ell-ass} because by a geometric
time some coordinate  $Y^j$ has experienced 
$2r$ identical steps in succession.
Throughout, let $r_0<r$ satisfy the assumptions of Lemma
\ref{YY-aux-lm-2}.  Once the statement is proved for large enough
$r$, we obtain it for all $r\ge 1$ by increasing $C_1$. 

Let $0=T_0=S_0\le T_1\le S_1\le T_2\le\dotsm$ be the successive
exit and  entrance
times into $B_{r_0}$. Precisely,
 for $i\ge 1$ as long as $S_{i-1}<\infty$ 
\[
T_i=\inf\{ n\ge S_{i-1}: Y_n\notin B_{r_0}\}
\quad\text{and}\quad
S_i=\inf\{ n\ge T_{i}: Y_n\in B_{r_0}\}
\]
Once $S_i=\infty$ then we set $T_j=S_j=\infty$ for all $j>i$. 
If $Y_0\in B_r\smallsetminus B_{r_0}$ then also $T_1=0$. 
Again by assumption \eqref{Y-ell-ass} (and as observed
in the proof of Lemma \ref{YY-aux-lm-2})
 there is a constant 
$0<b_1<\infty$ 
such that  
\be \sup_{x\in B_{r_0}}E_x[T_1]\le b_1^{r_0}.  \label{ell-escape}\ee
So a priori $T_1$ is finite but $S_1=\infty$ is possible. 
Since $T_1\le U<\infty$ we can decompose as follows: 
\be
\begin{split}
E_x[U]&= \sum_{j=1}^\infty E_x[U,\, T_j\le U<S_j]\\
 &= \sum_{j=1}^\infty E_x[T_j\,,\, T_j\le U<S_j] + 
\sum_{j=1}^\infty E_x[U-T_j\,,\, T_j\le U<S_j].
\end{split}
\label{YY-aux-11}\ee

We first treat the last sum in \eqref{YY-aux-11}. By an inductive
application of  Lemma \ref{YY-aux-lm-2}, 
for any $z\in B_r\smallsetminus B_{r_0}$,
\begin{align*}
&P_z[U>jr^{m_0},\, U<S_1] \le P_z[\text{ $Y_k\in B_r\smallsetminus B_{r_0}$
for $k\le jr^{m_0}$ }] \\
&=E_z\bigl[\one\{\text{ $Y_k\in B_r\smallsetminus B_{r_0}$          
for $k\le (j-1)r^{m_0}$ }\} P_{Y_{(j-1)r^{m_0}}}\{ \text{ 
$Y_k\in B_r\smallsetminus B_{r_0}$  for $k\le r^{m_0}$ }\}\,\bigr]\\
&\le \dotsm \le (1-\alpha_3r^{-1})^j.
\end{align*} 
Utilizing this, still for $z\in B_r\smallsetminus B_{r_0}$,
\be\begin{split}
E_z[ U,\,U<S_1]&=\sum_{m=0}^\infty P_z[U>m\,,\,U<S_1]\\
&\le r^{m_0} \sum_{j=0}^\infty P_z[ U>jr^{m_0}\,,\,U<S_1] \le r^{m_0+1}\alpha_3^{-1}.
\end{split}\label{YY-aux-12}\ee 
Next we take into consideration the failure to exit $B_r$ 
during the earlier excursions in $B_r\smallsetminus B_{r_0}$.
Let 
\[ H_i=\{ \text{$Y_n\in B_r$ for $T_i\le n<S_i$} \}  \]
be the event that in between the $i$th exit from $B_{r_0}$ and 
entrance back into $B_{r_0}$ the chain $Y$ does not exit $B_r$. 
We shall repeatedly use this  consequence of Lemma \ref{YY-aux-lm-2}:
 \be \text{  for $i\ge 1$, on the event $\{T_i<\infty\}$,   
 $P_x[ H_i\,\vert\,\Fc_{T_i}]\le 1-\alpha_3r^{-1}$.  
}\label{YY-aux-14}\ee
  Here is the 
first instance. 

\begin{align*}
&E_x[U-T_j\,,\, T_j\le U<S_j]
=
E_x\Bigl[\; \prod_{k=1}^{j-1}\one_{H_k} \cdot\one\{T_j<\infty\}
\cdot E_{Y_{T_j}}(U,\,U<S_1)\Bigr]\\
&\le  r^{m_0+1}\alpha_3^{-1} E_x\Bigl[ \;\prod_{k=1}^{j-1}\one_{H_k}
\cdot\one\{T_{j-1}<\infty\} \Bigr]
\le  r^{m_0+1}\alpha_3^{-1} (1-\alpha_3r^{-1})^{j-1}.
\end{align*}
Note that if $Y_{T_j}$ above lies outside $B_r$ then 
$E_{Y_{T_j}}(U)=0$.  In the other case $Y_{T_j}\in 
 B_r\smallsetminus B_{r_0}$ 
and \eqref{YY-aux-12} applies. 
So for  the last sum in \eqref{YY-aux-11}:
\be
\sum_{j=1}^\infty E_x[U-T_j\,,\, T_j\le U<S_j]
\le \sum_{j=1}^\infty  r^{m_0+1}\alpha_3^{-1} (1-\alpha_3r^{-1})^{j-1}
\le r^{m_0+2}\alpha_3^{-2}.
\label{YY-aux-17}
\ee

We turn to the second-last sum in \eqref{YY-aux-11}. 
Utilizing \eqref{ell-escape} and \eqref{YY-aux-14}, 
\be
\begin{split}
&E_x[T_j\,,\, T_j\le U<S_j] \ \le\  
\sum_{i=0}^{j-1} E_x\Bigl[\; \prod_{k=1}^{j-1}\one_{H_k} \cdot\one\{T_j<\infty\}\cdot
 (T_{i+1}-T_i)\Bigr] \\
&\le \  b_1^{r_0} (1-\alpha_3r^{-1})^{j-1}  \\
&\quad   +\; 
 \sum_{i=1}^{j-1} E_x\Bigl[\; \prod_{k=1}^{i-1}\one_{H_k}\cdot
(T_{i+1}-T_i)\one_{H_i} \cdot\one\{T_{i+1}<\infty\} \Bigr] (1-\alpha_3r^{-1})^{j-1-i} .
\end{split}
\label{YY-aux-18}\ee
Split the last expectation  as
\begin{align}
&E_x\Bigl[\; \prod_{k=1}^{i-1}\one_{H_k}\cdot              
(T_{i+1}-T_i)\one_{H_i}\cdot\one\{T_{i+1}<\infty\} \Bigr]\nn\\
&\le
E_x\Bigl[\; \prod_{k=1}^{i-1}\one_{H_k}\cdot              
(T_{i+1}-S_i)\one_{H_i} \cdot\one\{S_{i}<\infty\}\Bigr]\nn\\
&\qquad\quad  +\; E_x\Bigl[\; \prod_{k=1}^{i-1}\one_{H_k}\cdot              
(S_i-T_i)\one_{H_i}\cdot\one\{T_{i}<\infty\} \Bigr]\nn\\
&\le 
E_x\Bigl[\; \prod_{k=1}^{i-1}\one_{H_k}\cdot\one\{S_{i}<\infty\}\cdot              
E_{Y_{S_i}}(T_1) \Bigr]
+ E_x\Bigl[\; \prod_{k=1}^{i-1}\one_{H_k}\cdot \one\{T_{i}<\infty\}\cdot             
E_{Y_{T_i}}(S_1\cdot\one_{H_1}) \Bigr]\nn\\
&\le E_x\Bigl[\; \prod_{k=1}^{i-1}\one_{H_k}\cdot \one\{T_{i-1}<\infty\}\Bigr]
(b_1^{r_0} + r^{m_0+1}\alpha_3^{-1})\nn\\
&\le (1-\alpha_3r^{-1})^{i-1} (b_1^{r_0} + r^{m_0+1}\alpha_3^{-1}).              
\label{YY-aux-20}
\end{align}
In the second-last inequality above,
 before applying \eqref{YY-aux-14} 
to the $H_k$'s,  
$E_{Y_{S_i}}(T_1)\le b_1^{r_0}$ comes
from \eqref{ell-escape}. The other expectation is 
estimated again by iterating Lemma \ref{YY-aux-lm-2} and 
 again with $z\in B_r\smallsetminus B_{r_0}$:
\begin{align*}
E_{z}(S_1\cdot\one_{H_1})
&=\sum_{m=0}^\infty P_z[S_1>m\,,\,H_1]
\le\sum_{m=0}^\infty P_z[\text{ $Y_k\in B_r\smallsetminus B_{r_0}$
for $k\le m$ }]\\
&\le r^{m_0}\sum_{j=0}^\infty P_z[\text{ $Y_k\in B_r\smallsetminus B_{r_0}$
for $k\le jr^{m_0}$ }]
\le r^{m_0+1}\alpha_3^{-1}. 
\end{align*} 
Insert the bound from line \eqref{YY-aux-20} back up into
\eqref{YY-aux-18} to get the bound 
\begin{align*}
E_x[T_j\,,\, T_j\le U<S_j] \le 
 (2 b_1^{r_0} + r^{m_0+1}\alpha_3^{-1})j   (1-\alpha_3r^{-1})^{j-2}.
\end{align*}
Finally, bound  the second-last sum in \eqref{YY-aux-11}:
\begin{align*}
\sum_{j=1}^\infty E_x[T_j\,,\, T_j\le U<S_j] 
\le \bigl(2 b_1^{r_0}r^2\alpha_3^{-2} + r^{m_0+3}\alpha_3^{-3}\bigr)
 (1-\alpha_3r^{-1})^{-1}.
\end{align*}
Taking $r$ large enough so that $\alpha_3r^{-1}<1/2$ and 
combining this with \eqref{YY-aux-11} and \eqref{YY-aux-17} gives
\[
E_x[U]\le r^{m_0+2}\alpha_1^{-2} + 
4 b_1^{r_0}r^2\alpha_3^{-2} + 2r^{m_0+3}\alpha_3^{-3}.
\]
Since $r_0\le \kappa_1\log r$ for some constant $C$,
the above bound simplifies to $C_1r^{m_1}$.  
\end{proof} 

For the remainder of the proof we work with 
$B=B_r$ for $r=c_1\log n$. 
The above estimate gives us one part of the argument for 
\eqref{goal-Y-1}, namely that the Markov chain $Y$ exits
$B=[-c_1\log n, c_1\log n]^d$ fast enough.

Let $0=V_0<U_1<V_1<U_2<V_2<\dotsm$ be the successive entrance
times $V_i$ into $B$  and exit times $U_i$ from $B$ for the 
Markov chain $Y$, assuming that $Y_0=z\in B$. It is possible that 
some $V_i=\infty$.  But if $V_i<\infty$ then also $U_{i+1}<\infty$
due to assumption \eqref{Y-ell-ass},  as already observed.
The time intervals spent in $B$ are $[V_i, U_{i+1})$ each of length
at least 1. Thus, by applying Lemma \ref{YY-aux-lm-6}, 
\be\begin{split}
\sum_{k=0}^{n-1} P_z(Y_k\in B) &\le 
\sum_{i=0}^n E_z\bigl[\, (U_{i+1}-V_i) \one\{V_i\le n\}\bigr]\\
&\le \sum_{i=0}^n E_z\bigl[\, E_{Y_{V_i}} (U_1) \one\{V_i\le n\}\bigr]\\
&\le C(\log n)^{m_1} E_z\biggl[\, \sum_{i=0}^n   \one\{V_i\le n\}\biggr].
\end{split} 
\label{temp-Y-2}
\ee

Next we bound the expected number of returns to $B$ by the number of
 excursions outside $B$ that fit in a time of length $n$:
\begin{align} 
E_z\biggl[\, \sum_{i=0}^n   \one\{V_i\le n\}\biggr]
&=
E_z\biggl[\, \sum_{i=0}^n   \one\Bigl\{\,
\sum_{j=1}^i (V_j-V_{j-1})\le n\Bigr\}\biggr]\nn\\
&\le E_z\biggl[\, \sum_{i=0}^n   \one\Bigl\{\,
\sum_{j=1}^i (V_j-U_{j})\le n\Bigr\}\biggr]
\label{line-a7}
\end{align}

According to the usual notion of stochastic dominance, 
 the random vector $(\xi_1,\dotsc,\xi_n)$
dominates $(\eta_1,\dotsc,\eta_n)$ if 
\[ Ef(\xi_1,\dotsc,\xi_n)\ge Ef(\eta_1,\dotsc,\eta_n) \]
for any function $f$ that is 
coordinatewise nondecreasing.  If the 
 $\{\xi_i:1\le i\le n\}$ are adapted to the filtration
$\{\Gc_i:1\le i\le n\}$, and 
$P[\xi_i> a\vert\Gc_{i-1}]\ge 1-F(a)$ for some
distribution function $F$,  then the 
 $\{\eta_i\}$ can be taken i.i.d.\ $F$-distributed.  

\begin{lemma} There exist positive constants $c_1$, $c_2$
and $\gamma$ such that the following holds:   
the excursion lengths $\{V_j-U_j:1\le j\le n\}$ 
stochastically dominate i.i.d.\ variables $\{\eta_j\}$ whose
common distribution satisfies 
$\Pv[\eta\ge a]\ge c_1a^{-1/2}$ for $1\le a\le c_2n^\gamma$. 
\label{stoch-lm} 
\end{lemma} 

\begin{proof} 
Since $P_z[V_j-U_j\ge a\vert \Fc_{U_j}]=P_{Y_{U_j}}[V\ge a]$
where $V$ means first entrance time into $B$, we shall bound
$P_x[V\ge a]$ below uniformly over 
\[
\Bigl\{ x\notin B:  \sum_{z\in B} P_z[Y_{U_1}=x] >0\,\Bigr\}. 
\]
Fix such an $x$ and an index $1\le j\le d$ such that $x^j\notin[-r,r]$. 
Since the coordinate $Y^j$ can move out of $[-r,r]$, this
coordinate is not degenerate, and hence by assumption 
(A.ii) the random walk $\Ybar^j$ is nondegenerate.  
  As before we work through 
the case $x^j>r$ because the argument for the other case
$x^j<-r$  is the same.

Let $\wbar=\inf\{n\ge 1: \Ybar^j_n\le r\}$ 
be the first time the one-dimensional random walk $\Ybar^j$
 enters the half-line 
$(-\infty, r]$.   If both $Y$ and $\Ybar$ start at $x$ and 
stay coupled together  until time $\wbar$, then $V\ge\wbar$.  This 
 way we bound $V$ from below.   Since the random walk
is symmetric and 
can be translated, we can move the origin to $x^j$ and use 
classic results about the first entrance  time
into the left half-line,  $\Tbar=\inf\{ n\ge 1: \Ybar^j_n<0\}$. 
Thus 
\[
P_{x^j}[\wbar\ge a]\ge P_{r+1}[\wbar\ge a]= P_0[\Tbar\ge a]
\ge \frac{\alpha_5}{\sqrt{a}}
\]
for a constant $\alpha_5$. 
The last inequality follows for  one-dimensional symmetric walks
 from basic random walk theory. For example, combine 
equation (7) on p.~185 of \cite{spitzer} with a Tauberian theorem 
such as Theorem 5 on p.~447 of Feller \cite{fellerII}.   Or see directly 
Theorem 1a on p.~415 of \cite{fellerII}.  

Now start both $Y$ and $\Ybar$ from $x$.
 Apply Lemma 
\ref{YY-aux-lm-1} and recall that  $r=c_1\log n$.   
\begin{align*}
P_x[V\ge a]&\ge P_{x,x}[V\ge a, \text{ $Y_k=\Ybar_k$ for $k=1,\dotsc,\wbar$ }]\\
&\ge P_{x,x}[\wbar\ge a, \text{ $Y_k=\Ybar_k$ for $k=1,\dotsc,\wbar$ }]\\
&\ge P_{x^j}[\wbar\ge a] - P_{x,x}[\text{ $Y_k\ne\Ybar_k$ for some
 $k\in\{1,\dotsc,\wbar\}$ }]\\
&\ge \frac{\alpha_5}{\sqrt{a}} - Cn^{-c_1\alpha_1}. 
\end{align*}
This gives a lower bound 
\[
P_x[V\ge a] \ge \frac{\alpha_5}{2\sqrt{a}} 
\]
if $a\le \alpha_5^2(2C)^{-2}n^{2c_1\alpha_1}$. 
This lower bound is independent 
of $x$.  We have proved the lemma. 
\end{proof}

We can assume that the random variables $\eta_j$ given by
the lemma satisfy $1\le \eta_j\le c_2n^\gamma$ and we can 
assume both $c_2,\gamma\le 1$ because this merely weakens the result.  
 For the renewal process
determined by $\{\eta_j\}$ write  
\[
S_0=0\, ,\; S_k=\sum_{j=1}^k \eta_j\,,
\quad\text{and}\quad  K(n)=\inf\{ k: S_k > n\}
\]
for the renewal times 
and the number of renewals up to time $n$ (counting the 
renewal $S_0=0$).  
Since the random variables are
bounded, Wald's identity gives 
\[  \Ev K(n) \cdot \Ev\eta = 
\Ev S_{K(n)} \le n+c_2n^\gamma \le 2n,
\]
while 
\[
\Ev\eta \ge \int_1^{c_2n^\gamma}  \frac{c_1}{\sqrt{s}}\,ds \ge c_3 
n^{\gamma/2}.
\]
Together these give 
\[ \Ev K(n) \le \frac{2n}{\Ev\eta} \le C_2n^{1-\gamma/2}. 
\]

Now we pick up the development from line \eqref{line-a7}. 
Since the negative of the function of $(V_j-U_j)_{1\le i\le n}$ 
in the expectation on line \eqref{line-a7} is nondecreasing, 
the stochastic domination of Lemma \ref{stoch-lm} gives
 an upper  bound of  \eqref{line-a7} in terms of the i.i.d.\ $\{\eta_j\}$.  
Then 
we use the renewal bound from  above.
\begin{align*} 
E_z\biggl[\, \sum_{i=0}^n   \one\{V_i\le n\}\biggr]
&\le E_z\biggl[\, \sum_{i=0}^n   \one\Bigl\{\,
\sum_{j=1}^i (V_j-U_{j})\le n\Bigr\}\biggr]\\
&\le \Ev\biggl[\, \sum_{i=0}^n   \one\Bigl\{\,
\sum_{j=1}^i \eta_j\le n\Bigr\}\biggr]
=  \Ev K(n) \le C_2n^{1-\gamma/2}.  
\end{align*}
Returning back to \eqref{temp-Y-2} to collect the bounds, we
have shown that 
\[
\sum_{k=0}^{n-1} P_z(Y_k\in B) \le 
C(\log n)^{m_1} E_z\biggl[\, \sum_{i=0}^n   \one\{V_i\le n\}\biggr] 
\le C(\log n)^{m_1} C_2n^{1-\gamma/2}
\]
and thereby verified \eqref{goal-Y-1}.

%\input{intersect}

%%%%%%%%%
%%% If you want an appendix or acknowledgments:
%%%%%%%%%

%\begin{appendix}
%\appdx{Title}
%\bfsection{Title}
% or
% you can also \input{appendix}
%\end{appendix}

%%%%%%%%%
%%% Acknowledgments
%%%%%%%%%

%\begin{acknowledgment}
%Thank you
%\end{acknowledgment}

%or

%\begin{acknowledgments}
%Thank you
%\end{acknowledgments}

%%%%%%%%%
%% Bibliography:
%%%%%%%%%

%% This uses BibTex:
%% Make sure you put the refs file in the same directory

%% For the different styles, see 
%% http://www.google.com/search?q=bibtex%20style
%% or http://www.cs.stir.ac.uk/~kjt/software/latex/showbst.html

%\nocite{paper1}  % this allows to include reference "paper1" from the 
                  % BibTex file, even if it's not cited.

\bibliographystyle{abbrv} % A regular style that is supported 
%\bibliographystyle{aop} % An AOP-like style found and edited by Firas
                        % (the tabbing is not right though)
\bibliography{nnrefs}
%\bibliography{refsfiras,refstimo,refsmarton}

%% If you want to use thebibliography look below:

%% Uncomment this if you want to use Bibliography instead of References
%\renewcommand{\refname}{Bibliography}

%% Uncomment this, if you wanted the references in small font.
%% Don't forget to uncomment the end part!
%\begin{footnotesize} %or \begin{small}

%\begin{thebibliography}{99}
%Separate the author, title, journal, issue, pages, and year.
%It will make it easier later to put them in different formats.
%Example:
%
%\bibitem{ferrarifontes}
%P.A.~Ferrari and L.R.G.~Fontes:
%Title.....
%Electron. J. Probab.
%\textbf{3},
%1--34
%(1998)
%

%\end{thebibliography}

%\end{footnotesize} %or \end{small}

\end{document}